\documentclass[12pt,twoside]{amsart}
\usepackage{amssymb}
\usepackage{amscd}
\usepackage[abbrev,alphabetic]{amsrefs}
\usepackage{hyperref}
\usepackage{comment}
\usepackage{array,multirow,tabularx,longtable}
\usepackage{tikz}
\usetikzlibrary{cd}
\usepackage{here}
\usepackage{multirow}
\usepackage[margin=1.25in]{geometry}

\usepackage{framed}
\usepackage{fancybox}
\usepackage{ascmac}

\title[Prime Fano threefolds of genus 8 in positive characteristic]
{Prime Fano threefolds of genus 8 in positive characteristic} 

\author{Akihiro Kanemitsu}
\address{Department of Mathematical Sciences, Graduate School of Science, Tokyo Metropolitan University, 1-1 Minami-Osawa, Hachioji-shi, Tokyo 192-0397, Japan}
\email{kanemitsu@tmu.ac.jp}
\thanks{The first author is supported by JSPS KAKENHI Grant Number JP23K12948. The second author is supported by JSPS KAKENHI Grant Numbers JP22H01112 and JP23K03028. 
}

\author{Hiromu Tanaka} 
\address{Department of Mathematics, 
Graduate School of Science, 
Kyoto University, 
Kyoto 606-8502, JAPAN} 
\email{tanaka.hiromu.7z@kyoto-u.ac.jp}

\newcommand{\CH}[0]{{\operatorname{CH}}}
\newcommand{\ol}{\overline}
\newcommand{\wt}{\widetilde}

\newcommand{\pr}[0]{{\operatorname{pr}}}

\newcommand{\Bl}[0]{{\operatorname{Bl}}}

\newcommand{\rank}[0]{{\operatorname{rank}}}

\newcommand{\NE}[0]{{\operatorname{NE}}}
\newcommand{\red}[0]{{\operatorname{red}}}
\newcommand{\Coker}[0]{\operatorname{Coker}}
\newcommand{\Ker}[0]{\operatorname{Ker}}

\renewcommand{\Im}[0]{\operatorname{Im}}

\newcommand{\Proj}[0]{\operatorname{Proj}}
\newcommand{\Spec}[0]{\operatorname{Spec}}

\newcommand{\Hom}[0]{{\operatorname{Hom}}}

\newcommand{\Supp}[0]{\operatorname{Supp}}
\newcommand{\Pic}[0]{\operatorname{Pic}}

\newcommand{\Ex}[0]{{\operatorname{Ex}}}

\newcommand{\univ}[0]{{\operatorname{univ}}}

\newcommand{\Frac}[0]{{\operatorname{Frac}}}
\newcommand{\Aut}[0]{{\operatorname{Aut}}}

\newcommand{\Sing}[0]{\operatorname{Sing}}

\newcommand{\Gr}[0]{{\operatorname{Gr}}}

\newcommand{\PGL}[0]{{\operatorname{PGL}}}
\newcommand{\GL}[0]{{\operatorname{GL}}}
\newcommand{\SL}[0]{{\operatorname{SL}}}

\newcommand{\Gal}[0]{{\operatorname{Gal}}}

\newtheorem{thm}{Theorem}[section]
\newtheorem{lem}[thm]{Lemma}
\newtheorem{cor}[thm]{Corollary}
\newtheorem{prop}[thm]{Proposition}

\newtheorem*{claim*}{Claim}         
\newtheorem{step}{Step}

\theoremstyle{definition}
\newtheorem{ex}[thm]{Example}

\newtheorem{dfn}[thm]{Definition}

\newtheorem{rem}[thm]{Remark}
      
\newtheorem{nota}[thm]{Notation}         
\newtheorem{nasi}[thm]{}

\makeatletter
  
  \@addtoreset{equation}{thm}
  \makeatother

\newcommand{\cred}{\color{red}}
\newcommand{\cyan}{\color{cyan}}

\newcommand{\MO}{\mathcal{O}}

\newcommand{\R}{\mathbb{R}}
\newcommand{\Q}{\mathbb{Q}}
\newcommand{\Z}{\mathbb{Z}}
\newcommand{\F}{\mathbb{F}}

\renewcommand{\P}{\mathbb{P}}

\newcommand{\cQ}{\mathcal{Q}}
\newcommand{\cC}{\mathcal{C}}
\newcommand{\cK}{\mathcal{K}}

\DeclareMathOperator{\td}{td}
\DeclareMathOperator{\ch}{ch}
\DeclareMathOperator{\Fl}{Fl}

\usepackage{listings}
\begin{document}

\maketitle

\begin{abstract}
We prove that a prime Fano threefold of genus 8 over an algebraically closed field of positive characteristic 
is isomorphic to a linear section of the Grassmannian variety Gr(2, 6). 
As applications, it is shown that a prime Fano threefold of genus 8 is irrational and, if the characteristic is larger than two, then it is globally F-regular. 
\end{abstract}

\tableofcontents




\section{Introduction}

The classification of Fano varieties has been an interesting topic in algebraic geometry. 
The projective line $\P^1$ is the unique one-dimensional Fano variety. 
Two-dimensional Fano varieties are called 
del Pezzo surfaces. 
It is classically known that a del Pezzo surface is isomorphic to 
either $\P^1 \times \P^1$ or the blowup of $\P^2$ along at most general eight points. 
Study of three-dimensional Fano varieties, called Fano threefolds, 
was initiated by Fano himself. 
A 
classification of Fano threefolds in characteristic zero was given by Mori--Mukai \cite{MM81}, \cite{MM83}, \cite{MM03} 
based on earlier works by Iskovskih 
and Shokurov \cite{Isk77}, \cite{Isk78}, \cite{Sho79a}, \cite{Sho79b} (cf. \cite{IP99}, \cite{Tak89}). 
After Mori--Mukai's classification, 
decisive descriptions for prime Fano threefolds had been announced by Mukai \cite{Muk89}, which was later confirmed by Bayer--Kuznetsov--Macri \cite{BKM24}, \cite{BKM25}. 

\begin{dfn}
Let $k$ be an algebraically closed field. 
We say that $X$ is a {\em prime Fano threefold} over $k$ 
if $X$ is a smooth projective threefold over $k$ 
such that $-K_X$ is ample and $\Pic X$ is generated by $\omega_X$. 
The positive integer $g$ defined by $(-K_X)^3 = 2g-2$ is called the {\em genus} of $X$. 
\end{dfn}


Recently, Mori--Mukai's birational classification was extended to the case of positive characteristic by Asai and the second author 
\cite{FanoI}, \cite{FanoII}, \cite{FanoIII}, \cite{FanoIV}. 
Then it is natural to seek a positive-characteristic analogue of Mukai's description for prime Fano threefolds. 
The purpose of this paper is to settle 
the case of genus $8$. 
Specifically, the main result of this paper is as follows. 


\begin{thm}[Theorem \ref{t main}]\label{intro main}
Let $X$ be a prime Fano threefold of genus $8$ over 
 an algebraically closed field $k$ of characteristic $p>0$. 
Then there exists a $9$-dimensional linear subvariety $L$ on $\P^{14}_k$ such that 
$X$ is isomorphic to the scheme-theoretic intersection $\Gr(6, 2) \cap L$, 
where $\Gr(6, 2) \subset \P^{14}$ is the Pl\"{u}cker embedding 
of the Grassmannian variety $\Gr(6, 2)$. 
\end{thm}

\noindent
In characteristic zero, the above result was established  
by Gushel and Mukai \cite{Gus83}, \cite{Gus92}, \cite{Muk89}. 


\begin{rem}
Let $X \subset \P^{g+1}$ be a prime Fano threefold of genus $g$ such that $|-K_X|$ is very ample. 
Then it is known that  $3 \leq g \leq 12$ and $g \neq 11$ \cite[Theorem 1.1]{FanoII}. 
If $g \leq 5$, then $X$ is a complete intersection in $\P^{g+1}$ \cite[Proposition 2.8]{FanoII}. 
Mukai's description for the case $g=6$ is given in \cite[Theorem 5.4 and Theorem 5.5]{KTLift1}. 
Theorem \ref{intro main} settles the case when $g=8$. 
Thus the remaining case is  when $g \in \{ 7, 9, 10, 12\}$. 
\end{rem}

\subsection{Applications}

\subsubsection{Irrationality}


It is known that a prime Fano threefold $X$ of genus $8$ 
in characteristic zero is birational to a smooth cubic threefold, and hence is irrational. 
We shall extend this result to the case of positive characteristic.  


\begin{thm}[Theorem \ref{thm:birational to cubic 3-fold}, Corollary \ref{c irrational}]\label{intro irrational}
Let $X$ be a prime Fano threefold of genus $8$ over an algebraically closed field of characteristic $p>0$. 
Then $X$ is birational to a smooth cubic threefold. Moreover, $X$ is not rational. 
\end{thm}

We now overview the proof of Theorem \ref{intro irrational}, 
although it is quite similar to that in characteristic zero (\cite{Fan30}, \cite{Isk80}, \cite{Put82}, \cite{Kuz04}). 
Let $X$ be a prime Fano threefold of genus $8$ over an algebraically closed field of characteristic $p>0$. 
By using  Theorem \ref{intro main}, 
we obtain a suitable smooth cubic threefold $Y$, 
which is defined as a hyperplane section of the Pfaffian cubic hypersurface $\mathcal C \subset \P^{14}$ (Definition \ref{d cubic of X}, Proposition \ref{p cubic of X is smooth}). 
Then it is enough to show that $X$ is birational to $Y$, 
because  an arbitrary smooth cubic threefold is known to be irrational \cite[Theorem A]{Ciu} (cf.\ \cite{CG72}, \cite{Mur73}). 

By construction, we can find locally free sheaves $E$  and  $F$   of rank two 
on $X$ and $Y$, respectively, 
and the following diagram: 
\[
\begin{tikzcd}
\P_X(E) \arrow[rd,"\alpha"] \arrow[d, "\P^1\text{-bundle}"']  &   & \P_Y(F) \arrow[ld,"\beta"'] \arrow[d, "\P^1\text{-bundle}"] \\
X&  \P^5 & Y.
\end{tikzcd}
\]
It holds that $\Im(\alpha) = \Im(\beta)$, which is a quartic fourfold in $\P^5$  (Proposition \ref{p Palatini quartic}). 
For a general hyperplane section $Z$ of 
this quartic fourfold  $\Im(\alpha) = \Im(\beta)$, 
the resulting diagram 
\[
\begin{tikzcd}
\alpha^{-1}(Z) \arrow[rd] 
\arrow[d]  &   & \beta^{-1}(Z) \arrow[ld] \arrow[d] \\
X&  Z & Y
\end{tikzcd}
\]
consists of birational morphisms. 
Therefore, $X$ is birational to $Y$. 
For more details, see Section \ref{s irrational}.



\subsubsection{F-splitting}

Many classes of Fano (type) varieties in positive characteristic are known to be $F$-split. 
For example, toric varieties and Schubert varieties   are $F$-split 
\cite[3.4 and the proof of Theorem 3]{BTLM97}, \cite[Theorem 2 in page 38]{MR85}. 
Then it is natural to ask whether lower dimensional smooth Fano varieties are $F$-split. 
In this direction, the one-dimensional Fano variety (i.e., $\P^1$) is $F$-split, 
whilst smooth del Pezzo surfaces are $F$-split when $p>5$ \cite[Example 5.5]{Har98}. 
Concerning smooth Fano threefolds, 
there exists an integer $p_0 >0$ such that 
an arbitrary smooth Fano threefold of characteristic $p>p_0$ 
is $F$-split (\cite[Theorem 1.2]{SS10}, \cite[Theorem 1.1]{FanoIV}). 
On the other hand,  any explicit lower bound $p_0$ is not known. 
Based on the description as in Theorem \ref{intro main}, 
we give the optimal lower bound that works for prime Fano threefolds of genus $8$.

\begin{thm}[Theorem \ref{t Fsplit optimal}, Remark \ref{intro p=2 non-F-split}]\label{intro F-split}
Let $X$ be a prime Fano threefold of genus $8$ 
over an algebraically closed field of characteristic $p>0$. 
Then the following hold. 
\begin{enumerate}
\item $X$ is quasi-$F$-regular. 
\item If $p>2$, then $X$ is globally $F$-regular. 
\end{enumerate}
\end{thm}

\begin{rem}\label{intro p=2 non-F-split}
\begin{enumerate}
\item 
By definition, globally $F$-regular (resp. quasi-$F$-regular) 
varieties are $F$-split (resp. quasi-$F$-split). 
\item Given a smooth Fano variety $X$ over an algebraically closed field $k$ of positive characteristic, $X$ is $F$-split (resp. quasi-$F$-split) 
if and only if $X$ is globally $F$-regular (resp. quasi-$F$-regular) 
\cite[Theorem 3.9 and Theorem 4.3]{SS10}, \cite[Corollary 7.6]{KTTWYY3}. 
\item 
The bound $p>2$ in Theorem \ref{intro F-split}(2) is optimal, that is, 
there exists a prime Fano threefold of genus $8$ in characteristic two 
which is not $F$-split (Proposition \ref{p p=2 non-F-split}).
This Fano threefold of genus 8 is the reduction modulo $2$ of a Fano threefold in \cite{GP01}, which is a birational model of the moduli space of $(1,11)$-polarized abelian surfaces with level structure of canonical type.
\end{enumerate}
\end{rem}


In what follows, we  overview only the proof of Theorem \ref{intro F-split}(1), 
as that of  Theorem \ref{intro F-split}(2) is similar but technically more complicated. 
By the Cartier operator criterion for quasi-$F$-splitting \cite[Theorem F]{KTTWYY1}, 
it is enough to prove the following  (cf. Lemma \ref{l GFS criterion}(1)): 
\[
\text{$H^2(X, \Omega_X^1(k))=0$ for every $k\geq p$.}
\]
Using the conormal exact sequence with respect to the closed embedding $X \subset \Gr(6, 2)$ as in Theorem \ref{intro main}, 
the problem is reduced to showing 
\[\text{
$H^i(\Gr(6, 2), \Omega^1_{\Gr(6, 2)}(k))=0$ if $i \geq 2$ and $k\geq p-5$.
}
\]
The proof of this equality is divided into the following two cases: 
\[
\text{(I)}\,\,k >0
\hspace{10mm} 
\text{(II)}\,\,k \leq 0. 
\]

(I) Assume $k>0$. 
For the universal subbundle $\cK$ (resp.\ the quotient bundle $\cQ$) on $\Gr(6, 2)$, 
we have 
\[
\Omega^1_{\Gr(6, 2)} \simeq \cK \otimes \cQ^{\vee} \simeq \pi_*\mathcal L
\]
for a suitable invertible sheaf $\mathcal L$ on $F := \P_{\Gr(6, 2)}(\cK) \times_{\Gr(6, 2)}\P_{\Gr(6, 2)}(\cQ^{\vee})$ and 
the induced morphism $\pi : F \to \Gr(6, 2)$. 
Then we can show 
\[
H^i(\Gr(6, 2), \Omega^1_{\Gr(6, 2)}(k)) \simeq H^i(F, \mathcal L \otimes \pi^*\MO_{\Gr(6, 2)}(k)) \overset{(\star)}{=}0, 
\]
where 
$(\star)$ follows from the Kodaira vanishing, which is applicable by $k>0$ and the fact that the variety $F$ is $F$-split. 
For more details, see Subsection \ref{ss QFS}. 

\medskip

(II) Assume $k\leq 0$. 
This, together with $k \geq p-5$, implies  $p-5 \leq k \leq 0$. 
Hence there are only finitely many possibilities for 
the pair $(p, k)$. 
Then the dimension  $h^i(\Gr(6, 2), \Omega^1_{\Gr(6, 2)}(k))$ is determined by using 
computer algebra systems, e.g., 
Macaulay2 \cite{M2}. 
In particular, our proof of Theorem \ref{intro F-split} is computer-assisted. 
Moreover, the non-$F$-split example of characteristic two mentioned in Remark \ref{intro p=2 non-F-split} will be  also confirmed by using Macaulay2. 
For more details, see Subsection \ref{ss Macaulay2}. 







\subsection{Overview of the proof of Theorem \ref{intro main}}

We start by recalling the strategy in characteristic zero 
(cf.\ \cite{Gus83}, \cite{Gus92}). 
Let $X \subset \P^9$ be a prime Fano threefold of genus $8$ over an algebraically closed field of characteristic zero. 
The main part is to construct a suitable vector bundle $E$ on $X$ of rank $2$. 
The argument consists of the following three steps. 
\begin{enumerate}
\item[(0-a)] Construction of a quintic elliptic curve $C$ on $X$. 
\item[(0-b)] Existence of K3 surfaces $S \in |-K_X|$ with $C \subset S$. 
\item[(0-c)] Construction of a vector bundle $E$ via elementary transform by using $C \subset S \subset X$. 
\end{enumerate}

(0-a) 
Take a general point $P$ and let $\sigma : Y \to X$ be the blowup at $P$. 
Then it is easy to see that $|-K_Y|$ is base point free and 
the image $Z := \varphi_{|-K_Y|}(Y) \subset \P^5$ is a complete intersection $Z = H_2 \cap H_3$, where 
each $H_d$ is a hypersurface on $\P^5$ of degree $d$. 
Moreover, $H_2$ is smooth and $\psi : Y \to Z$ is a flopping contraction. 
For $E := \Ex(\sigma) (\simeq \P^2)$ and $E_Z := \psi(E)$, 
we see that $\psi|_E : E \xrightarrow{\simeq} E_Z$ and $E_Z \subset \P^5$ is a Veronese surface. 
By an explicit computation on $\CH(H_2)$, we see that $E_Z = V+3V'$ in $\CH(H_2)$ for some planes $V$ and $V'$ on $H_2$. 
After replacing $V$ by a general member of its deformation family, 
we may assume that the intersection $C_Z := H_3 \cap V$ is a smooth cubic curve on the plane $V =\P^2$. 
For the proper transform $C_Y \subset Y$ of $C_Z$ and $C := \sigma(C_Y) \subset X$, 
we have $C_Z \xleftarrow{\simeq} C_Y \xrightarrow{\simeq} C$ and $C$ is a quintic elliptic curve on $X$. 

\medskip

(0-b) Let $\tau : \wt{X} \to X$ be the blowup along $C$. 
We can check that $\langle C \rangle \cap X = C$, 
which implies that $|-K_{\wt{X}}|$ is base point free. 
Take a general member $T \in |-K_{\wt{X}}|$ and set $S := \tau(T) \in |-K_X|$. 
By using the fact that $T|_{\Ex(\tau)}$ is smooth, 
we see that also $S$ is a smooth projective surface. 
By construction, we have $C \subset S$. 
It follows from the adjunction formula that $S$ is a K3 surface.

\medskip

(0-c) 
Since the complete linear system $|C|$ on $S$ is base point free, 
we have a surjection $V \otimes \MO_S \to \MO_S(C)$ for $V := H^0(S, \MO_S(C))$. 
By using $\dim_k V = 2$, we get the following exact sequence: 
\[
0 \to \MO_S(-C) \to V \otimes \MO_S \to \MO_S(C) \to 0. 
\]
Taking the tensor product with $\MO_S(S) := \MO_X(S)|_S$, we get the following homomorphism between exact sequences: 
\[
\begin{tikzcd}
0 \arrow[r]  &
E \arrow[r] \arrow[d] & 
V \otimes_k \MO_X(S) \arrow[d] \arrow[r] &
\MO_S(S+C) \arrow[d, equal] \arrow[r] & 0\\
0 \arrow[r]  &
\MO_S(S-C) \arrow[r] &
V \otimes_k \MO_S(S) \arrow[r] &
\MO_S(S+C) \arrow[r]  &0,
\end{tikzcd}
\]
where 
the right square consists of the natural arrows, 
$E$ is defined as the kernel, and the left vertical arrow is the induced one. 
Then it is well known that  $E$ is a locally free sheaf on $X$ of rank $2$. 

\medskip

We now switch to the situation of characteristic $p>0$. 
Let $X \subset \P^9$ be a prime Fano threefold of genus $8$ over an algebraically closed field of characteristic $p>0$. 
As in characteristic zero, we construct a suitable vector bundle $E$ on $X$ of rank $2$. 
Corresponding to (0-a), (0-b), (0-c), the proof consists of the following three steps. 
\begin{enumerate}
\item[(p-a)] Construction of a regular plane cubic curve $C$ on $X_{\kappa} := X \times_k \kappa$ for some field extension $k \subset \kappa$. 
\item[(p-b)] Existence of K3-like surfaces 
$S \in |-K_{X_{\kappa'}}|$ with $C_{\kappa'} \subset S$ 
for some further field extension $k \subset \kappa \subset \kappa'$. 
\item[(p-c)] Construction of a vector bundle $E$ on $X$ via elementary transform by using 
$C_{\kappa'} \subset S \subset X_{\kappa'}$. 
\end{enumerate}

(p-a) 
We apply the same argument as in (0-a). 
Then we do not see whether the intersection $H_3 \cap V$ (which was the definition of $C_Z$ in characteristic zero) is actually smooth. 
By taking the generic member instead of taking a general member $V$, we can find a 
(possibly non-smooth) regular 
plane cubic curve $C_Z$ on $X_{\kappa} := X \times_k \kappa$ for a suitable field extension $k \subset \kappa$. 
Here a general member $V$ is a general fibre of the suitable flat family 
\[
\mathcal V \hookrightarrow H_2 \times_k \P^3_k \to \P^3_k, 
\]
whilst the generic member is nothing but the generic fibre of 
this family $\mathcal V \to \P^3_k$. 
Hence the new base field $\kappa$ is the function field of $\P^3_k$. 
As in (0-a), $C \subset X_{\kappa}$ is defined as the proper transform of $C_Z$.

\medskip

(p-b) 
Let $\tau : \wt{X} \to X_{\kappa}$ be the blowup along a regular curve $C$. 
As in (0-b), we see that $|-K_{\wt{X}}|$ is base point free. 
The generic member $T$ of $|-K_{\wt{X}}|$ is a regular prime divisor on $\wt{X}_{\kappa'} := \wt{X} \times_{\kappa} \kappa'$ for some field extension $\kappa \subset \kappa'$. 
By the same argument as in (0-b), we see that $S := \tau(T)$ is still a regular projective surface on $X_{\kappa'} := X \times_k \kappa'$ such that $S \sim -K_{X_{\kappa'}}$. 
Such a surface $S$ is called a K3-like surface, as we have $K_S \sim 0$ and $H^1(S, \MO_S)=0$. 

\medskip

(p-c) 
Applying the same argument as in (0-c), we can find a vector bundle $E'$ on $X_{\kappa'}$ 
of rank $2$ satisfying suitable properties. 
This vector bundle descends to $X_R := X \times_k R$ for some intermediate ring $k \subset R \subset \kappa'$ which is finitely generated over $k$. 
Then taking the fibre of $X_R \to \Spec R$ over a general closed point of $\Spec R$, 
we get a required vector bundle $E$ on $X$.

\begin{rem}
Replacing $S$ by a general member of $|-K_{X_\kappa}|$, we may avoid using the second field extension $\kappa \subset \kappa'$ in (p-c) 
(cf.\ Subsection \ref{ss construction vb}). 
\end{rem}

\section{Preliminaries}

\subsection{Notation}\label{ss-notation}

In this subsection, we summarise notation used in this paper. 

\begin{enumerate}
\item We will freely use the notation and terminology in \cite{Har77} and \cite{KM98}. 
In particular, $D_1 \sim D_2$ means linear equivalence of Weil divisors. 
\item 
Throughout this paper, 
we work over an algebraically closed field $k$ 
of characteristic $p>0$ unless otherwise specified. 
\item For an integral scheme $X$, 
we define the {\em function field} $K(X)$ of $X$ 
as the local ring $\MO_{X, \xi}$ at the generic point $\xi$ of $X$. 
For an integral domain $A$, $K(A)$ denotes the function field of $\Spec A$. 
\item 
For a scheme $X$, its {\em reduced structure} $X_{\red}$ 
is the reduced closed subscheme of $X$ such that the induced closed immersion 
$X_{\red} \to X$ is surjective. 
\item 
Our notation will not distinguish between invertible sheaves and 
Cartier divisors. For example, $\Pic X= \Z K_X$ 
means that $\Pic X$ is generated by $\omega_X$ as an abelian group. 
\item We say that $X$ is a {\em variety} (over $k$) if 
$X$ is a separated integral scheme which is of finite type over $k$. 
We say that $X$ is a {\em curve} (resp. a {\em surface}, resp. a {\em threefold})  
if $X$ is a variety over $k$ of dimension one (resp. {\em two}, resp. {\em three}). 
\item 
Given a variety $Y$ and a closed subscheme $Z$ of $Y$, 
$\Bl_Z\,Y$ denotes the blowup of $Y$ along $Z$. 
In this case, $Z$ is called {\em the (blowup) centre} of the induced blowup $\Bl_Z\,Y \to Y$. 
Let  $\Ex(f)$ be the exceptional divisor 
equipped with reduced scheme structure. 
In particular, if $Y$ is a smooth threefold and $Z$ is a smooth curve on $Y$, 
then we have $K_X \sim f^*K_Y +\Ex(f)$ for $X := \Bl_Z\,Y$.
\item We say that $f: X \to Y$ is a {\em contraction} if $f$ is a morphism of schemes satisfying $f_*\MO_X = \MO_Y$. 
Here the equality $f_*\MO_X = \MO_Y$ means that the induced ring homomorphism $\MO_Y \to f_*\MO_X$ is an isomorphism. 

\item 
For the definition of types of extremal rays for smooth projective threefolds, 
we refer to \cite[Definition 3.3]{FanoIII}. 
\item Given a variety $X$, $\Sing X$ denotes the singular locus of $X$. 
\item We say that $T \subset \P^5$ is a {\em Veronese surface} if 
$T \subset \P^5$ is isomorphic to the image of $\P^2$ by the morphism of the complete linear system $|\MO_{\P^2}(2)|$. 
Note that Veronese surfaces are unique up to $\Aut(\P^5)$. 
\item 
Given a closed subscheme $Z$ on a projective space $\P^N_k$, 
$\langle Z \rangle$ denotes the smallest linear subvariety of $\P^N_k$ containing $Z$. 
\item 
Take integers $n$ and $r$ satisfying $n > r >0$. 
For an $n$-dimensional $k$-vector space $V$, 
let $\Gr(r, V)$ (resp. $\Gr(V, r)$) be the Grassmannian variety 
parametrising $r$-dimensional $k$-vector subspaces (resp. quotients). 
Set $\Gr(r, n) := \Gr(r, k^n)$ and $\Gr(n, r) := \Gr(k^n, r)$.

Set $\P(V) = \Gr(V,1)$, which is the Grothendieck projectivisation $\Proj(S(V))$ of the vector space $V$. 
\item Similarly, for integers $n>k_1>\cdots >k_m>0$, $\Fl(V;k_1,\dots,k_m)$ denotes the flag variety parametrising flags of quotients $V \to V_1 \to \cdots \to V_m$ with $\dim V_i = k_i$.
Set $\Fl(n;k_1,\dots,k_m) := \Fl(k^n;k_1,\dots,k_m)$.
\item 
A {\em hypersurface} $H$ of degree $d$ on $\P^N_k$ is a prime divisor on $\P^N_k$ satisfying 
$\MO_{\P^N}(H) \simeq \MO_{\P^N}(d)$. 
In particular, a hypersurface is an integral scheme under our terminologies. 
A {\em quadric} (resp.\ {\em cubic}) hypersurface is a hypersurface of degree $2$ (resp. $3$). 
\item 
We say that $X \subset \P^N_k$ is an {\em intersection of quadrics} if 
the equality 
${\displaystyle X = \bigcap_Q Q}$ of closed subschemes on $\P^N_k$ holds, 
where ${\displaystyle \bigcap_Q Q}$ denotes the scheme-theoretic intersection and $Q$ runs over all the quadric hypersurfaces on $\P^N_k$ containing $X$. 
\item 
We set $\dim \emptyset := -\infty$ when the empty set $\emptyset$ is considered as a topological space. 
\end{enumerate}

\subsection{Fano threefolds}

We say that $X$ is a {\em Fano threefold} (over $k$) 
if $X$ is a smooth projective threefold over $k$ such that $-K_X$ is ample. 
We say that $X$ is a {\em prime Fano threefold} 
if $X$ is a Fano threefold such that $\Pic X  = \Z K_X$. 
For a prime Fano threefold $X$, 
the integer $g$ defined by $(-K_X)^3 = 2g -2$ is called the {\em genus} of $X$. 
It is well known that the following hold  \cite[Theorem 1.1]{FanoI}, \cite[Theorem 1.2]{FanoII}: 
\begin{enumerate}
\item $2 \leq g \leq 12$ and $g \neq 11$. 
\item If $g \geq 4$, then $|-K_X|$ is very ample. 
\end{enumerate}
If $X$ is a prime Fano threefold and $|-K_X|$ is very ample, 
then we have the closed embedding $X \subset \P^{g+1}$ induced by $|-K_X|$. 
In this case, $X \subset \P^{g+1}$ is also called a prime Fano threefold of genus $g$. 
More rigorously, we say that $X \subset \P^{g+1}$ is a 
{\em prime Fano threefold of genus $g$}  
if $X$ is a prime Fano threefold of genus $g$, 
$X$ is a closed subscheme of $\P^{g+1}$, $\MO_{\P^{g+1}}(1) \simeq \MO_X(-K_X)$, and the induced map 
\[
H^0(\P^{g+1}, \MO_{\P^{g+1}}(1)) \to H^0(X, \MO_{\P^{g+1}}(1)|_X)
\]
is an isomorphism.


\begin{rem}
Let $X$ be a prime Fano threefold of genus $8$, that is, 
a Fano threefold satisfying $\Pic X = \Z K_X$ and $(-K_X)^3 =14$. 
Then the following hold. 
\begin{enumerate}
\item $|-K_X|$ is very ample \cite[Theorem 1.1]{FanoI} and 
we have 
the closed embedding $X \subset \P^9$ induced by $|-K_X|$. 
\item 
For integers $m \geq 0$ and $i>0$, we have $H^i(X, -mK_X) = 0$  and 
the following  \cite[Corollary 4.5]{FanoI}: 
\[
h^0(X, -mK_X)= \frac{7}{6}m(m+1)(2m+1) + (2m+1).  
\]
\item 
The graded $k$-algebra $\bigoplus_{d \geq 0} H^0(X, -dK_X)$ is generated by $H^0(X, -K_X)$ as a $k$-algebra \cite[Theorem 6.2]{FanoI}. 
\item $X \subset \P^9$ is an intersection of quadrics \cite[Theorem 1.2]{FanoI}. 
\end{enumerate}

\end{rem}


\subsection{Quadric hypersurfaces}

\begin{prop}\label{p quadric CH}
Let $Q \subset \P^{n+1}$ be a smooth quadric hypersurface. 
Set $H \in \CH^1(Q)$ to be the class of a hyperplane section. 
For every $0 \leq i < n/2$, 
let $L_i$ be the class of an $i$-dimensional linear subvariety on $Q$. 
Then $\CH^i(Q)$ is a free $\Z$-module. 
Moreover, the following hold. 
\begin{enumerate}
\item 
If $0 \leq  i < n/2$, then $\CH^i(Q) = \Z H^i$. 
\item 
If $n/2 < i \leq n$, then $\CH^i(Q) = \Z L_{n-i}$. 
\item 
Assume $m := n/2 \in \Z$. 
Then there exist exactly two classes $L$ and $L'$ of $m$-dimensional linear subvarieties on $Q$. 
Furthermore, $H^m = L + L'$ 
and $\CH^m(Q) = \Z L \oplus \Z L'$. 
\end{enumerate}
\end{prop}

\begin{proof}
All the assertions 
follow from \cite[Proposition 68.1 and Proposition 68.2]{EKM08}. 
\end{proof}

\begin{rem}\label{r quad LL' symmetry}
Let $Q \subset \P^{n+1}$ be a smooth quadric hypersurface with $n \in 2\Z$. 
For $m := n/2$, let $L$ and $L'$ be the two different classes of 
$m$-dimensional linear subvarieties on $Q$. 
Then there exists an automorphism $\iota: Q \xrightarrow{\simeq} Q$ such that $\iota^2 = {\rm id}$ and $\iota_*(L) = L'$. 

Indeed, we may assume 
\[
Q = \{ x_0x_{2m+1} +x_1x_{2m}+ x_2 x_{2m-1}+\cdots + x_{m}x_{m+1} =0\},\
\]
\[
L= \{ x_0=x_1= \cdots = x_{m}=0\}, \qquad 
L'= \{ x_{2m+1}  = x_{2m} = \cdots = x_{m+1}=0\}. 
\]
Then the automorphism $\iota: Q \to Q, x_i \mapsto x_{2m+1-i}$ satisfies 
$\iota^2 = {\rm id}$ and $\iota(L)=L'$. 
\end{rem}

\begin{lem}\label{l quad Vero Sing neq pt}
Let $Q \subset \P^5$ be a quadric hypersurface. 
Assume that there exists a Veronese surface $T \subset \P^5$ satisfying $T \subset Q$. 
Then $\dim(\Sing Q)$ is equal to $-\infty, 1$, or $2$.  
\end{lem}

\begin{proof}
Fix homogeneous coodinates on $\P^2_k$ and $\P^5_k$ as follows: 
\[
\P^2_k = \Proj k[x_0, x_1, x_2], \qquad 
\P^5_k = \Proj k[y_0, y_1, y_2, z_0, z_1, z_2]. 
\]
We may assume that $T =i(\P^2)$ for the closed immersion 
\[
i: \P^2_k \hookrightarrow \P^5_k, \qquad 
[x_0, x_1, x_2] \mapsto [x_0^2 : x_1^2 : x_2^2: x_1x_2: x_2x_0: x_0x_1],  
\]
and hence 
\[
T = \{ y_0y_1 = z_2^2, y_1y_2 = z_0^2, y_2y_0 = z_1^2, 
z_0z_1 = z_2y_2, z_1z_2=z_0y_0, z_2z_0 = z_1y_1\}. 
\]

For $V := H^0(\P^2, \MO_{\P^2}(1))$, 
we now construct the following commutative diagram of $k$-linear maps in which each horizontal sequence is exact:  
\[
\begin{tikzcd}
0 \arrow[r] 
  & H^0(\P^5, I_T \otimes \MO_{\P^5}(2)) \arrow[r, "\varphi"] \arrow[d, "\alpha", "\simeq"'] 
  & H^0(\P^5, \MO_{\P^5}(2)) \arrow[r, "\psi"] \arrow[d, "\beta", "\simeq"'] 
  & H^0(T, \MO_{\P^5}(2)|_T) \arrow[r] \arrow[d, "\gamma", "\simeq"'] 
  & 0 \\
0 \arrow[r] 
  & S^2(\wedge^2 V) \arrow[r, "\varphi'"] 
  & S^2(S^2(V)) \arrow[r, "\psi'"] 
  & S^4(V) \arrow[r] 
  & 0. 
\end{tikzcd}
\]
The upper sequence is the standard one and the  lower one is given
as follows.  
\begin{itemize}
\item $\varphi'((a\wedge b)(c\wedge d)) = (ac)(bd)-(ad)(bc)$. 
\item $\psi'((ab)(cd)) = abcd$. 
\end{itemize}
Here $\gamma$ is the isomorphism given  by 
\[
H^0(T, \MO_{\P^5}(2)|_T) \simeq H^0(\P^2, \MO_{\P^2}(4)) \simeq S^4(V),  
\]
and $\beta$ is the isomorphism induced by 
\[
H^0(\P^5, \MO_{\P^5}(2)) \simeq S^2(H^0(\P^5, \MO_{\P^5}(1))) 
\xrightarrow{\simeq} S^2( H^0(\P^2, \MO_{\P^2}(2))) 
\simeq S^2( S^2(V)). 
\]
It is easy to check that the right square is commutative, i.e., 
$\gamma \circ \psi = \psi' \circ \beta$. 
Note that $\dim_k(S^2(\wedge^2 V)) =6$, $\dim_k(S^2(S^2(V))) = 21,$ and 
$\dim_k(S^4(V)) = 15$. 
Then the lower horizontal sequence is exact, 
because $\varphi'$ is injective, $\psi'$ is surjective, and $\psi' \circ \varphi' =0$. 
Hence also the upper horizontal sequence is exact. 
Let $\alpha$ be the induced isomorphism between the kernels. 

By using the $k$-linear basis $V = H^0(\P^2, \MO_{\P^2}(1)) 
= kx_0 \oplus k x_1 \oplus kx_2$, 
we have $\SL_3 := \SL_3(k) \simeq \SL(V)$, and hence 
$V$ is a $k[\SL_3]$-module. 
Then the lower horizontal sequence is an exact sequence 
of  $k[\SL_3]$-modules. 
Via the vertical isomorphisms, 
we may assume that 
also the upper  sequence is an exact sequence 
of  $k[\SL_3]$-modules. 
Since $\SL_3$ is a semisimple algebraic group, 
there exists a unique one-dimensional $k[\SL_3]$-module up to isomorphisms, and hence we have a $k[\SL_3]$-module isomorphism  $\theta: \bigwedge^3 V \xrightarrow{\simeq} k$, where $k$ is equipped with the trivial $\SL_3$-action. 
Then we get the following $k[\SL_3]$-module isomorphisms:  
\[
\zeta: \bigwedge^2 V \xrightarrow{\simeq} \Hom_k(V, \bigwedge^3 V) 
\xrightarrow{\simeq, \theta \circ} \Hom_k(V, k) =:V^*, 
\]
which induces another $k[\SL_3]$-module isomorphism 
\[
\xi := S^2(\zeta) : S^2(\bigwedge^2 V) \xrightarrow{\simeq} S^2(V^*). 
\]
Since $S^2(V^*)$ is the space of quadratic forms on $V$, 
there exist exactly three nonzero $\SL_3$-orbits 
represented by $X_0X_1 + X_2^2, X_0X_1, X_2^2$, 
where 
$X_0, X_1, X_2$ are the $k$-linear basis of $V^*$ defined by 
$\zeta(x_0 \wedge x_1) =:X_2, 
\zeta(x_1 \wedge x_2)=: X_0,
\zeta(x_2 \wedge x_0)=: X_1$. 
We have 
\begin{align*}
\beta^{-1} \circ \varphi' \circ \xi^{-1}(X_0X_1) 
&=\beta^{-1} \circ \varphi'( (x_1 \wedge x_2)(x_2 \wedge x_0))\\
&= 
\beta^{-1}( (x_1x_2)(x_2x_0)  - (x_1x_0)(x_2^2))\\
&= z_0z_1 - z_2y_2,
\end{align*}
\begin{align*}
\beta^{-1} \circ \varphi' \circ \xi^{-1}(X_2^2) 
&=\beta^{-1} \circ \varphi'( (x_0\wedge x_1)(x_0\wedge x_1)) )\\
&= 
\beta^{-1}( (x_0^2)(x_1^2)  - (x_0x_1)(x_0x_1))\\
&= y_0y_1 - z_2^2,
\end{align*}
\[
\beta^{-1} \circ \varphi' \circ \xi^{-1}(X_0X_1+X_2^2) =  
z_0z_1 - z_2y_2 + y_0y_1 - z_2^2. 
\]
Moreover, the $\SL_3$-action on $H^0(\P^5, \MO_{\P^5}(2))$ preserves the isomorphism class of $Q$. 
Therefore, $\dim(\Sing Q)$ is equal to 
$1$, $2$, or $-\infty$. 
\qedhere


\end{proof}

\begin{lem}\label{l Veronese V+3V'}
Let $T \subset \P^5$ be a Veronese surface and let $Q \subset \P^5$ be a 
smooth quadric hypersurface such that $T \subset Q$. 
Then there exist planes $V$ and $V'$ on $Q$ such that 
$\CH^2(Q) = \Z V \oplus \Z V'$, 
$V^2 =V'^2 =1$, $V \cdot V' =0$, and $T = V + 3V'$. 
\end{lem}

Here a plane means a two-dimensional linear subvariety. 

\begin{proof}
We use the same notation as in the proof of Lemma \ref{l quad Vero Sing neq pt}. 
In particular, 
$T$ is the image of 
\[
i: \P^2_k \hookrightarrow \P^5_k, \qquad 
[x_0, x_1, x_2] \mapsto [x_0^2 : x_1^2 : x_2^2: x_1x_2: x_2x_0: x_0x_1],  
\]
and 
hence 
\[
T = \{ y_0y_1 = z_2^2, y_1y_2 = z_0^2, y_2y_0 = z_1^2, 
z_0z_1 = z_2y_2, z_1z_2=z_0y_0, z_2z_0 = z_1y_1\}. 
\]
Since 
all the smooth quadric hypersurfaces in $\P^5$ containing $T$ 
form a single $\SL_3$-orbit in $H^0(\P^5, I_T \otimes \MO_{\P^5}(2))$ (cf. the proof of  Lemma \ref{l quad Vero Sing neq pt}), 
we may assume that  
\[
Q = \{ z_0z_1 - z_2y_2 + y_0y_1 -z_2^2=0\} \subset \Proj k[y_0 y_1, y_2, z_0, z_1, z_2] = \P^5. 
\]

There exist planes $V$ and $V'$ on $Q$ such that 
$\CH^2(Q) = \Z V \oplus \Z V'$ (Proposition \ref{p quadric CH}). 
Moreover, $V^2 = V'^2 =1$ (cf. \cite[Exercise 68.3]{EKM08}). 
By $2 = H^4 = (V+V')^2 =V^2 + 2L\cdot V' + V'^2$, we have $V \cdot V' = 0$. 
We may assume that 
\[
V' = \{ y_0=z_0 = z_2 =0\}, 
\]
because the classes of $V$ and $V'$ can be switched (Remark \ref{r quad LL' symmetry}) and 
any plane on $Q$ is rationally equivalent to either $V$ or $V'$. 
We then get the following scheme-theoretic isomorphism 
\[
T \cap V' \simeq i^{-1}(V') = \{ x_0^2 = x_1x_2 =x_2x_0=0\}. 
\]
The rightmost term is set-theoretically equal to 
the two-point set $\{[0:0:1], [0: 1:0]\} (\subset \P^2)$. 
On the other hand, it is scheme-theoretically isomorphic to $\Spec k \amalg \Spec k[x_0]/(x_0)^2$, which implies $T \cdot V' =3$. 

We can write  $T = a V + a'V'$ for some integers $a$ and $a'$. 
Since $T \subset \P^5$ is of degree $4$, 
it holds that $4 = T \cdot H^2 = (aV+a'V') \cdot H^2 =a+a'$. 
By $T \cdot V' =3$, we get $3 = T \cdot V' = (aV+a'V') \cdot V' = a'$, and hence $T = V+3V'$. 
\end{proof}

\subsection{Grassmannian varieties}

In this subsection, we summarise some known results on 
Grassmannian varieties for later usage. 


\begin{lem}\label{l univ bdl det}
Take integers $n > m \geq 1$. 
Let $\mathcal Q$ (resp. $\cK$) be the universal quotient bundle (resp. subbundle) on $\Gr(n,m)$. 
Then $\det \cQ \simeq \MO_{\Gr(n,m)}(1)$ and $\det \cK \simeq  \MO_{\Gr(n,m)}(-1)$, 
where $\MO_{\Gr(n,m)}(1)$ denotes the ample generater of $\Pic (\Gr(n,m)) (\simeq \Z)$. 
\end{lem}

\begin{proof}
See \cite[Section 5.6.2]{EH16} or \cite[14.6 and 14.7]{Ful98}.
\end{proof}

\begin{lem}\label{l h0 compute}
It holds that $h^0(\Gr(6, 2), \MO_{\Gr(6, 2)}(2)) = 105$. 
\end{lem}

\begin{proof}
If the base field $k$ is of characteristic zero, then 
the assertion follows from 
\[
h^0(\Gr(6, 2), \MO_{\Gr(6, 2)}(2)) = 
\prod_{j=3}^6 \frac{ \binom{j+1}{2}}{\binom{j-1}{2}}
=
\frac{\binom{4}{2} \binom{5}{2} \binom{6}{2} \binom{7}{2} }{\binom{2}{2} \binom{3}{2} \binom{4}{2} \binom{5}{2} } = 105,
\]  
where the first equality follows from Weyl's character formula. 

Assume that $k$ is of characteristic $p>0$. 
Since $\Gr(6, 2)$ is $F$-split, it satisfies the Kodaira vanishing, and hence 
\[
h^0(\Gr(6, 2), \MO_{\Gr(6, 2)}(2)) = \chi(\Gr(6, 2), \MO_{\Gr(6, 2)}(2)). 
\]
As $\Gr(6, 2)$ lifts to $W(k)$, 
the number $\chi(\Gr(6, 2), \MO_{\Gr(6, 2)}(2))$ is equal to the one of  characteristic zero. 
\end{proof}




\section{Birationality of $|-K_Y|$: point blowup}\label{s birationality}

\begin{nota}\label{n g=8 1}
Let $X \subset \P^{9}$ be a prime  Fano threefold of genus $8$. 
Fix a closed point $P$ on $X$ which is not contained in any line on $X$, 
whose existence is guaranteed by \cite[Proposition 5.4]{FanoII}. 
Let $\sigma: Y \to X$ be the blowup at $P$. 
Recall that $|-K_Y|$ is base point free and $-K_Y$ is big 
\cite[Proposition 2.1(1)]{KTLift1}. 
For the induced morphism $\varphi_{|-K_Y|}: Y \to \P^5$ 
(Proposition \ref{p FanoY cont}(1)) 
and its image $\ol{Z} := \varphi_{|-K_Y|}(Y) \subset \P^5$, 
take the Stein factorisation of the induced morphism $\ol{\psi} : Y \to \overline{Z}$: 
\[
\ol{\psi} : Y \xrightarrow{\psi} Z \xrightarrow{\theta} \overline{Z}. 
\]
Set $A_{\overline Z} := \MO_{\P^5}(1)|_{\overline Z}$. 
Note that $\dim \Ex(\psi)=1$ \cite[Proposition 6.2, Corollary 6.5]{FanoII}.  
Let $\psi^+ : Y^+ \to Z$ be the flop of $\psi : Y \to Z$ 
and let $\tau : Y^+ \to W$ be the contraction of the  $K_{Y^+}$-negative extremal ray. 
Note that $\tau$ is of type $E_1$ and 
$W$ is a smooth quadric or cubic hypersurface on $\P^4$ \cite[Theorem 6.3]{FanoII}. 

\end{nota}

The purpose of this section is to prove that, in the above situation, 
$\theta$ is an isomorphism (i.e., $Z = \overline Z$) and 
$Z  \subset \P^5$ is a complete intersection of a quadric and a cubic (Theorem \ref{t g=8 Z 2cap3}). 

\begin{prop}\label{p FanoY cont}
We use Notation \ref{n g=8 1}. 
Then the following hold. 
\begin{enumerate}
\item $h^0(Y, -K_Y) = 6$ and the restriction map 
$H^0(Y, -K_Y)  \to H^0(E, -K_Y|_E)$ is an isomorphism. 
\item $\deg \theta =1$ or $\deg \theta =2$. 
\item If $\deg \theta =1$, then $\theta$ is an isomorphism  and 
$Z= \overline{Z}$ is a complete intersection $H_2 \cap H_3$ in $\P^5$, 
where each $H_d$ is a prime divisor on $\P^5$ of degree $d$. 
\item If $\deg \theta =2$, then $\Delta(\overline{Z}, A_{\overline Z})=0$, 
where $\Delta(\overline{Z}, A_{\overline Z}) := \dim \overline{Z} + A_{\overline Z}^3 - h^0(\overline Z, A_{\overline Z})$. 
\end{enumerate}
\end{prop}

\begin{proof}
Let us show (1). 
By \cite[Proposition 6.6]{FanoII}, 
the restriction map 
$H^0(Y, -K_Y)  \to H^0(E, -K_Y|_E)$ is an isomorphism. 
It holds that  
\[
h^0(Y, -K_Y) = h^0(E, -K_Y|_E) = 6, 
\]
where the latter equality follows from 
\cite[Proposition 6.1(3)]{FanoII}. 
Thus (1) holds. 
It follows from \cite[Proposition 2.1(2)(4)]{KTLift1} that (2) and (4) holds.

Let us show (3). 
Assume $\deg \theta =1$. 
By the same argument as in \cite[Theorem 6.2]{FanoI}, 
\[
\bigoplus_{m =0}^{\infty} H^0(Z, \MO_Z(-mK_Z))
\]
 is generated by $H^0(Z, \MO_Z(-K_Z))$ as a $k$-algebra. 
 In particular, $\theta$ is an isomorphism. 
It is enough to show that $Z= \overline{Z}$ is a complete intersection $H_2 \cap H_3$ in $\P^5$. 
We apply the same argument as in \cite[Proposition 2.8]{FanoII}. 
Then it suffices to prove $H^i(Z, -mK_Z) =0$ for every $i>0$ and $m \geq 0$. 
The $m=0$ case (i.e., $H^i(Z, \MO_Z) =0$ for $i>0$) 
follows from 
$H^{>0}(Y, \MO_Y)=0$ and $R^{>0}\psi_*\MO_Y = 0$ \cite[Proposition 6.10]{Tan-elliptic}. 

Fix $i>0$ and assume $m>0$. 
We prove $H^i(Z, -mK_Z)=0$ by induction on $m$. 
Since $Z$ is smooth outside finitely many closed points, 
the generic member $S$ of $|-K_Z|$ is regular. 
Let $\kappa$ be the function field of $\P(H^0(Z, -K_Z))$. 
We set $(-)_{\kappa} := (-) \times_k \kappa$ (e.g., $Z_{\kappa} := Z \times_k \kappa$). 
It follows from the adjunction formula that  $K_S \sim (K_{Z_{\kappa}} + (-K_{Z_{\kappa}})) =0$. 
By the exact sequence 
\[
0 \to \MO_{Z_{\kappa}}(-S) \to \MO_{Z_{\kappa}} \to \MO_S \to 0,
\]
together with $H^1(Z_{\kappa}, \MO_{Z_{\kappa}})=0$ 
and $h^2(Z_{\kappa}, \MO_{Z_{\kappa}}(-S) =h^1(Z_{\kappa}, \MO_{Z_{\kappa}})=0$, 
we get $H^1(S, \MO_S)=0$. 
Then $S$ is a regular K3-like surface in the sense of \cite[Definition 3.1]{FanoI}. 
We have the exact sequence 
\[
H^i(Z_{\kappa}, -(m-1)K_{Z_{\kappa}}) \to 
H^i(Z_{\kappa}, -mK_{Z_{\kappa}}) \to 
H^i(S, -mK_{Z_{\kappa}}|_S).  
\]
By the induction hypothesis (together with $H^i(Z, \MO_Z)=0$), 
it is enough to show $H^i(S, -mK_{Z_{\kappa}}|_S)=0$. 
To this end, it suffices to prove that $S$ is geometrically integral over $\kappa$ \cite[Theorem 3.4(2)]{FanoI}, 
which follows from  \cite[Proposition 2.11]{FanoI}. 
\end{proof}

We now recall some terminologies and results on rational normal scrolls. 
For details, we refer to \cite{EH87}. 

\begin{dfn}\label{d rat scroll}
For non-negative integers 
\[
0 \leq a_0 \leq a_1 \leq \cdots \leq a_d, 
\]
it is well known that $|\MO_P(1)|$ is base point free 
and $h^0(P, \MO_P(1)) = d+1+\sum_{i=0}^d a_i$
for 
\[
P := \P_{\P^1}(\MO_{\P^1}(a_0) \oplus \MO_{\P^1}(a_1) \oplus\cdots \oplus 
\MO_{\P^1}(a_d)). 
\]
Set $S_{\P^1}(a_0, a_1, ..., a_d) \subset \P^{d+\sum_{i=0}^d a_i}_k$ to be the image of $\varphi_{|\MO_P(1)|}$, 
which is called a {\em rational normal scroll}. 
\end{dfn}

\begin{rem}\label{r rat scroll}
\begin{enumerate}
\item $|\MO_P(1)|$ is very ample if and only if $a_0>0$ (i.e., all $a_0, ..., a_d$ are positive). 
\item $S_{\P^1}(0, a_1, ..., a_d)\subset \P^{d+\sum_{i=1}^d a_i}_k$ 
is the cone over 
$S_{\P^1}(a_1, ..., a_d)\subset \P^{d-1+\sum_{i=1}^d a_i}_k$ \cite[Page 6]{EH87}. 
\item $\deg S_{\P^1}(a_0, a_1, ..., a_d) = a_0+a_1+ \cdots + a_d$ \cite[Page 7]{EH87}. 
\end{enumerate}
\end{rem}

\begin{prop}\label{p birationality key}
We use Notation \ref{n g=8 1}. 
Then $\deg \theta \neq 2$. 
\end{prop}

\begin{proof}
Suppose that $\deg \theta = 2$. 
Let us derive a contradiction. 
We start by proving ($\star$) below. 
\begin{enumerate}
\item[($\star$)] 
$\overline{Z} = S_{\P^1}(0, 1, 2) \subset \P^5$. 
\end{enumerate}
Since 
$W$ is a quadric or cubic hypersurface on $\P^4$ (Notation \ref{n g=8 1}), 
$Y$ is not isomorphic to  $Y^+$ as a $k$-scheme. 
Hence $\overline Z$ is not $\Q$-factorial 
\cite[Lemma 2.2]{KTLift1}. 
Moreover, we have $\Delta(\overline Z, A_{\overline Z})=0$ (Proposition \ref{p FanoY cont}). 

It holds that  
$2 A^3_{\overline Z}  = (-K_Y)^3 = 2g-10 = 6$. 
By $1 = \rho(Z) \geq \rho(\overline Z)$, we see that $\rho(\overline Z)=1$. 
By classification of varieties of minimal degree 
(see \cite[Theorem 1]{EH87} or \cite[Theorem 5.10, Theorem 5.15]{Fuj90}), one of the following holds. 
\begin{enumerate}
\renewcommand{\labelenumi}{(\roman{enumi})}
\item $\overline Z = \P^3$ and $A^3_{\overline Z}=1$. 
\item $\overline Z \subset \P^4$ is a quadric hypersurface and $A^3_{\overline Z} =2$. 
\item $\overline Z$ is the cone over the  Veronese surface and 
$A^3_{\overline Z}=4$. 
\item $\overline Z$ is a cone over a rational normal curve. 
\item $\overline Z$ is a cone over a Hirzebruch surface. 
\end{enumerate}
In any case except for (v), $\overline Z$ is $\Q$-factorial 
(cf. \cite[Proposition A.5]{FanoI}).  
Hence (v) holds. 
We then get 
$\overline{Z} = S_{\P^1}(0, a, b) \subset \P^5$, 
which is the cone of 
$S_{\P^1}(a, b) \subset \P^4$ for $b \geq a >0$ (Remark \ref{r rat scroll}). 
It follows from Remark \ref{r rat scroll} that $a+b= A^3_{\overline Z} = 3$. 
Therefore, we have a unique solution $(a, b) =(1, 2)$. 
This completes the proof of ($\star$).

\medskip

Fix a plane $V$ contained in $\overline{Z}$, 
whose existence is guaranteed by the fact that the image in
$S_{\P^1}(1, 2)$ of a fibre of 
the $\P^1$-bundle $\P_{\P^1}( \MO_{\P^1}(1) \oplus \MO_{\P^1}(2)) \to \P^1$ is a line. 
Fix a closed point $P \in \ol{Z} \setminus V$. 
Pick a hyperplane $\wt{H} \subset \P^5$ satisfying 
$V \subset \wt{H}$ and $P \in \wt{H}$. 
Set $H := \wt{H} \cap \ol{Z}$. 
Since $\ol{Z} \subset \P^5$ is non-degenerate, 
$H$ is an effective Cartier divisor on $\ol{Z}$. 
By $V \subset H$ and $P \in H$, 
$H$ is reducible. 
In particular, $\ol{\psi}^*H$ is a reducible effective Cartier divisor on $Y$. 
As the Veronese surface $\overline \psi (E) \subset \P^5$ is non-degenerate, 
we get $\overline \psi (E) \not\subset \Supp H$, 
which implies $E \not\subset \Supp (\ol{\psi}^*H)$. 
Thus its push-forward $\sigma_*\ol{\psi}^*{H}$  on $X$ is still reducible.
This contradicts the fact that 
$\Pic X = \Z K_X$ and 
$-K_X \sim \sigma_*(-K_Y) \sim \sigma_*\ol{\psi}^*{H}$. 
\end{proof}

\begin{thm}\label{t g=8 Z 2cap3}
Let $X \subset \P^{9}$ be a prime  Fano threefold of genus $8$. 
Take a closed point $P$ on $X$ which is not contained in any line on $X$. 
Let $\sigma: Y \to X$ be the blowup at $P$. 
Then 
\begin{enumerate}
\item $|-K_Y|$ is base point free, $-K_Y$ is big, 
and $h^0(Y, -K_Y)=6$. 
\end{enumerate}
Moreover, the following hold 
for the induced morphism $\psi : Y \to Z$ 
onto the image $Z := \varphi_{|-K_Y|}(Y) \subset \P^5$ 
by $\varphi_{|-K_Y|}: Y \to \P^5$. 
\begin{enumerate}
\setcounter{enumi}{1}
\item  $\psi : Y \to Z$ is a flopping contraction, i.e., 
$\psi : Y \to Z$ is a birational morphism satisfying 
$\psi_*\MO_Y = \MO_Z$ and $\dim \Ex(\psi)=1$. 
\item $Z$ is a complete intersection $H_2 \cap H_3$, where 
each $H_d$ is a prime divisor  on $\P^5$ of degree $d$. 
\end{enumerate} 
\end{thm}

\begin{proof}
We use Notation \ref{n g=8 1}. 
By Proposition \ref{p FanoY cont}, 
it suffices to show $\deg \theta \neq 2$, 
which follows from Proposition \ref{p birationality key}. 
\end{proof}

\section{Existence of quintic genus-one curves and K3-like surfaces}

\subsection{Smoothness of the quadric $H_2$}

\begin{nota}\label{n g=8 2}
Let $X \subset \P^{9}$ be a prime  Fano threefold of genus $8$. 
Take a closed point $P$ on $X$ which is not contained in any line on $X$. 
Let $\sigma: Y \to X$ be the blowup at $P$. 
Note that $|-K_Y|$ is base point free, $-K_Y$ is big, 
and $h^0(Y, -K_Y)=6$ (Theorem \ref{t g=8 Z 2cap3}). 
For the image $Z := \varphi_{|-K_Y|}(Y) \subset \P^5$ 
by $\varphi_{|-K_Y|}: Y \to \P^5$, 
the induced morphism $\psi : Y \to Z$ is a flopping contraction and 
$Z=H_2 \cap H_3$ is a complete intersection, where 
each $H_d$ is a prime divisor on $\P^5$ of degree $d$  (Theorem \ref{t g=8 Z 2cap3}). 
Set $E := \Ex(\sigma)$ and $E_Z := \psi(E) \subset Z \subset \P^5$, 
where $E_Z \subset \P^5$  is a Veronese surface \cite[Proposition 6.6]{FanoII}. 
\end{nota}

\begin{prop}\label{p H2 smooth}
We use Notation \ref{n g=8 2}. 
Then $H_2$ is smooth and $H_3$ is normal. 
\end{prop}

\begin{proof}
Since $Z = H_2 \cap H_3$ is a complete intersection which is smooth outside finitely many closed points, we get 
$\dim(\Sing H_2) \leq 1$ and $\dim(\Sing H_3) \leq 1$. 
By Serre's criterion for normality, 
both $H_2$ and $H_3$ are normal. 

We have  $\dim(\Sing H_2) \neq 0$ (i.e., $\Sing H_2$ is not a point) by 
Lemma \ref{l quad Vero Sing neq pt}. 
Suppose $\dim(\Sing H_2) =1$, i.e., $\Sing H_2 =:L$ is a line. 
It follows from  $Z  =H_2 \cap H_3$ that 
\[
Z \cap L = Z \cap \Sing H_2 = H_3 \cap \Sing H_2 \subset \Sing (H_3 \cap H_2) = \Sing Z. 
\]
Since $\Sing Z$ is a finite set, 
we get $L \not\subset H_3$, which implies that 
$H_3 \cap L$ is a zero-dimensional scheme satisfying $\dim_k \MO_{H_3 \cap L}=3$. 
Since every curve on a Veronese surface $E_Z \subset \P^5$ is of even degree, 
$E_Z$ contains no line on $\P^5$. 
Thus either $E_Z \cap L=\emptyset$ or $E_Z \cap L$ is $0$-dimensional. 
Since a Veronese surface $E_Z \subset \P^5$ is an intersection of quadrics, 
it holds that $\dim \MO_{E_Z \cap L} \leq 2$. 
By $H_3 \cap L \subset \Sing Z \subset E_Z$ and $E_Z \subset Z \subset H_3$,  we get 
\[
H_3 \cap L = H_3 \cap L \cap E_Z = E_Z \cap L, 
\]
which leads to the following contradiction: 
\[
3 = \dim \MO_{H_3 \cap L} = \dim \MO_{E_Z \cap L} \leq 2. 
\]
\end{proof}

\subsection{Quintic quasi-elliptic curves}


\begin{lem}\label{l general irre}
Let $f: X \to Y$ be a projective surjective morphism  of normal varieties such that every fibre of $f$ is connected. 
Then general fibres of $f$ are (geometrically) irreducible. 
\end{lem}

This result should be well known, but we give a proof for the sake of completeness. 

\begin{proof}
We first reduce the problem to the case when $f_*\MO_X = \MO_Y$. 
Let 
\[
X \to Y' \xrightarrow{\theta} Y
\]
be the Stein factorisation of $f$. 
Then  the generic fibres coincide: $X_{\eta'} = X_{\eta}$. 
Since every fibre of $f$ is connected, $\theta$ is a finite purely inseparable surjective morphism. 
Then  $\eta' \to \eta$ is a universal homeomorphism, and hence 
a general fibre of $X_y$ of $f$ is homeomorphic to 
a general fibre $X_{y'}$ of $f'$. 
Hence we may assume that  $f_*\MO_X = \MO_Y$.

Let $X_{\overline{\eta}}$ be the geometric generic fibre and take the normalisation of its reduced structure: 
\[
\mu : (X_{\overline{\eta}})_{\red}^N \to (X_{\overline{\eta}})_{\red} \to X_{\overline{\eta}}. 
\]
Note that $\mu$ is a universal homeomorphism \cite[Lemma 2.2(2)]{Tan18b}. 
Then there exists a finite surjective morphism $Y' \to Y$ from a normal variety 
and 
\[
Z \to X \times_{Y} Y' \to Y'
\]
such that $Z \to X \times_Y Y'$ is a finite surjective morphism and 
the base change of $Z \to X \times_{Y} Y' \to Y'$ 
by the geometric generic point $\overline{\eta}$ is given by 
\[
(X_{\overline{\eta}})_{\red}^N \to X_{\overline{\eta}}. 
\]
After shrinking $Y$, we may assume that every fibre of $Z \to Y'$ is geometrically irreducible. 
Then a fibre of $X \times_{Y} Y' \to Y'$ is the image of the corresponding fibre of $Z \to Y'$, and hence it is irreducible. 
Thus general fibres of $X \to Y$ is irreducible, because its base change is a fibre of $X \times_{Y} Y' \to Y'$. 
\end{proof}

\begin{lem}\label{l ell linear intersection}
Let $X \subset \P^9$ be a prime Fano threefold of genus $8$. 
Let $L(X) \subset X$ be the union of all the lines on $X$, 
which is a proper closed subset of $X$. 
Let $C$ be a Gorenstein genus-one curve $C$ on $X$ such that $-K_X \cdot C=5$. 
Assume that 
\begin{enumerate}
\item $C \not\subset L(X)$, and 
\item there exists a closed point $P \in C \setminus L(X)$ such that $C_Y \cap \Ex(\psi)=\emptyset$ 
for the blowup $\sigma: Y \to X$ at $P$, the proper transform $C_Y$ of $C$ on $Y$, and the flopping contraction $\psi : Y \to Z$. 
\end{enumerate}
Then it holds that $\langle C \rangle \cap X = C$, 
where $\langle C \rangle$ denotes the smallest linear subvariety 
of $\P^9$ containing $C$. 
\end{lem}

\begin{proof}
Fix a \emph{general} closed point $P \in C \setminus L(X)$. 
We use Notation \ref{n g=8 2}. 
Let $C_Y$ be the proper transform of $C$ on $Y$ and set $C_Z := \psi(C_Y)$. 
By (2), we have $C_Y \cap \Ex(\psi)=\emptyset$ 
(consider the one-parameter family of blowups at points on $C$). 
Hence we get 
\[
C \xleftarrow{\simeq} C_{Y} \xrightarrow{\simeq} C_{Z}. 
\]
We have 
\[
-K_Z \cdot C_Z = -K_Y \cdot C_Y = (-\sigma^*K_X -2E) \cdot C_Y = 
-K_X \cdot C -2 E \cdot C_Y = 5-2 \cdot 1 = 3, 
\]
and hence 
$C_{Z}$ is a Gorenstein genus-one curve on $Z \subset \P^5$ of degree $3$ such that 
$C_{Z} \cap \psi(\Ex(\psi)) = \emptyset$.  
For the plane $V := \langle C_{Z} \rangle$ generated by $C_{Z}$, 
we have 
\[
C_{Z} \subset Z \cap V = (H_2 \cap V) \cap (H_3 \cap V). 
\]
If $H_2 \cap V \subsetneq V$, then $C_{Z}$ would be contained an effective divisor on $V = \P^2$ of degree two, which is a contradiction. 
Hence $V \subset H_2$.

We now prove $V \not\subset H_3$. 
Suppose $V \subset H_3$. Then $V \subset H_2 \cap H_3  = Z$. 
Since $Z$ is smooth around $C_{Z}$, 
$Z$ is smooth at the closed point $Q := C_{Z} \cap E_{Z}$. 
Hence 
each of  $V$ and $E_{Z}$ is an effective Cartier divisor on $Z$ around $Q$. 
Therefore, we get $\dim (V \cap E_{Z}) \geq 1$. 
In particular, we can find a curve $B \subset  V \cap E_{Z}$. 
By 
\[
B \cap C_{Z} \subset E_{Z} \cap C_{Z} = Q, 
\]
$B$ and $C_{Z}$ are curves on $V \simeq \P^2$ whose scheme-theoretic intersection 
is a reduced point $Q$. 
Then both $B$ and $C_{Z}$ must be lines, which contradicts the fact that 
$C_{Z}$ is a cubic curve on $V \simeq \P^2$. 
This completes the proof of $V \not\subset H_3$.

Then we get the scheme-theoretic equality $C_{Z} = Z \cap V$. 
In particular, $C_Z = \bigcap_{i \in I} H_{i, Z}$ for some members $H_{i, Z} \in |-K_Z| = |\MO_{\P^5}(1)|_Z|$. 
Recall that $\psi : Y \to Z$ is isomorphic around $C_Y \xrightarrow{\simeq} C_Z$. 
For $H_{i, Y} := \psi^*H_{i, Z} \in |-\psi^*K_Z| = |-K_Y|$, 
it holds that 
\[
C_Y = \psi^{-1}(C_Z)
= \psi^{-1}\left(\bigcap_{i \in I} H_{i, Y}\right)
= \bigcap_{i \in I} H_{i, Y}. 
\]
Set $H_i := \sigma_*H_{i, Y} \in | -\sigma_*K_Y| = |-K_X|$. 
We then get 
\[
C \subset \bigcap_{i \in I} H_i, 
\]
and this inclusion is an equality outside $P$. 
Recall that $P$ was chosen to be an arbitrary general closed point. 
Applying the same procedure after replacing $P$ by another general closed point $P'$, we can find a set $\{ H'_j\}_{j \in J}$ of hyperplane sections of $X$ such that 
\[
C \subset \bigcap_{j \in J} H'_j
\]
and that this inclusion is an equality outside $P'$. 
Then 
\[
C = \bigcap_{i \in I} H_i \cap \bigcap_{j \in J} H'_j, 
\]
as required. 
\end{proof}

\begin{prop}\label{p q-ell exist}
Let $X \subset \P^9$ be a prime Fano threefold of genus $8$ over $k$. 
Set $\kappa := K(\P^3_k)$. 
Then 
there exists a geometrically integral regular genus-one curve $C$ on $X_{\kappa} := X \times_k \kappa$ such that $-K_{X_{\kappa}} \cdot C=5$ and $\langle C \rangle \cap X_{\kappa} = C$. 
\end{prop}

\begin{proof}
Fix a closed point $P$ on $X$ which is not contained in any line on $X$. 
We use Notation \ref{n g=8 2}. 
Recall that $Z = H_2 \cap H_3 \subset \P^5$, where $H_2$ is a smooth quadric fourfold and $H_3$ is a normal cubic fourfold on $\P^5$ (Proposition \ref{p H2 smooth}). 
By Lemma \ref{l Veronese V+3V'}, 
there exist planes $V$ and $V'$ on $H_2$ 
such that $V^2 =V'^2=1$, $V \cdot V' =0$, 
$\CH^2(H_2) = \Z V \oplus \Z V',$ and 
\[
E_Z = V + 3V' \qquad \text{in} \qquad \CH^2(H_2). 
\]
By $H_2 \simeq \Gr(2, 4)$, 
$V$ (resp. $V'$) is a member of 
the deformation family parametrised by $\Gr(1, 4) \simeq \P^3$ 
(resp. $\Gr(3, 4) \simeq \P^3$). 



For a field $K$ and a morphism $\alpha : \Spec K \to \P^3_k$, consider the following commutative diagram 
\begin{equation}\label{e1 q-ell exist}
\begin{tikzcd}
Z \arrow[d, hook] & \mathcal C := \pi_1^{-1}(Z) \arrow[d, hook] \arrow[l, "\pi_1"', "\text{$\P^1$-bundle}"] & C_{Z, \alpha} \arrow[l] \arrow[d, hook]\\
H_2  & 
\mathcal V \arrow[d, "\pi_2"', "\text{$\P^2$-bundle}"]
\arrow[l, "\pi_1"', "\text{$\P^1$-bundle}"] &
V_{\alpha} \simeq \P^2_K \arrow[l] \arrow[d] \\
& \P^3_k & \Spec K \arrow[l, "\alpha"'], 
\end{tikzcd}
\end{equation}
where $\pi_2: \mathcal V \hookrightarrow H_2 \times_k \P^3_k \xrightarrow{\pr_2} \P^3_k$ is the flat family of planes on $H_2$ containing $[V]$, 
each square is cartesian, and all the upper vertical arrows are closed immersions. 
In particular, 
we have a closed embedding  
$\mathcal C \subset Z \times_k \P^3_k$ 
and $C_{Z, \alpha}$ is naturally a closed subscheme of $Z \times_k K$. 


\setcounter{step}{0}

\begin{claim*}
Assume that     $K=k$ and $\alpha : \Spec k \to \P^3_k$ is a general closed point. 
Then the following hold. 
\begin{enumerate}
\item $C_{Z, \alpha}$ is a plane cubic curve on $V_{\alpha} \simeq \P^2_k$ (in particular, $C_{Z, \alpha}$ is an integral scheme). 
\item The scheme-theoretic intersection $C_{Z, \alpha} \cap E_Z$ is a reduced point. 
\item $C_{Z, \alpha} \cap \Sing Z = \emptyset$.  
\end{enumerate}
\end{claim*}

\begin{proof}[Proof of Claim]
Replacing $V$ by $V_{\alpha}$, we may assume $V = V_{\alpha}$. 
In this case, we have  
\[
C_{Z, \alpha} = Z \cap V_{\alpha} 
= H_3 \cap V_{\alpha} = H_3 \cap V. 
\]
It holds that 
\[
E_Z \cdot V = (V+3V') \cdot V = V^2 =1,
\]
i.e., the scheme-theoretic intersection $Q := E_Z \cap V$ is a  reduced point. Then 
\[
C_{Z, \alpha} \cap E_Z =  H_3 \cap V \cap E_Z  \overset{(\star)}{=} V \cap E_Z = Q, 
\]
where $(\star)$ follows from $E_Z  \subset Z \subset H_3$. 
Thus (2) holds. 
Since $V = V_{\alpha}$ is a general fibre of $\pi_2 : \mathcal V \to \P^3_k$, 
$V$ can avoid finitely many fixed closed points on $H_2$ 
(cf. (\ref{e1 q-ell exist})), 
and hence we may assume that $C_{Z, \alpha} \cap \Sing Z = \emptyset$. 
Thus (3)  holds.

Let us show (1). 
Since $\pi_1 : \mathcal V \to H_2$ is surjective and 
$\alpha : \Spec k \to \P^3_k$ is a general closed point (so $V= V_{\alpha}$ is a general fibre of $\pi_1$), 
we have that $V \not\subset H_3$, which implies that $C_{Z, \alpha} = H_3 \cap V = H_3|_V$ is an effective Cartier divisor on the plane $V =V_{\alpha} = \P^2_k$ of degree $3$. 
In particular, a general fibre $C_{Z, \alpha}$ of $\mathcal C \to \P^3_k$ (over a general closed point) is connected.
Note that $\P^3_k$ is smooth and that $\mathcal C$ is smooth outside the $\pi_1$-fibres over $\Sing \, Z$. It follows from Lemma \ref{l general irre} 
that $C_{Z, \alpha}$  is irreducible. 

Therefore, it is enough to show that $C_{Z, \alpha}$ is reduced. 
Since $C_{Z, \alpha}$ is an effective Cartier divisor on the plane $V = \P^2$, 
$C_{Z, \alpha}$ is Cohen-Macaulay. 
Thus $C_{Z, \alpha}$ has no embedded primes. 
Hence 
the zero-dimensional closed subscheme
$E_Z \cap C_{Z, \alpha} =E_Z|_{C_{Z, \alpha}} (=Q)$ on $C_{Z, \alpha}$, 
whose defining equation is locally principal, is an effective Cartier divisor on  $C_{Z, \alpha}$. 
As $E_Z|_{C_{Z, \alpha}}$ is smooth, so is $C_{Z, \alpha}$ around $E_Z|_{C_{Z, \alpha}} = Q$. 
Hence $C_Z$ satisfies $R_0$ and $S_1$ (as it is Cohen-Macaulay), i.e., $C_Z$ is reduced. 
Thus (1) holds. 
This completes the proof of Claim. 
\end{proof}

Let $\xi : \Spec \kappa \to \P^3_k$ be the generic point. 
Then  $C_Z := C_{Z, \xi}$ is the generic fibre of $\mathcal C = \pi^{-1}_1(Z) \to \P^3_k$, 
which is a closed subscheme on $Z_{\kappa} := Z \times_k \kappa$. 
By Claim, the following statements (1)'-(3)'  corresponding to (1)-(3) hold for the generic fibre $C_Z = C_{Z, \xi}$. 
\begin{enumerate}
\item[(1)'] $C_Z$ is a geometrically integral regular genus-one curve. 
\item[(2)']  The scheme-theoretic intersection $Q := C_Z \cap E_{Z, \kappa}$ 
is a reduced $\kappa$-rational point. 
\item[(3)']  $C_Z \cap \Sing {Z}_{\kappa} = \emptyset$, 
where $\Sing {Z}_{\kappa}$ denotes the non-smooth locus of $Z_{\kappa}$ (which coincides with the non-regular locus of $Z_{\kappa}$). 
\end{enumerate}

Set $C_Y := \psi^{-1}(C_Z) \subset Y_{\kappa}$ and $C := \sigma(C_Y) \subset X_{\kappa}$. 
By $C_Z \cap \Sing Z_{\kappa} = \emptyset$, we have $C_Y \xrightarrow{\psi|_{C_Y}, \simeq} C_Z$ and 
\[
C_Y \cap E_{\kappa} \xrightarrow{\simeq} C_Z \cap E_{Z, \kappa} = Q. 
\]
Hence we get $P_{\kappa} \in C$ (recall that $P_{\kappa}$ is a $\kappa$-rational point which is the blowup centre of $\sigma: Y_{\kappa} \to X_{\kappa}$) and 
$C$ is smooth at $P_{\kappa}$. Therefore, we get 
\[
C \xleftarrow{\simeq, \sigma|_{C_Y}} C_Y \xrightarrow{\simeq, \psi|_{C_Y}} C_Z. 
\]
By (1)',  $C$ is a geometrically integral regular genus-one curve. 
It holds that 
\[
-K_{X_{\kappa}} \cdot C = (-K_{Y_{\kappa}} +2E_{\kappa}) \cdot C_Y = 
(-K_{Z_{\kappa}} +2E_{Z, \kappa}) \cdot C_Z = 3 + 2 \cdot 1 =5. 
\]

It is enough to show $\langle C \rangle \cap X_{\kappa} = C$. 
Let $\overline{\kappa}$ be the algebraic closure of $\kappa$. 
By the flat base change theorem: 
\[
H^0(\P^9_{\kappa}, I_C \otimes \MO_{\P^9}(1)) \otimes_{\kappa} \overline{\kappa} 
\simeq 
H^0(\P^9_{\ol{\kappa}}, I_{C_{\ol{\kappa}}} \otimes \MO_{\P^9}(1)), 
\]
it suffices to prove the corresponding statement 
$\langle C_{\ol{\kappa}} \rangle \cap X_{\overline{\kappa}} = C_{\ol{\kappa}}$ 
after taking the base change $(-) \times_{\kappa} \overline{\kappa}$, 
which follows from $C_Y \cap \Ex(\psi) = \emptyset$ and Lemma \ref{l ell linear intersection} 
(applied for a prime Fano threefold $X_{\ol{\kappa}} \subset \P^9_{\ol{\kappa}}$ 
of genus $8$ over $\overline{\kappa}$ and a Gorenstein genus-one curve $C_{\ol{\kappa}}$ on $X_{\ol{\kappa}}$). 
\qedhere

\end{proof}

\subsection{Regular K3-like surfaces}


\begin{prop}\label{p exist K3-like}
Let $X \subset \P^9$ be a prime Fano threefold of genus $8$ over $k$. 
Take a field extension $k \subset \kappa$. 
Let $C$ be a geometrically integral regular genus-one  curve $C$ on $X_{\kappa}$ 
such that $-K_{X_{\kappa}} \cdot C =5$ and $C = X \cap \langle C\rangle$, 
where $\langle C\rangle$ denotes the smallest linear subvariety of $\P^9$ containing $C$. 
Then there exist a purely transcendental field extension $\kappa'/\kappa$ 
of finite transcendental degree and 
a regular prime divisor $S$ on $X_{\kappa'}$ 
such that $S \sim -K_{X_{\kappa'}}$ and  $C_{\kappa'} \subset S$, 
where $X_{\kappa'} := X \times_k \kappa'$ and $C_{\kappa'} := C \times_{\kappa} \kappa'$. 
Moreover, $S$ is a geometrically integral regular K3-like surface. 
\end{prop}

\begin{proof}
Let $\sigma: Y \to X_{\kappa}$ be the blowup along $C$. 
We have the following (cf.\ \cite[Lemma 3.21(2)]{FanoII}): 
\begin{enumerate}
\item 
$(-K_Y)^3 = (-K_{X_{\kappa}})^3 -2 (-K_{X_{\kappa}}) \cdot C +2p_a(C) -2 = 14 - 10 =4$. 
\item 
$(-K_Y)^2 \cdot E = (-K_{X_{\kappa}}) \cdot C -2p_a(C) +2 = 5$. 
\item 
$(-K_Y) \cdot E^2 =2p_a(C) -2 =0$. 
\item 
$E^3 = -\deg N_{C/X_{\kappa}} = 
-(-K_{X_{\kappa}}) \cdot C -2p_a(C) +2 =-5$. 
\end{enumerate}
Then $|-K_Y|$ is base point free, and hence its generic member $T$ on $Y_{\kappa'}$ is regular, where $\kappa' := K(\P(H^0(Y, -K_Y)))$ 
\cite[Theorem 4.9, Definition 5.6, Remark 5.8]{Tana}. 
By $(-K_{Y_{\kappa'}})^3 = (-K_Y)^3 >0$, $T$ is big, and hence $T$ is connected. 
Thus $T$ is a regular prime divisor. 
Set $S := \sigma_*T \in |-\sigma_*K_{Y_{\kappa'}}| = |-K_{X_{\kappa'}}|$, which is a prime divisor on $X_{\kappa'}$.

Let us show that $S$ is normal. 
Consider the following exact sequence 
\[
0 \to \sigma_* \MO_{Y_{\kappa'}}(-T) \to \sigma_*\MO_{Y_{\kappa'}} \to 
\sigma_*\MO_T \to R^1\sigma_*\MO_{Y_{\kappa'}}(-T). 
\]
Comparing this with the corresponding exact sequence on $X_{\kappa'}$, 
it suffices to show that 
$R^1\sigma_*\MO_{Y_{\kappa'}}(-T)=0$ (as this implies that $\MO_S \to \sigma_*\MO_T$ is surjective, and hence bijective; which implies that $S$ is normal). 
This follows from 
\[
R^1\sigma_*\MO_{Y_{\kappa'}}(-T) \simeq 
R^1\sigma_*\MO_{Y_{\kappa'}}(K_{Y_{\kappa'}}) =0, 
\]
where the vanishing holds by \cite[Corollary 6.8]{Tan-elliptic}.

Suppose that $S :=\sigma_*T$ is not regular. 
By $K_S \sim 0$ and $K_T \sim 0$, we get $\tau := \sigma|_T : T \to S$ is a resolution of singularities satisfying $K_T \sim\tau^*K_S$. 
Then $S$ has at worst canonical singularities.  
Since $S$ is not regular, $T$ contains a one-dimensional fibre of $\sigma : Y_{\kappa'} \to X_{\kappa'}$. 
In particular, $T|_{E_{\kappa'}}$ is an effective Cartier divisor on the $\P^1$-bundle $E_{\kappa'}$ over the regular projective curve $C_{\kappa'}$, 
and $T|_{E_{\kappa'}}$ contains a fibre of 
\[
\pi:= \sigma|_{E_{\kappa'}} : E_{\kappa'} \to C_{\kappa'}
\]
Since $E$ is regular and $T$ is the generic member on $Y$ of $|-K_Y|$, 
$T|_{E_{\kappa}}$ is the generic member of some base point free linear system, 
and hence $T|_{E_{\kappa}}$ is regular \cite[Remark 4.5 and Theorem 4.9]{Tana}.

On the other hand, it holds that 
\[
(T|_{E_{\kappa'}}) \cdot (\sigma^*(-K_{X_{\kappa'}})|_{E_{\kappa'}}) 
= -K_{Y_{\kappa'}} \cdot E_{\kappa'} \cdot \sigma^*(-K_{X_{\kappa'}}) = 
(-E) \cdot E \cdot \sigma^*(-K_X) >0. 
\]
Hence the effective divisor $T|_{E_{\kappa'}}$ on $E_{\kappa'}$ 
dominates the base curve $C_{\kappa'}$. 
Therefore, $T|_{E_{\kappa'}}$ satisfies the following three conditions, 
which is absurd.  
\begin{itemize}
\item $T|_{E_{\kappa'}}$ contains a $\pi$-horizontal irreducible component. 
\item $T|_{E_{\kappa'}}$ contains a $\pi$-fibre. 
\item $T|_{E_{\kappa'}}$ is regular. 
\end{itemize}

Since $S$ is a regular K3-like surface, it is enough to show that $S$ is geometrically integral. 
For the algebraic closure $\overline{\kappa'}$ of $\kappa'$, 
the linear equivalence $-K_{X_{\kappa'}} \sim S$ implies 
$-K_{X_{\overline{\kappa'}}} \sim S \times_{\kappa'} \overline{\kappa'}$. 
By $\Pic X_{\overline{\kappa'}} = \Z K_{X_{\overline{\kappa'}}}$, 
$S \times_{\kappa'} \overline{\kappa'}$ is still a prime divisor on 
$X_{\overline{\kappa'}}$, and hence $S$ is geometrically integral. 
\qedhere



\end{proof}

\begin{lem}\label{l ell on K3 bpf}
We work over a field $\kappa$. 
Let $S$ be a geometrically integral regular K3-like surface over $\kappa$ and 
let $C$ be a geometrically integral regular genus-one curve on $S$. 
Then $C^2=0$, $|C|$ is base point free, and $\dim_{\kappa}H^0(S, C)=2$. 
\end{lem}

\begin{proof}
Since $S$ and $C$ are geometrically integral, 
we have $h^0(S, \MO_S) = h^0(C, \MO_C)=1$. 
The adjunction formula implies $0 = 2g(C) -2 = (K_S+C) \cdot C = C^2$. 
By Serre duality: $h^2(S, C) = h^0(S, -C)=0$,  we  have 
\[
h^0(S, C) \geq 
h^0(S, C) - h^1(S, C) = \chi(S, C) = \chi(S, \MO_S) + \frac{1}{2} C^2 = 2. 
\]
Hence there is an effective divisor $D$ such that $D \neq C$ and $D \sim C$. 
By $C^2=0$, we get $D \cap C= \emptyset$. Thus $|C|$ is base point free. 
Then the exact sequence 
\[
0 \to \MO_S \to \MO_S(C) \to \MO_S(C)|_C \to 0, 
\]
together with $H^1(S, \MO_S)=0$ and $h^0(C, \MO_S(C)|_C)= h^0(C, \MO_C)=1$, 
implies
\[
h^0(S, \MO_S(C)) =h^0(S, \MO_S) + h^0(C, \MO_C(C))=2.
\]
\end{proof}

 \section{Linear section theorem}\label{s linear section}

\subsection{Construction of vector bundles}\label{ss construction vb}


\begin{lem}\label{l construct E via SC0}
Let $X \subset \P^{9}_k$ be a prime  Fano threefold of genus $8$ over $k$. 
Take a field extension $k \subset \kappa$ 
and let $C$ be a geometrically integral regular genus-one curve 
on $X_{\kappa} := X \times_k \kappa$  satisfying $\langle C \rangle  \cap X_{\kappa} = C$. 
Then there exists a prime divisor $S$ on $X_{\kappa}$ such that 
the following hold for $\MO_S(S-C) := (\MO_X(S)|_S) \otimes_{\MO_S} \MO_S(-C)$: 
\begin{enumerate}
\item $S \sim -K_{X_\kappa}$ and $S$ is a geometrically integral projective Gorenstein surface with $\omega_S \simeq \MO_S$. 
Moreover, 
$C \subset S$ and $C$ is an effective Cartier divisor on $S$. 
\item $\MO_S(C)$ and $\MO_S(S-C)$ is globally generated. 
\item $h^0(S, \MO_S(C))=2$ and $h^0(S, \MO_S(S-C))=4$. 
\end{enumerate}
\end{lem}

\begin{proof}
We use the same notation as in the proof of Proposition \ref{p exist K3-like} except for $S$, which we replace by $S'$. 
In particular, $\sigma : Y \to X_{\kappa}$ is the blowup along $C$, 
$T$ is the generic member of the base point free complete linear system $|-K_Y|$, and $S' := \sigma_*T \in |-\sigma_*K_{Y_{\kappa'}}| = |-K_{X_{\kappa'}}|$ is a prime divisor on $X_{\kappa'} := X_{\kappa} \times_\kappa \kappa'$, where $\kappa':=K(\P(H^0(Y, -K_Y)))$ and $Y_{\kappa'} := Y \times_\kappa \kappa'$. 
Note that $\sigma : Y_{\kappa'} \to X_{\kappa'}$ is the blowup along 
the regular curve $C_{\kappa'} := C \times_{\kappa} \kappa'$. 

We now show (1)'--(3)' below. 
\begin{enumerate}
\item[(1)'] 
$S' \sim -K_{X_{\kappa'}}$ and $S'$ is a geometrically integral 
regular projective surface such that $\omega_S \simeq \MO_S$ 
and $C_{\kappa'} \subset S'$. 
\item[(2)']  $\MO_{S'}(C_{\kappa'})$ and $\MO_{S'}(S'-C_{\kappa'})$ is globally generated. 
\item[(3)']  $h^0(S', \MO_{S'}(C_{\kappa'}))=2$ and $h^0(S', \MO_{S'}(S'-C_{\kappa'}))=4$. 
\end{enumerate}
By the conclusion  of Proposition \ref{p exist K3-like}, 
it holds that 
$S' \sim -K_{X_{\kappa'}}$, 
$C_{\kappa'} \subset S'$, and $S'$ is a geometrically integral regular K3-like surface. 
In particular, (1)' holds. 
It follows from Lemma \ref{l ell on K3 bpf} that 
$h^0(S', \MO_{S'}(C))=2$ and 
$\MO_{S'}(C)$ is globally generated.

By construction, 
we have an isomorphism $\sigma|_{T} : T \xrightarrow{\simeq} S'$. 
Via this isomorphism, we have 
\[
|(-K_{X_{\kappa'}}|_{S'}) -C| = |(-\sigma^*K_{X_{\kappa'}} -E)|_{T}| = |-K_{Y_{\kappa'}}|_{T}|
\]
Since $|-K_Y|$ is base point free, so is $|-K_Y|_{T}|$. 
Hence (2)' holds. 
By 
\[
(-K_Y|_T)^2 = (-K_Y)^3 =4 >0,
\]
$-K_Y|_T$ is nef and big. 
Since $T( \simeq S)$ is a geometrically integral regular K3-like surface, we have 
$H^i(T, -K_Y|_{T}) =0$ for every $i>0$ \cite[Theorem 3.4(2)]{FanoI}. 
We then get  
\[
 h^0(S', \MO_{S'}(S'-C)) = h^0(T, -K_Y|_{T}) = \chi(T, -K_Y|_{T}) = 2 + (-K_Y|_{T})^2/2 = 2 + 2 =4. 
\]
Thus (3)' holds. This completes the proof of (1)'--(3)'. 

\medskip

Recall that $\Spec \kappa'$ is the generic point of $\P(H^0(Y, -K_Y))$. 
By expanding $S'\to \Spec \kappa'$ constructed above over a non-empty open subset 
$U \subset \P(H^0(Y, -K_Y))$, we can find a closed subscheme 
$\mathcal S \subset X_{\kappa} \times_{\kappa} U$ 
whose base change by $(-) \times_U \Spec \kappa'$ coincides with $S' \subset X_{\kappa'}$. 
Take a general $\kappa$-rational point $P$ of $U$ and 
we define $S$ as the fibre of  $\mathcal S \to U$ over $P$. 
Then the properties (1)--(3) hold by the corresponding properties (1)'--(3)' 
for the generic fibre $S'$. 
\end{proof}

\begin{prop}\label{p construct E via SC}
Let $X \subset \P^{9}_k$ be a prime  Fano threefold of genus $8$ over $k$. 
Take a field extension $k \subset \kappa$ 
and let $C$ be a 
a geometrically integral regular genus-one curve 
on $X_{\kappa} := X \times_k \kappa$  satisfying $\langle C \rangle  \cap X_{\kappa} = C$. 
Then there exists a locally free sheaf $E$ on $X_{\kappa}$ such that 
\begin{enumerate}
\item $E$ is globally generated, 
\item $\dim_{\kappa} H^0(X_{\kappa}, E) = 6$, 
\item $\bigwedge^2 E \simeq \omega_{X_\kappa}^{-1}$, and 
\item $H^0(X_{\kappa}, E \otimes \MO_{X_{\kappa}}(K_{X_{\kappa}})) = 
H^1(X_{\kappa}, E \otimes \MO_{X_{\kappa}}(K_{X_{\kappa}})) = 0$. 
\end{enumerate}
\end{prop}

The following argument is based on \cite[Section 5]{IP99}. 


\begin{proof}
Take a prime divisor $S$ on $X_{\kappa}$ as in Lemma \ref{l construct E via SC0}. 
Set $V := H^0(S, \MO_S(C))$. We have $\dim_{\kappa}  V=2$ (Lemma \ref{l construct E via SC0}(3)). 

We now show that there exists an exact sequence 
\begin{equation}\label{e1 construct E via SC}
0 \to \MO_S(-C) \to V \otimes_{\kappa} \MO_S \to \MO_S(C) \to 0. 
\end{equation}
Since $\MO_S(C)$ is globally generated (Lemma \ref{l construct E via SC0}(2)), 
the induced $\MO_S$-module homomorphism 
$V  \otimes_{\kappa} \MO_S \to \MO_S(C)$ is surjective. 
Let $K$ be its kernel, which completes the exact sequence 
$0 \to K \to V \otimes_{\kappa} \MO_S \to \MO_S(C) \to 0$. 
It suffices to show $K \simeq \MO_S(-C)$. 
Since both $V \otimes_{\kappa} \MO_S$ and $\MO_S(C)$ are locally free, 
so is $K$. 
Applying the wedge product to $0 \to K \to V \otimes_{\kappa} \MO_S \to \MO_S(C) \to 0$, 
we get $K \otimes \MO_S(C) \simeq \bigwedge^2 (V \otimes_{\kappa} \MO_S) \simeq \MO_S$, 
which implies $K \simeq \MO_S(-C)$. 
This completes the proof of the existence of (\ref{e1 construct E via SC}).


Set $\MO_S(S) := \MO_{X_\kappa}(S)|_S$. 
We then get the following homomorphism between exact sequences: 
\begin{equation}\label{e2 construct E via SC}
\begin{tikzcd}
0 \arrow[r]  &
E \arrow[r] \arrow[d] & 
V \otimes_\kappa \MO_{X_{\kappa}}(S) \arrow[d] \arrow[r] &
\MO_S(S+C) \arrow[d, equal] \arrow[r] & 0\\
0 \arrow[r]  &
\MO_S(S-C) \arrow[r] &
V \otimes_\kappa \MO_S(S) \arrow[r] &
\MO_S(S+C) \arrow[r]  &0,
\end{tikzcd}
\end{equation}
where 
the lower horizontal sequence is obtained by applying 
$(-) \otimes \MO_S(S)$ to (\ref{e1 construct E via SC}), 
the right square consists of the natural arrows, 
$E$ is defined as the kernel, and the left vertical arrow is the induced one. 
Then  $E$ is a locally free sheaf on $X_\kappa$ of rank $2$ \cite[Lemma 16 in Page 41]{Fri98}. 
Applying the snake lemma to (\ref{e2 construct E via SC}), we get an exact sequence 
\begin{equation}\label{e3 construct E via SC}
0 \to V \otimes_\kappa \MO_{X_\kappa} \to E \to \MO_S(S-C) \to 0. 
\end{equation}


Let us show (1). 
Taking the homomorphisms obtained from the global sections to (\ref{e3 construct E via SC}), we get the following 
commutative diagram in which each horizontal sequence is exact (here we used $H^1(X_\kappa, \MO_{X_\kappa}) = H^1(X, \MO_X) \otimes_k \kappa=0$):  
\[
\begin{tikzcd}
0 \arrow[r]  &
V \otimes_\kappa \MO_{X_\kappa} \arrow[r] \arrow[d, equal] & 
H^0(E) \otimes_\kappa \MO_{X_\kappa} \arrow[d, "\alpha"] \arrow[r] &
H^0(\MO_S(S-C)) \otimes_\kappa \MO_{X_\kappa} \arrow[d, "\beta"] \arrow[r] & 0\\
0 \arrow[r]  &
V \otimes_\kappa \MO_{X_\kappa} \arrow[r] &
E \arrow[r] &
\MO_S(S-C) \arrow[r]  &0.
\end{tikzcd}
\]
By the snake lemma, 
$\alpha$ is surjective if and only if $\beta$ is surjective. 
Moreover, $\beta$ is surjective if and only if 
the induced homomorphism 
\[
\beta' : H^0(\MO_S(S-C)) \otimes_\kappa \MO_S \to 
\MO_S(S-C) \]
is surjective, because $\MO_{X_\kappa} \to \MO_S$ is surjective. 
By Lemma \ref{l construct E via SC0}(2), $\beta'$ is surjective, and hence $\alpha$ is surjective. 
Thus (1) holds. 

Recall that $h^0(S, \MO_S(S-C))=4$ (Lemma \ref{l construct E via SC0}(3)).
Then it follows from 
 (\ref{e3 construct E via SC}) and $H^1(X_\kappa, \MO_{X_\kappa})=0$ 
that 
\[
h^0(X_\kappa, E) = h^0(X_\kappa,  V \otimes_\kappa \MO_{X_\kappa}) + h^0(S, \MO_S(S-C)) = 2 + 4=6. 
\]
Thus (2) holds. By \cite[Lemma 16 in Page 41]{Fri98}, we get 
\[
c_1(E) = c_1(V \otimes_\kappa \MO_{X_\kappa}(S) ) -c_1(\MO_{X_\kappa}(S))=2S-S = -K_{X_\kappa}. 
\]
Thus (3) holds. 

Let us show (4). 
By (\ref{e3 construct E via SC}),  we have the following  exact sequence: 
\[
0 \to V \otimes_\kappa \MO_{X_\kappa}(K_{X_\kappa}) \to E \otimes \MO_{X_\kappa}(K_{X_\kappa}) \to \MO_S(-C) \to 0.  
\]
Then the assertion follows from 
$H^0(\MO_{X_\kappa}(K_{X_\kappa})) = H^1(\MO_{X_\kappa}(K_{X_\kappa})) = H^0(S, \MO_S(-C)) = H^1(S, \MO_S(-C))=0$. 
Here 
$H^1(S, \MO_S(-C))=0$ holds by 
$H^1(S, \MO_S)=0$ and the exact sequence 
\[
0 \to \MO_S(-C) \to \MO_S \to \MO_C \to 0. 
\]
Note that we have the induced isomorphism $H^0(S, \MO_S) \xrightarrow{\simeq} H^0(C, \MO_C)$, because 
both $S$ and $C$ are geometrically integral (Lemma \ref{l construct E via SC0}(1)). 
Thus (4) holds. 
\qedhere



\end{proof}

\begin{rem}\label{r construct E via SC} 
We use the same notation as in the proof of Proposition \ref{p construct E via SC}. 
We have an injection 
\[
\iota: H^0(S, \MO_S(C)) \hookrightarrow H^0(X_\kappa, E). 
\]
For the element $s \in H^0(S, \MO_S(C))$ defining $C$ (which is unique up to a scalar), 
the zero locus $Z(\iota(s))$ of its image $\iota(s) \in H^0(X, E)$ is scheme-theoretically equal to $C$. 
\end{rem}

\begin{prop}\label{p construct E only X}
Let $X \subset \P^9_k$ be a prime Fano threefold of genus $8$ over $k$. 
Then there exists 
a locally free sheaf $E$ on $X$ of rank $2$ which satisfies {\rm (1)-(3)}. 
\begin{enumerate}
\item $\bigwedge^2 E \simeq \omega_X^{-1}$. 
\item $E$ is globally generated. 
\item $h^0(X, E) = 6$ and $H^0(X, E \otimes \MO_X(K_X)) = 
H^1(X, E \otimes \MO_X(K_X)) = 0$. 
\end{enumerate}
\end{prop}

\begin{proof}
By 
Proposition \ref{p q-ell exist}, 
there exist a field extension $\kappa/k$ and a curve $C$ on $X_{\kappa} :=X \times_k \kappa$ 
which satisfy all the assumptions in Proposition \ref{p construct E via SC}. 
Then there exists a locally free sheaf $E'$ of rank $2$ on $X_{\kappa} := X \times_k \kappa$ 
such that 
\begin{enumerate}
\item[(1)'] $\bigwedge^2 E' \simeq \omega_{X_{\kappa}}^{-1}$,  
\item[(2)'] $E'$ is globally generated, 
\item[(3)'] $\dim_{\kappa} H^0(X_{\kappa}, E') = 6$, and $H^0(X_{\kappa}, E' \otimes \MO_{X_{\kappa}}(K_{X_{\kappa}})) = 
H^1(X_{\kappa}, E' \otimes \MO_{X_{\kappa}}(K_{X_{\kappa}})) =  0$. 
\end{enumerate}
Then we can find an intermediate ring $k \subset R \subset \kappa$ such that 
$R$ is a finitely generated $k$-algebra and $E'$ is defined over $R$, 
i.e., there exists a locally free sheaf $\mathcal E$ on $X_R := X \times_k R$ 
satisfying $E' \simeq \rho^*\mathcal E$ for the induced morphism $\rho : X_{\kappa} \to X_{R}$. 
Let $\alpha : X_R \to \Spec R$ be the induced morphism. 
After enlarging $R$ if necessary, we may assume that 
\begin{enumerate}
\item[(1)''] $\bigwedge^2 \mathcal E \simeq \omega_{X_R/R}^{-1}$, and 
\item[(2)''] $\mathcal E$ is globally generated. 
\end{enumerate}
For a general closed point $P \in \Spec R$, 
we set $E := \mathcal E|_{X \times_k \{P\}}$, which can be considered as a locally free sheaf on $X$ via the isomorphism $\pr_1 : X \times_k \{P\} \xrightarrow{\simeq} X$. 
Then (1)'' and (2)'' imply (1) and (2), respectively. 
The assertion (3) follows from (3)', $E' \simeq \rho^*\mathcal E$, 
and the upper semicontinuity \cite[Ch. III, Theorem 12.8]{Har77} (note that $X_{\kappa}$ is the base change of the generic fibre of $X_R \to \Spec R$ by the induced field extension $\Frac(R) \hookrightarrow \kappa$). 
\end{proof}

\subsection{Proof of linear section theorem}


\begin{nota}\label{n g=8 3} 
Let $X \subset \P^{9}$ be a prime  Fano threefold of genus $8$ 
and let $E$ be a locally free sheaf on $X$ of rank $2$ 
such that (1)-(4) hold. 
\begin{enumerate}
\item $\bigwedge^2 E \simeq \omega_X^{-1}$. 
\item $E$ is globally generated. 
\item $h^0(X, E) = 6$. 
\item $H^0(X, E \otimes \MO_X(K_X)) = H^1(X, E \otimes \MO_X(K_X)) =  0$. 
\end{enumerate}
Set $U_6 := H^0(X, E)$. 
Let $T$ be a general member of $|-K_X|$ (which is a smooth K3 surface). 
For $E_T := E|_T$, we have the following natural evaluation maps 
\begin{align*}
\lambda_X \colon \wedge^2 U_6 = \wedge^2 H^0(X,E)  
&\to H^0(X,\wedge^2 E) \simeq H^0(X,-K_X)\\
\lambda_T \colon \wedge^2 U_6 = \wedge^2 H^0(X,E) 
\xrightarrow{(\star), \simeq} \wedge^2 H^0(T,E_T)  &\to H^0(T,\wedge^2 E_T) \simeq H^0(T, 
\MO_X(-K_X)|_T), 
\end{align*}
where $(\star)$ is induced by the restriction map 
$H^0(X,E) \to H^0(T, E_T)$, which is an isomorphism by (4). 
By (2)-(4), 
there exists a morphism 
\[
\varphi : X \to \Gr(U_6,2) 
\]
such that $E \simeq \varphi^*\mathcal Q$, where $q_{\univ}: U_6 \otimes \MO_{\Gr(U_6,2)} \to \mathcal Q$ denotes the universal quotient bundle on $\Gr(U_6, 2)$. 
We have the following commutative diagrams: 
\[
\begin{tikzcd}
 &\Gr(U_6,2) \arrow[r, hook, "\text{Pl\"{u}cker}"]  & \P(\wedge^2 U_6) (\simeq \P^{14})\\
X \arrow[r] \arrow[ru, "\varphi"] & \varphi(X) \arrow[u, hook] \arrow[r, hook] & \langle\varphi(X)\rangle = \P(\Im \lambda_X) \arrow[u, hook] 
\end{tikzcd}
\]
and
\[
\begin{tikzcd}
\Gr(2,U_6) \arrow[r, hook, "\text{Pl\"{u}cker}"]  & \P(\wedge^2 U_6^\vee) (\simeq \P^{14})\\
\Gr(2,U_6)\cap \P( (\Ker \lambda_X)^{\vee}) \arrow[u, hook] \arrow[r, hook]  & \P((\Ker \lambda_X)^{\vee}). \arrow[u, hook] 
\end{tikzcd}
\]

\end{nota}



\begin{dfn}[Schubert divisor]\label{d Schebert div}
We use Notation \ref{n g=8 3}.  
Via the restriction isomorphism 
\[
H^0(\P(\wedge^2 U_6), \MO(1)) \xrightarrow{\simeq} 
H^0(\Gr(U_6,2), \MO(1)|_{\Gr(U_6,2)}),
\]
 $\P(\wedge^2 U_6 ^\vee)$ parametrises the hyperplane sections of $\Gr(U_6, 2)$.
 Let $[W] \in \Gr(2,U_6) (\subset \P(\wedge^2 U_6 ^\vee))$ be a closed point, i.e., $W \subset U_6$ is a two-dimensional $k$-vector subspace. 
 The hyperplane section $H_{W}$ of $\Gr(U_6, 2)$ corresponding to $[W]$ is called the {\em Schubert divisor}.
\end{dfn}

\begin{rem}\label{r Schebert div}
We use Notation \ref{n g=8 3}.  
Take a  two-dimensional $k$-vector subspace $W \subset U_6$. 
We have the following  $\MO_{\Gr(U_6,2)}$-module homomorphisms: 
\[
W \otimes_k \MO_{\Gr(U_6,2)} \hookrightarrow U_6 \otimes_k \MO_{ \Gr(U_6,2)} \xrightarrow{q_{\univ}} \cQ. 
\]
By taking the exterior power $\wedge^2$, 
we get the induced 
$\MO_{\Gr(U_6,2)}$-module homomorphism 
$\wedge^2 W \otimes_k \MO_{\Gr(U_6,2)} \to \wedge^2 \cQ$.
Then $H_{W}$ is the zero set of this map, i.e., 
$H_W$ is equal to the member of 
$|\MO_{\Gr(U_6,2)}(1)|$ corresponding to the one-dimensional image of 
\[
\wedge^2 W = H^0(\Gr(U_6,2), \wedge^2 W \otimes \MO_{\Gr(U_6,2)}) 
\to 
\]
\[
H^0(\Gr(U_6,2), \wedge^2 \cQ) = H^0(\Gr(U_6,2), \MO_{\Gr(U_6,2)}(1)) =\wedge^2 U_6. 
\]
\end{rem}

\begin{lem}\label{l Gr P(Ker) empty}
We use Notation \ref{n g=8 3}.  
Then it holds that 
\[
\Gr(2,U_6) \cap \P((\Ker \lambda_X)^{\vee}) = \emptyset.
\]
In particular, $\dim (\Ker \lambda_X) \leq 6$, $\dim (\Im \lambda_X) \geq 9$, and $\dim (\Coker \lambda _X) \leq 1$.
\end{lem}

\begin{proof}
Suppose $\Gr(2,U_6) \cap \P((\Ker \lambda_X)^{\vee}) \neq  \emptyset$. 
Fix a closed point $[W] \in \Gr(2,U_6) \cap \P((\Ker \lambda_X)^{\vee})$. 
Then 
$\varphi (X)$ is contained in the Schubert divisor $H_W$, 
because $\P((\Ker \lambda_X)^{\vee})$ parametrises all the hyperplanes on 
$\P(\wedge^2 U_6^\vee)$ containing 
$\P(\Im \lambda_X)  (=\langle\varphi(X)\rangle)$ (cf. Notation \ref{n g=8 3}). 
By applying the pullback $\varphi^*$ 
to 
$W \otimes \MO_{\Gr(U_6,2)} \hookrightarrow U_6 \otimes \MO_{ \Gr(U_6,2)} \to \cQ$, 
we obtain an $\MO_X$-module homomorphism $\alpha : W \otimes \MO_X \to U_6 \otimes \MO_X \to E$, whose exterior power  $\wedge^2 \alpha : \wedge^2 W \otimes \MO_X \to \wedge^2 E$ is zero 
(because the image of $\wedge^2 W = H^0(X, \wedge^2 W \otimes \MO_X) \to H^0(X, \wedge^2 E) \simeq H^0(X, -K_X)$ corresponds to the inverse image $\varphi^{-1}(H_W)$, which is equal to the whole space $X$ by $\varphi(X) \subset H_W$). 

Since $W \subset U_6 = H^0(X, E)$ is a nonzero subspace, 
the image $\Im(\alpha)$ 
of 
$\alpha : W \otimes \MO_X  \to E$ is not of rank zero. 
This, together with $\wedge^2 \alpha =0$, implies that $\Im(\alpha)$ is of rank one. 
For the double dual $L := (\Im(\alpha))^{**}$, 
$L$ is an invertible sheaf on $X$ and we get 
a sequence of $\MO_X$-module homomorphisms
\[
\alpha : W \otimes \MO_X \xrightarrow{\beta} (\Im(\alpha))^{**}= L 
\xrightarrow{\gamma} E. 
\]
By $\alpha \neq 0$, we get $\beta \neq 0$ and $\gamma \neq 0$. 
As $L$ is an invertible sheaf, $\gamma$ is injective. 
Since  
\[
H^0(\alpha) : W = H^0(X, W \otimes \MO_X) \to H^0(X, E) =U_6
\]
coincides with the natural inclusion $W \subset U_6$, $H^0(\alpha)$ is injective, and hence so is $H^0(\beta)$. 
Therefore, we get $h^0(X, L) \geq h^0(X, W \otimes \MO_X) = \dim_k W = 2$. 
This, together with $\Pic X = \Z K_X$, implies 
 $L \simeq \MO_X(-nK_X)$ for some integer $n >0$. 
However, this leads to the following contradiction: 
\[
9\leq h^0(X, -nK_X) =h^0(X, L) \leq  h^0(X, E) = 6. 
\]
This completes the proof of $\Gr(2,U_6) \cap \P((\Ker \lambda_X)^{\vee}) = \emptyset$.


It remains to prove the in-particular part. 
We have $\dim \P(\wedge^2 U_6^\vee) = 14$ and $\dim \Gr(2, U_6)=8$. 
By $\Gr(2,U_6) \cap \P((\Ker \lambda_X)^\vee) = \emptyset$, 
we obtain $\dim (\Ker \lambda_X) -1 = \dim \P((\Ker \lambda_X)^\vee)  <14-8 =6$, and hence  $\dim (\Ker \lambda_X) <7$.  
Thus it holds that $\dim (\Im \lambda_X) = \dim (\wedge^2 U_6) - \dim (\Ker \lambda_X)  \geq 15-6 =9$. 
We then get $\dim (\Coker \lambda_X) = h^0(X, -K_X) -\dim (\Im \lambda_X) \leq 10-9 =1$. 
\end{proof}

\begin{prop}\label{p lambdaT surje}
We use Notation \ref{n g=8 3}.  
Then $\lambda_T$ is surjective and 
the composition $\varphi|_T \colon T \hookrightarrow X \xrightarrow{\varphi} \Gr(U_6,2)$ is the closed immersion induced 
by the complete linear system $|\MO_X(-K_X)|_T|$. 
Moreover, 
{$\dim(\Gr(2,U_6)\cap \P((\Ker \lambda_T)^\vee) \leq 0$, i.e., 
$\Gr(2,U_6)\cap \P((\Ker \lambda_T)^\vee)$ 
is either empty or zero-dimensional.} 
\end{prop}

\begin{proof}
We now prove that $\lambda_T$ is surjective. 
We have the following commutative diagram in which the lower horizontal sequence is exact: 
\[
\begin{tikzcd}
& &
\wedge^2 H^0(X,E) \arrow[r, "\simeq"] \arrow[d, "\lambda_X"]&  
\wedge^2 H^0(T,E_T) \arrow [d, "\lambda_T"] \\
0 \arrow[r]&
H^0(X, \MO_X) \arrow[r]&
 H^0(X,\wedge^2 E)\arrow[r, "\rho"] &  
 H^0(T,\wedge^2 E_T) \arrow[r] & 0. 
\end{tikzcd}
\]
If $\lambda_X$ is surjective, then so is $\lambda_T$. 
By $\dim (\Coker \lambda _X) \leq 1$, we may assume that $\dim (\Coker \lambda _X)  = 1$. 
Since $T$ is chosen to be general, 
we get $\Im(\lambda_X) \cap \Ker(\rho) = 0$, i.e., 
the induced map 
$\Im(\lambda_X) \to H^0(T, \wedge^2 E_T)$ is injective. 
It holds that 
\[
\dim(\Im(\lambda_X)) = h^0(X,\wedge^2 E) - \dim (\Coker \lambda _X) 
= h^0(X,\wedge^2 E) -1 = h^0(T,\wedge^2 E_T), 
\]
and hence $\Im(\lambda_X) \xrightarrow{\simeq} H^0(T, \wedge^2 E_T)$. 
Therefore, the composite arrow $\rho \circ \lambda_X$ is surjective, and hence so is $\lambda_T$. 
Since the composition $\varphi|_T \colon T \hookrightarrow X \xrightarrow{\varphi} \Gr(U_6,2)$ is defined by the surjection $H^0(T,E_T) \otimes \MO_T \to E_T$, 
$\varphi|_T$ is the morphism induced by the complete linear system 
$|\wedge^2 E_T | = |\MO_X(-K_X)|_T|$, which is a closed immersion.

{Let us show that $Z := \Gr(2,U_6)\cap \P((\Ker \lambda_T)^\vee)$ is either empty or zero-dimensional.} 
By applying the snake lemma to the above diagram, we get 
the following exact sequence: 
\[
0 \to \Ker(\lambda_X) \to \Ker(\lambda_T) \to H^0(X, \MO_X). 
\]
Hence either $\P((\Ker\lambda_X)^\vee) = \P((\Ker\lambda_T)^\vee)$ or 
$\P((\Ker\lambda_X)^\vee)$ is a hyperplane on $\P((\Ker\lambda_T)^\vee)$. 
The closed subset $Z = \Gr(2,U_6)\cap \P((\Ker \lambda_T)^\vee)$ of the projective space $\P( (\Ker\lambda_T)^\vee)$ 
satisfies $Z \cap \P((\Ker\lambda_X)^\vee) = \emptyset$ by 
\[
Z \cap \P((\Ker\lambda_X)^\vee) =  \Gr(2,U_6)\cap \P((\Ker \lambda_T)^\vee) \cap 
 \P((\Ker\lambda_X)^\vee)  \subset \Gr(2,U_6) \cap \P((\Ker\lambda_X)^\vee) = \emptyset. 
\]
Therefore, $Z = \Gr(2,U_6)\cap \P((\Ker \lambda_T)^\vee)$ is either empty or zero-dimensional. 
\end{proof}


\begin{lem}\label{l K3 linear section}
We use Notation \ref{n g=8 3}.  
Then 
$\varphi(T) = \Gr(U_6,2) \cap \langle \varphi(T) \rangle$.
\end{lem}

Although the following proof is similar to  \cite[Proof of the claim after Lemma 3.10]{Muk93}, 
we provide some details for the reader's convenience. 

\begin{proof}
We identify $T$ with $\varphi(T)$ by abuse of notation. 
Set  $\P^8 := \langle T \rangle$. 
Since each of  $ \Gr(U_6,2) \subset \P(\wedge^2 U_6)$ 
and $T \subset \P^8$ is an intersection of quadrics, 
it is enough to show that the restriction map 
\[
\rho : H^0(\P(\wedge^2 U_6), I_{\Gr(U_6,2)}(2)) \to H^0(\P^8, I_{T}(2))
\]
is an isomorphism, 
where $I_{\Gr(U_6,2)}$ (resp. $I_T$) denotes the ideal sheaf 
on the projective space $\P(\wedge^2 U_6)$ (resp. $\P^8$) 
defining $\Gr(U_6,2)$ (resp. $T$). 

By $h^0(\P(\wedge^2 U_6), I_{\Gr(U_6,2)}(2))=15=h^0(\P^8,I_{T}(2))$ 
(cf.\ Lemma \ref{l h0 compute}), it is enough to prove 
that $\rho$ is injective. 
Let $Q \subset \P(\wedge^2 U_6) (\simeq \P^{14})$ be a quadric hypersurface containing $\Gr(U_6,2)$. 
Suppose that $Q$ contains $\P^8$. 
It suffices to derive a contradiction. 
By \cite[Proposition 1.4]{Muk93}, 
it holds that  $\dim(\Sing Q) = -\infty, 4,$ or $8$. 
We have $\dim(\Sing Q) \neq  -\infty$ (i.e., $\Sing Q \neq \emptyset$), because a smooth quadric hypersurface on $\P^{14}$ 
does not contain $\P^8$ by Proposition~\ref{p quadric CH}.

Assume $\dim(\Sing Q)=4$. 
By \cite[Proposition 1.6 and Proposition 1.8(2)]{Muk93}, 
$Q^*$  is an $8$-dimensional 
smooth quadric hypersurface on the linear subspace $(\Sing  Q)^\perp \simeq \P^9$   such that $\dim (Q^* \cap \Gr(2, U_6)) =5$. 
By $\CH^3(Q^*) = \Z H^3$ for a hyperplane section $H$ on $Q^*$ 
(Proposition \ref{p quadric CH}), 
a $5$-dimensional irreducible component $W$ of $Q^* \cap \Gr(2, U_6)$ satisfies 
$W  = m H^3$ in $\CH^3(Q^*)$ for some integer $m>0$. 
Moreover, by \cite[Proposition 1.7]{Muk93} (applied for $P := \P((\Ker \lambda_T)^{\vee})$), we have $\dim ( Q^* \cap \P((\Ker \lambda_T)^{\vee}) ) \geq 4$. 
Thus $\dim (Q^* \cap \Gr(2, U_6) \cap  \P((\Ker \lambda_T)^{\vee}) ) \geq 1$.
This contradicts {$\dim(\Gr(2,U_6)\cap \P((\Ker \lambda_T)^{\vee})) \leq 0$} (Proposition \ref{p lambdaT surje}).  

Assume $\dim(\Sing Q)=8$. 
By \cite[Proposition 1.6 and Proposition 1.8(1)]{Muk93}, 
$Q^*$ is a $4$-dimensional smooth quadric hypersurface on 
the linear subspace $(\Sing  Q)^\perp \simeq \P^5$ 
such that $Q^* \subset \Gr(2, U_6)$. 
By \cite[Proposition 1.7]{Muk93} (applied for $P := \P((\Ker \lambda_T)^{\vee})$), 
it holds that $\dim ( Q^* \cap \P((\Ker \lambda_T)^{\vee}) ) \geq 2$.
Thus we get 
\[
\dim (\Gr(2, U_6) \cap  \P((\Ker \lambda_T)^{\vee})) \geq  \dim ( Q^* \cap \P((\Ker \lambda_T)^{\vee}) ) \geq 2, 
\]
which  contradicts {$\dim (\Gr(2,U_6)\cap \P((\Ker \lambda_T)^{\vee})) \leq 0$  (Proposition \ref{p lambdaT surje}).} 
\end{proof}

\begin{thm}\label{t main}
We use Notation \ref{n g=8 3}.  
Then 
$\lambda_X$ is surjective, 
$\varphi \colon  X \to \Gr(U_6,2)$ 
is the closed immersion induced  by the complete linear system $\left| - K_X \right|$, and $X = \Gr(U_6,2) \cap \langle X \rangle$. 
\end{thm}

\begin{proof}
Suppose that $\lambda_X$ is not surjective. 
Then  $\dim \Im \lambda _X = \dim \Im \lambda_T$.
 Note that $H^0(X,-K_X) \to H^0(T,-K_X|_T) \simeq H^0(T,\MO_T(T))$ is surjective.
 Thus $\Im \lambda _X \simeq  \Im \lambda_S$.
 Then
 \[
\varphi(X) \subset \Gr(U_6,2) \cap \langle \varphi(X) \rangle = \Gr(U_6,2) \cap \langle T \rangle = T,
 \]
which is impossible.

Thus $\varphi \colon X \hookrightarrow \Gr(U_6,2)$ is the closed immersion 
induced by  the complete linear system $\left|-K_X \right|$. 
Set $\P^{14} := \P(\wedge^2 U_6), \P^9 := \langle X \rangle$, and $\P^8 := \langle S \rangle$. 
We have the following restriction maps: 
\[
\beta \circ \alpha : H^0(\P^{14},I_{\Gr(U_6,2)}(2)) \xrightarrow{\alpha} 
H^0(\P^9,I_X(2)) \xrightarrow{\beta, \simeq} 
H^0(\P^8, I_T(2)), 
\]
where $I_{\Gr(U_6,2)}$ (resp. $I_X$, resp. $I_T$) 
is the ideal sheaf on $\P^{14}$ (resp. $\P^9$, resp. $\P^8$) 
defining $\Gr(U_6, 2)$ (resp. $X$, resp. $T$). 
Since $T$ is a hyperplane section of $X$, $\beta$ is an isomorphism (cf. \cite[Lemma 7.6(1)]{FanoI}). 
By $T = \Gr(U_6,2) \cap \langle T \rangle$, 
the composition $\beta \circ \alpha$ is an isomorphism (Lemma \ref{l K3 linear section}). 
Therefore, $\alpha$ is an isomorphism. 
Since $X$ (resp. $\Gr(U_6,2)$) is an intersection of quadrics on $\P^9$ (resp. $\P^{14}$), we get $X = \Gr(U_6, 2) \cap \P^9$. 
\end{proof}

We close this section by pointing out the fact that 
the conclusion $\dim(\Gr(2,U_6)\cap \P((\Ker \lambda_T)^\vee) \leq 0$
in  Proposition \ref{p lambdaT surje} can be strengthened as follows. 

\begin{prop}\label{p lambdaT surje2}
We use Notation \ref{n g=8 3}. 
Then $\dim(\Gr(2,U_6)\cap \P((\Ker \lambda_T)^{\vee})) =\emptyset$. 
\end{prop}

\begin{proof}
Set $\P^{14} := \P(\wedge^2 U_6)$ and $(\P^{14})^{\vee} := \P(\wedge^2 U_6^{\vee})$. 
We have the following inclusions of linear subvarieties: 
\[
\P^{14} \supset \P(\Im \lambda_X) (\simeq \P^9) 
\supset \P(\Im \lambda_T) (\simeq \P^8).  
\]
Taking their projective duals, 
we get the corresponding inclusions of linear subvarieties: 
\[
(\P^{14})^{\vee} \supset \P( (\Ker \lambda_T)^{\vee}) =: \P^5 \supset 
\P((\Ker \lambda_X)^{\vee}) =: \P^4.  
\]
Recall that $\P(\Im \lambda_X) = \langle X \rangle$ and 
$\P(\Im \lambda_T) = \langle T \rangle$. 
Since $T$ is a general hyperplane section of $X$ (Notation \ref{n g=8 3}), 
$\P(\Im \lambda_T)$ is a general hyperplane section of $\P(\Im \lambda_X)$. 
Hence $\P( (\Ker \lambda_T)^{\vee}) = \P^5$ is a general $5$-dimensional linear subvariety of $(\P^{14})^{\vee}$ containing $\P((\Ker \lambda_X)^{\vee}) = \P^4$. 
Let 
\[
\pi : (\P^{14})^{\vee} \setminus \P^4 
\to 
\P^9. 
\]
be the projection from $\P^4$. 
For an irreducible closed subset $W$ of $(\P^{14})^{\vee}$ satisfying 
$W \setminus\P((\Ker \lambda_X)^{\vee}) \neq \emptyset$, 
we set $\pi(W) := \overline{\pi( W \setminus\P((\Ker \lambda_X)^{\vee}))}$ 
by abuse of notation. 
Then $\pi(\P^5)$ is a general point on $\P^9$. 
By $\dim \pi(\Gr(2, U_6)) \leq \dim \Gr(2, U_6) =8$, 
we get 
$\pi(\Gr(2, U_6)) \cap \pi(\P^5) = \emptyset$. 
Hence it holds that 
\[
(\Gr(2, U_6)\cap \P^5) \setminus \P^4
\subset 
(\Gr(2, U_6) \setminus \P^4) \cap \pi^{-1}(\pi(\P^5)) 
\subset \pi^{-1}(\pi ( \Gr(2, U_6) \setminus \P^4) \cap  \pi(\P^5))  = \emptyset, 
\]
which implies 
$\Gr(2, U_6) \cap \P^5 \subset \P^4$. 
Therefore, we get $\Gr(2, U_6) \cap \P^5 \subset \Gr(2, U_6) \cap \P^4 
\overset{(\star)}{=} \emptyset$, where 
$(\star)$ is guaranteed by Lemma \ref{l Gr P(Ker) empty}. 
\qedhere


\end{proof}

\section{Irrationality}\label{s irrational}
In this section,  we prove:
\begin{thm}\label{thm:birational to cubic 3-fold}
 Any prime Fano $3$-fold $X$ of genus $8$ is birational to a 
 smooth cubic threefold $Y \subset \P^4$.
\end{thm}

By \cite{CG72}, \cite{Mur73}, \cite[Theorem A]{Ciu}, 
smooth cubic threefolds are irrational.
Thus we have
\begin{cor}\label{c irrational}
Any prime Fano $3$-fold $X$ of genus $8$ is irrational.
\end{cor}

The proof is based on an account of \cite{Kuz04}, which goes back to the works \cite{Fan30}, \cite{Isk80}, \cite{Put82}.

\subsection{Pfaffian cubic hypersurface}
Let $U_6$ be a vector space of dimension $6$.
Then the vector space $\wedge^2 U_6$ parametrizes alternating forms on $U_6^\vee$.
Since alternating forms are classified by their ranks \cite[Lemma 1.4, Proposition 1.8]{EKM08}, there are three nonzero $\GL_6$-orbits in $\wedge^2 U_6$, or equivalently,
there are three $\PGL_6$-orbits in $\P(\wedge^2 U_6^\vee)$.
The open orbit parametrizes alternating forms of rank $6$, the \emph{Pfaffian cubic hypersurface} $\cC$ parametrizes those of rank $\leq 4$, and the Grassmann variety $\Gr(2,U_6)$ parametrizes those of rank $2$:
\[
\Gr(2,U_6) \subset \cC \subset \P(\wedge^2 U_6^\vee).
\]
The Pfaffian cubic hypersurface $\cC$ is defined by the Pfaffian of alternating forms.

\begin{rem}
We have the following commutative diagram 
\[
\begin{tikzcd}
\text{Mat}_6(k)  \arrow[r, phantom, "\simeq"] \arrow[d, phantom, "\rotatebox{90}{$\subset$}"] &  
{\rm Bil}(U_6^{\vee} \times U_6^{\vee}, k)  \arrow[r, phantom, "\simeq"] \arrow[d, phantom, "\rotatebox{90}{$\subset$}"] & 
\Hom_k(U_6^{\vee} \otimes_k U_6^{\vee}, k)  \arrow[r, phantom, "\simeq"]  \arrow[d, phantom, "\rotatebox{90}{$\subset$}"] & \Hom_k(U_6^{\vee}, U_6) \arrow[d, phantom, "\rotatebox{90}{$\subset$}"] \\
\text{Alt}_6(k)  \arrow[r, phantom, "\simeq"]           & 
{\rm Alt}(U_6^{\vee} \times U_6^{\vee}, k) \arrow[r, phantom, "\simeq"]          & 
\Hom_k(U_6^{\vee} \wedge U_6^{\vee}, k) \arrow[r, phantom, "\simeq"]          & 
\text{Alt}_6(U_6^{\vee}, U_6), 
\end{tikzcd}
\]
where the middle and right horizontal isomorphisms are the natural ones, 
whilst each of the left horizontal isomorphisms depends on a fixed $k$-linear isomorphism $k^6 \xrightarrow{\simeq} U_6^{\vee}$. 
Here we use the following terminologies:
\begin{itemize}
\item $\text{Mat}_6(k)$ is the $k$-vector space consisting of $6 \times 6$ matrices, and 
$\text{Alt}_6(k)$ is the $k$-vector subspace  of alternating matrices. 
\item ${\rm Bil}(U_6^{\vee} \times U_6^{\vee}, k)$ is the $k$-linear space of bilinear forms on $U_6^{\vee}$, and ${\rm Alt}(U_6^{\vee} \times U_6^{\vee}, k)$ is the 
$k$-linear subspace of alternating forms. 
\item $\Hom_k(V, W) := \{ f : V \to W \,|\, f\text{ is a $k$-linear map}\}$. 
\item $\text{Alt}_6(U_6^{\vee}, U_6)$ is the $k$-linear subspace of 
{\em alternating maps} $\phi: U_6^{\vee} \to U_6$, i.e., 
$\phi : U_6^{\vee} \to U_6$ is a $k$-linear map such that 
the composition $M \to U_6^{\vee} \to U_6 \to M^{\vee}$ is zero for every $1$-dimensional $k$-vector subspace $M \subset U_6^{\vee}$.
\end{itemize}

\end{rem}

\subsection{Cubic threefolds}
Let $X$ be a Fano $3$-fold of genus $8$, which is obtained as the linear section of $\Gr(U_6,2)$.
Recall that we have the following diagram:
\[
\begin{tikzcd}
\Gr(2,U_6) \arrow[r]  & \P(\wedge^2 U_6^\vee) (\simeq \P^{14})\\
\Gr(2,U_6)\cap \P((\Ker \lambda_X)^\vee) =\emptyset \arrow[u] \arrow[r]  & \P((\Ker \lambda_X )^\vee) (\simeq \P^4). \arrow[u] 
\end{tikzcd}
\]

\begin{dfn}\label{d cubic of X}
Let the notations be as above.
\begin{enumerate}
 \item  The cubic $3$-fold $Y$ of $X$ is the cubic hypersurface
 \[
Y = \cC \cap \P((\Ker \lambda_X)^\vee) \subset \P((\Ker \lambda_X )^\vee).
 \]
\item Since $\Gr(2,U_6)\cap \P((\Ker \lambda_X)^\vee) =\emptyset$, the cubic threefold $Y$ parametrizes alternating forms of rank $4$.
Namely, a closed point $y \in Y$ defines an alternating map $\phi _y \colon U_6 ^\vee \to U_6$ with $2$-dimensional cokernel $Q_y$.
By sending $y \in Y$ to $[Q_y]$, we have the following map
\[
\theta \colon Y \to \Gr(U_6,2).
\]
\end{enumerate}
\end{dfn}

\begin{prop}[{\cite[Theorem 2.2, Proposition A.4]{Kuz04}}]\label{p cubic of X is smooth}
 Let the notations be as above.
\begin{enumerate}
 \item $\theta (\Sing Y) = \theta(Y) \cap X = \Sing X = \emptyset$.
 \item $c_1(\theta^*\cQ) = 2 h$ and $c_2(\theta^*\cQ) \equiv  \frac{5}{3} h^2$, where $h$ is the class of the hyperplane section of $Y \subset \P((\Ker \lambda_X )^\vee)$.
\end{enumerate}
\end{prop}

\begin{proof}
Let
\[
0 \to \cK \to \MO_{\Gr(U_6,2)}^{\oplus 8} \to \cQ \to 0
\]
be the universal sequence on $\Gr(U_6,2)$.

First,  we prove that the map $\theta$ is injective.
Otherwise, there are two different points in 
$\Ker \lambda_X$ 
(more precisely, in $\P( (\Ker \lambda_X)^{\vee})$)  that define alternating maps $U_6 ^\vee \to U_6$ with the same $2$-dimensional cokernel $Q$.
Then some linear combination of these two forms is of rank two, which contradicts 
$\Gr(2,U_6)\cap \P((\Ker \lambda_X)^\vee) =\emptyset $.

Second,  we prove $\theta (\Sing Y) = \theta(Y) \cap X$.
By the definition of $X$, we have
\[
X = \{[Q] \in \Gr(U_6,2) \mid \text{the composition $\Ker \lambda_X \to \wedge^2 U_6 \to \wedge^2Q$ is zero} \}.
\]
Thus the intersection $\theta(Y) \cap X$ parametrizes $[Q] \in \Gr(U_6,2)$ such that
\begin{itemize}
 \item the composite $\Ker \lambda_X \to \wedge^2 U_6 \to \wedge^2Q$ is zero, and
 \item there is a rank $4$ point $y \in \Ker \lambda_X$ such that  $\phi_y \colon U_6 ^\vee \to U_6$ has the $2$-dimensional cokernel $Q$.
\end{itemize}

The conormal 
bundle ${C_{Y/\P((\Ker \lambda_X )^\vee)} := I_Y/I_Y^2}$ of $Y \subset  \P((\Ker \lambda_X )^\vee)$ is naturally isomorphic to $(\wedge^2 \theta^*\cQ)^\vee \otimes \MO(-1)$
%
(cf.\ \cite[Remark 1.13]{Muk01}). 
Note that $C_{Y/\P((\Ker \lambda_X )^\vee)}$  is an invertible sheaf on $Y$.  
Moreover, the $\MO_Y$-module homomorphism 
\[
C_{Y/\P((\Ker \lambda_X )^\vee)} \to \Omega^1_{\P( (\Ker \lambda_X)^\vee)}|_Y
\]
at the point $y \in Y$ is identified with the natural induced map
\begin{equation} \label{equation:normal_map}
(\wedge^2 Q_y)^\vee \to  (\Ker \lambda_X / L_y)^\vee .
\end{equation}
Here $L_y \subset \Ker \lambda_X$ is the $1$-dimensional subspace corresponding to the point $y \in Y \subset \P( (\Ker \lambda_X)^\vee)$, and map \eqref{equation:normal_map} is induced from the composition $\Ker \lambda_X \to \wedge^2 U_6 \to \wedge^2Q$.
Thus we have
\[
y \in \Sing Y \iff\text{\eqref{equation:normal_map} is zero} \iff \theta(y) \in X,
\]
where the first equivalence follows from \cite[Ch. II, Theorem 8.17]{Har77}. 
Therefore, we have $\theta (\Sing Y) = \theta(Y) \cap X$.

Thirdly, we show $\theta(Y) \cap X = \Sing X$.
Fix a closed point $[Q] \in X \subset \Gr(U_6,2)$ and 
consider the induced 
exact sequence
\[
0 \to K \to U_6 \to Q \to 0.
\]
Since the cotangent bundle $\Omega^1_{\Gr(U_6, 2)}$
of the Grassmannian variety $\Gr(U_6, 2)$ is isomorphic to $\mathcal Hom(\mathcal Q, \cK)$, 
the cotangent space 
$\Omega^1_{\Gr(U_6, 2)} \otimes k([Q])$ of $\Gr(U_6,2)$ at the point $[Q]$ is naturally isomorphic to $Q^\vee \otimes K \simeq  (\wedge^2Q)^\vee \otimes Q \otimes K$.
On the other hand, the conormal bundle of the inclusion 
$\P(\Im \lambda_X) \subset \P(\wedge^2 U_6)$ 
at this point is naturally isomorphic to
\[
\Hom (\wedge^2 Q  , \Ker \lambda_X ) \simeq (\wedge^2 Q)^\vee \otimes \Ker \lambda_X .
\]
Then the natural map
$C_{X/ \Gr} \otimes_{\MO_X} k(x) \to (\Omega_{\Gr(U_6,2)}^1|_X) \otimes_{\MO_X} k(x) $ is identified with the map
\begin{equation}\label{equation:restriction}
(\wedge^2 Q)^\vee \otimes \Ker \lambda_X  \to (\wedge^2Q)^\vee \otimes Q \otimes  K.
\end{equation}
Here the map is obtained as follows:
By taking the second exterior power of the sequence
\[
0 \to K \to U_6 \to Q \to 0,
\]
we have the following exact sequences:
\begin{align*}
 0 \to F \to \wedge^2 U_6 \to \wedge^2 Q \to 0,\\
 0 \to \wedge^2 K \to F \to Q \otimes K \to 0,
\end{align*}
Since the composite $\Ker \lambda_X \to \wedge^2 U_6 \to \wedge^2 Q$ is zero,  we have the induced map $\Ker \lambda_X \to F$.
By composing the map $F \to Q \otimes K$ with this induced map, we have a map $\Ker \lambda_X  \to  Q \otimes  K$.
By tensoring  $(\wedge^2Q)^\vee$,  we have \eqref{equation:restriction}.
Then, for $x \in X$, we have
\[
x \in \Sing X \iff \text{\eqref{equation:restriction} is not injective} \iff \wedge^2 K \cap \Ker \lambda_X \neq 0 
\iff x \in \theta(Y),
\]
where the first equivalence follows from \cite[Ch. II, Theorem 8.17]{Har77}.
Thus $\theta(Y) \cap X = \Sing X$.

Finally, we prove $c_1(\theta^*\cQ) = 2 h$ and $c_2(\theta^*\cQ) \equiv  \frac{5}{3} h^2$.
By the definition, we have the following exact sequence
\[
0 \to U_6^\vee \otimes \MO_{\P((\Ker \lambda_X )^\vee)}(-1) \to U_6 \otimes \MO_{\P((\Ker \lambda_X )^\vee)} \to \theta_*\theta^* \cQ \to 0.
\]
Since ${C_{Y/\P((\Ker \lambda_X )^\vee)}} \simeq (\wedge^2 \theta ^*\cQ)^\vee \otimes \MO(-1)$
it holds that $c_1(\theta^*\cQ) = 2 h$.
The assertion for $c_2(\theta^*\cQ)$  follows from the following equation obtained from the Grothendieck-Riemann-Roch theorem:
\[
\ch (\theta_*\theta^* \cQ) = \theta_*(\ch(\theta^*\cQ) \cdot \td(N_{Y/\P((\Ker \lambda_X )^\vee)})^{-1}).
\]
\end{proof}

\subsection{Palatini quartics}\label{ss Palatini}
Let $E$ be the bundle defining the Grassmann embedding $\varphi \colon X \to \Gr(U_6,2)$, i.e., $E = \varphi^* \cQ$.
By the surjection $U_6 \otimes \MO_X \to E$, we get the induced morphism 
\[
\alpha \colon \P_X(E) \hookrightarrow \P_X(U_6 \otimes \MO_X) 
= \P(U_6) \times_k X \xrightarrow{\pr_1} \P(U_6). 
\]
Similarly, we have the induced morphism $\beta \colon \P_Y(\theta^*\cQ) \to \P(U_6)$, and the following diagram:
\[
\begin{tikzcd}
\P_X(E) \arrow[rd,"\alpha"] \arrow[d, "\P^1\text{-bundle}"']  &   & \P_Y(\theta^*\cQ)) \arrow[ld,"\beta"'] \arrow[d, "\P^1\text{-bundle}"] \\
X&  \P(U_6)(\simeq \P^5) & Y.
\end{tikzcd}
\]
\begin{prop}[{\cite{Fan30}, \cite{Isk80}, \cite{Put82}, \cite[Proposition 2.11]{Kuz04}}]\label{p Palatini quartic}
The images of $\alpha$ and $\beta$ are the same quartic hypersurface (called \emph{Palatini quartic}).
Moreover, the maps $\alpha$ and $\beta$ are birational onto its image.
\end{prop}

\begin{proof}
By the definition of $\alpha$ and $\beta$, we see that, for $[L] \in \P(U_6)$,
\begin{align*}
& [L] \in \Im(\alpha)\\
\iff & \text{$L$ is obtained as successive quotients $U_6 \to Q \to L$ for some $[Q] \in X$}
\end{align*}
and
\begin{align*}
 &[L] \in \Im(\beta)\\
 \iff &\text{$L$ is obtained as successive quotients $U_6 \to Q \to L$ for some $[Q] \in \theta(Y)$}.
\end{align*}

Note that, for $[L] \in \P(U_6)$ and $y \in \wedge^2 U_6 $, the composite $k$-linear map  
\[
L^\vee \to U_6^\vee \xrightarrow{\phi_y} U_6 \to L
\]
is zero, and hence induces a $k$-linear map
\[
\psi_{L,y} \colon U_6^\vee / L^\vee \to L.
\]
By using the Euler sequence, 
we see that this map corresponds to the $k$-linear map of 
$T_{\P(U_6)}(-1) \to \MO_{\P(U_6)}(1) \otimes_k \wedge^2 U_6^\vee$ at $y$,  
and hence we have the following map on $\P(U_6)$:
\[
\psi \colon T_{\P(U_6)}(-1) \to \MO_{\P(U_6)}(1) \otimes_k (\Ker \lambda_X)^\vee.
\]

Then the following equivalences hold for $[L] \in \P(U_6)$: 
\begin{align*}
 & [L]\in \Im (\beta)\\
\iff & \text{there is a rank $4$ point $y \in \Ker \lambda_X$ such that the composite $U_6^\vee \xrightarrow{\phi_y} U_6 \to L$ is zero}.\\
\iff & \text{there is a point $y \in \Ker \lambda_X$ such that the composite $U_6^\vee \xrightarrow{\phi_y} U_6 \to L$ is zero}.\\
\iff & \text{there is a point $y \in \Ker \lambda_X$ such that $\psi_{L,y}$ is zero}.\\
\iff & \text{$\wedge^5 \psi=0$ at $[L]$}.
\end{align*}

Similarly.
\begin{align*}
 & [L]\in \Im (\alpha)\\
\iff & \text{$L$ is obtained as a successive quotients $U_6 \to Q \to L$ for some $[Q] \in X$}.\\
\iff & \text{$L$ is obtained as a successive quotients $U_6 \to Q \to L$ for some $[Q] \in \Gr(U_6,2)$}\\
&\text{such that $Q^\vee \to U_6^\vee \xrightarrow{\phi_y}  \ U_6 \to Q$ is zero for any $y \in \Ker \lambda_X$}.\\
\iff & \text{$L$ is obtained as a successive quotients $U_6 \to Q \to L$ for some $[Q] \in \Gr(U_6,2)$}\\
&\text{such that $\psi_{L,y}|_{Q^\vee/L^\vee}\colon Q^\vee/L^\vee \to L $ is zero for any $y \in \Ker \lambda_X$}.\\
\iff & \text{$\cap_{y \in \Ker \lambda_X}\Ker \psi_{L,y}$ is nonzero}.\\
\iff & \text{$\wedge^5 \psi=0$ at $[L]$}.
\end{align*}
Therefore, $\Im (\alpha) = \Im(\beta)$.


Recall that the projective bundle $\Pi \colon \P_{\Gr(U_6,2)}(\cQ) \to \Gr(U_6,2)$ is the flag variety parametrizing successive quotients $U_6 \to Q \to L$ with $\dim Q =2$ and $\dim L =1$.
Moreover, the natural map $A \colon \P_{\Gr(U_6,2)}(\cQ) \to \P(U_6) (\simeq \P^5)$ is a $\P^4$-bundle (associated with $\Omega_{\P(U_6)} (2)$), and each fiber $\P^4$ is mapped isomorphically onto a linear $4$-space contained in $\Gr(U_6,2)$ by $\Pi$.

Let $p \in \P(U_6)$ be a point.
Then 
\[
\alpha^{-1}(p) = A^{-1}(p) \cap \P_X(E) \simeq \Pi(A^{-1}(p)) \cap X
\]
is isomorphic to a linear section of the projective space $\Pi(A^{-1}(p)) \simeq \P^4$.
Thus each (non-empty) $\alpha$-fiber is isomorphic to a projective space.
On the other hand, since we know the Chern classes of $E$, we can show that the degree of the tautological divisor of the projective bundle $\P_X(E)$ is $4$.
Thus $\alpha$ is birational onto a quartic hypersurface in $\P(U_6)$.
Similarly, we can show that the degree of the tautological divisor of the projective bundle $\P_Y(\theta^*\cQ)$ is $4$.
Hence $\beta$ is also birational.
\end{proof}


\begin{proof}[Proof of Theorem~\ref{thm:birational to cubic 3-fold}]
 Let $Z$ be a general hyperplane section of the quartic hypersurface $\Im(\alpha) = \Im(\beta)$.
 Then $X$ and $Y$ are both birational to $Z$, and the assertion follows.
\end{proof}
 
\section{Two-ray game for quintic genus-one curves}


Let $X$ be a prime Fano threefold of genus $8$. 
There exist a field extension $k \subset \kappa$ and a geometrically integral regular genus-one curve $C$ on $X_{\kappa} := X \times_k \kappa$. 
The purpose of this subsection is to provide a concrete description for the two-ray game starting with the blowup of $X_{\kappa}$ 
along $C$. 
We give the details for results in Subsection \ref{ss 7.1}. 
On the other hand, 
we omit some details in Subsection \ref{ss 7.2}, which are left to the reader. 



\subsection{Birationality and smallness of $|-K_Y|$}\label{ss 7.1}

\begin{nota}\label{n qell blowup}
Let $X \subset \P^{9}_k$ be a prime Fano threefold of genus $8$. 
Take a field extension $k \subset \kappa$ and let $C$ be a geometrically integral regular genus-one curve 
on $X_{\kappa} := X \times_k \kappa$ 
such that $-K_{X_{\kappa}} \cdot C =5$ and $\langle C\rangle \cap X = C$.
Let $\sigma : Y \to X_{\kappa}$ be the blowup along $C$. 
It holds that $|-K_Y|$ is base point free, $-K_Y$ is big, and 
$h^0(Y, -K_Y)=5$ (Lemma \ref{l qell blowup basic}). 
For the induced morphism $\varphi_{|-K_Y|} : Y \to \P^4_{\kappa}$, 
let $\varphi : Y \to V := \varphi_{|-K_Y|}$ be the morphism onto the image and 
let $\psi : Y \to Z$ be its Stein factorisation: 
\[
\varphi: Y \xrightarrow{\psi} Z \to  V \subset \P^4_{\kappa}. 
\]
Then $\psi : Y \to Z$ is a birational morphism (possibly an isomorphism). 
Set $E := \Ex(\sigma)$. 
\end{nota}

\begin{lem}\label{l qell blowup basic}
We use Notation \ref{n qell blowup}. 
Then 
\begin{enumerate}
\item $(-K_Y)^3 =4$. 
\item $(-K_Y)^2 \cdot E = 5$. 
\item $(-K_Y) \cdot E^2 =0$. 
\item $E^3 = -5$. 
\item $h^0(Y, -K_Y) = 5$. 
\end{enumerate}
\end{lem}

\begin{proof}
The assertions (1)-(4) hold by the same argument as in the proof of 
Proposition \ref{p exist K3-like}.


Let us show (5). 
Since $|-K_Y|$ is base point free, 
the generic member $S$ is a regular prime divisor. 
Note that $Y$ is geometrically normal, because $Y$ is Cohen-Macaulary and smooth outside the curve lying over the  non-smooth points on $C$. 
By \cite[Proposition 2.11]{FanoI}, $S$ is geometrically integral. 
Hence we get $h^0(S, -K_Y|_S)= 2 + \frac{1}{2}(-K_Y|_S)^2 = 2 + \frac{1}{2}(-K_Y)^3 = 4$ \cite[Theorem 3.4(3)]{FanoI}. 
By $H^1(Y, \MO_Y)=0$ and the exact sequence 
\[
0 \to \MO_{Y} \to \MO_Y(-K_Y) \to \MO_S(-K_Y) \to 0, 
\]
it holds that 
\[
h^0(Y, -K_Y) = h^0(Y, \MO_Y) + h^0(S, -K_Y|_S) = 1+4=5. 
\]
Thus (5) holds. 
\end{proof}

\begin{prop}\label{p-div-cont1}
We use Notation \ref{n qell blowup}. 
Then $\dim \Ex(\psi) \neq 2$. 
\end{prop}

\begin{proof}
Suppose $\dim \Ex(\psi)=2$. 
Let us derive a contradiction. 
Set  $D := \Ex(\psi)$ and $B := \psi(D) = \psi(\Ex (\psi))$. 
\begin{equation}
\begin{tikzcd}
& Y \arrow[ld, "\sigma"'] \arrow[rd, "\psi"]\\
X & & Z
\end{tikzcd}
\end{equation}

We have $(-K_Y)^3 =4, (-K_Y)^2 \cdot E=5$, and $(-K_Y) \cdot E^2 =0$ 
(Lemma \ref{l qell blowup basic}). 
Let $\zeta$ be a curve on $Y$ such that $-K_Y \cdot \zeta =0$ 
(i.e., $\psi(\zeta)$ is a point). 
By using the  the same argument as in \cite[Proposition 4.2]{FanoII}, the following properties (1)-(8) hold. 

\begin{enumerate}
\item $D$ is a prime divisor. 
\item $D \cdot \zeta <0$ and $-D$ is $\psi$-ample. 
\item $E \cdot \zeta >0$, $E$ is $\psi$-ample, and $D \neq E$. 
\item We can uniquely write 
\[
D \sim -\alpha K_Y -\beta E
\]
for some $\alpha, \beta \in \Z_{>0}$. 
\item $B$ is a curve. 
Furthermore, $(-K_Y)^2 \cdot D = 0$, $(-K_Y) \cdot D^2 <0$, and 
$\alpha (-K_Y)^3 = \beta (-K_Y)^2 \cdot E$. 
\item 
$(\psi|_D)_* \MO_D = \MO_B$, 
where {$\psi|_D : D  \to B$} denotes the induced morphism. 
\item 
If $\xi$ is a curve on $D$ contained in a general fibre of $\psi|_D : D \to B$, 
then $K_Y \cdot \xi =0$ and $-D \cdot \xi \in \{1, 2\}$.   
\item $\beta \in \{1, 2\}$. 
\end{enumerate}
We now point out some minor differences we need to correct. 
In the proof of (5), use $(-K_Y) \cdot E^2 =0$ instead of 
$(-K_Y) \cdot E^2 >0$. 
Use \cite[Corollary 6.8]{Tan-elliptic} 
instead of \cite[Thoerem 0.5]{Tan15} cited in the proof of (6). 

\medskip

It follows from (5) that  
\[
4\alpha =\alpha (-K_Y)^3 = \beta (-K_Y)^2 \cdot E = 5\beta. 
\]
This, together with (4), implies $\beta \in 4\Z$, which contradicts (8). 
\end{proof}

\begin{prop}\label{p qell blowup main}
We use Notation \ref{n qell blowup}. 
Then the prime divisor $V$ on $\P^4_{\kappa}$ is of degree $4$ and 
$\psi : Y \to V$ is birational. 
\end{prop}

\begin{proof}
We have the induced morphisms $\varphi : Y \xrightarrow{\psi} Z \xrightarrow{\theta} V \subset \P^4_{\kappa}$, 
where $\psi$ is birational and $\theta$ is finite. 
For $A_V := \MO_{\P^4}(1)|_V$, we have 
\[
-K_Y \sim -\psi^*K_Z \sim \varphi^*A_V.
\]
Since $V$ is geometrically integral, we have 
$\Delta(V, A_V) := \dim V + A_V^3 -h^0(V, A_V) \geq 0$, which implies 
\[
0 \leq \Delta(V, A_V) = \dim V + A_V^3 -h^0(V, A_V)
=3 + \frac{(-K_Y)^3}{ \deg \theta} -h^0(Y, -K_Y) 
= \frac{4}{ \deg \theta}-2. 
\]
Hence either $\deg \theta =1$ or $\deg \theta =2$. 
We have 
\[
4 = ( -K_Y)^3 = (\varphi^*A_V)^3 = (\deg \theta) \cdot (\deg V). 
\]
Therefore, it suffices to show that $\deg \theta  \neq 2$. 

Suppose $\deg \theta =2$. 
Let us derive a contradiction. 
In this case, we have $\deg V = 2$, i.e., $V$ is a quadric hypersurface on $\P^4_{\kappa}$. 
Since $V$ is a geometrically integral quadric, 
$V$ is geometrically normal, $\rho(V)=1$, and $V$ is $\Q$-factorial. 
Recall that either 
\begin{enumerate}
\item[(i)] $\dim (\Ex(\psi))=1$, or 
\item[(ii)] $\psi: Y \to Z$ is an isomorphism. 
\end{enumerate}

Assume that $\theta : Z \to V$ is inseparable. 
If $\psi : Y \xrightarrow{\simeq} Z$ (i.e., (ii) holds), then we get a contradiction 
$2 = \rho(Y) = \rho(Z) \overset{(\star)}{=} \rho(V) =1$, 
where $(\star)$ is guaranteed by \cite[Proposition 2.4(1)]{Tan18b}. 
Thus (i) holds. 
In this case, $Z$ is not $\Q$-factorial. 
However, this contradicts the fact that $V$ is $\Q$-factorial and $\theta: Z \to V$ is a homeomorphism \cite[Lemma 2.5]{Tan18b}.

\medskip

In what follows, we assume that $\theta : Z \to V$ is separable. 
By $\deg(\theta)=2$, $K(Z)/K(V)$ is a Galois extension. 
For $\Gal(K(Z)/K(V)) = \{{\rm id}, \iota\}$, we have an automorphism $\iota: Y \to Y$ over $V$. 

(i) 
Assume $\dim(\Ex(\psi)) =1$. 
For $\psi' := \psi \circ \iota$, $X'_{\kappa} := X_{\kappa}, Y' := Y$, and $\sigma' :=\sigma$, we  have the following diagram: 
\begin{equation}
\begin{tikzcd}
Y \arrow[d, "\sigma"'] \arrow[rd, "\psi"]& & Y' \arrow[ld, "\psi' =\psi \circ \iota"'] \arrow[d, "\sigma'"]\\
X_{\kappa} & Z & X'_{\kappa}.
\end{tikzcd}
\end{equation}
Set $D := \Ex(\sigma')$. 
Let $E' \subset Y'$  be the proper transform of $E \subset Y$. 
By ${\rm Cl}\,Y = \Pic Y = \sigma^*(\Pic X) \oplus \Z E  = \Z K_Y \oplus \Z E$, 
we have $\Pic Y' = {\rm Cl}\,Y' = \Z K_{Y'} \oplus \Z E'$. 
Hence 
we can write  
\[
D \sim -\alpha K_{Y'} - \beta E'. 
\]
for some $\alpha, \beta \in \Z$. 
Since $V$ is $\Q$-factorial, 
we have $\theta^*(\varphi(E)) = c (E_Z+\iota^*E_Z)$ for $E_Z := \psi(E)$ and some $c \in \Q_{>0}$. Then it holds that 
\[
0 \equiv_{\psi'} \psi'^*\theta^*(\varphi(E)) = c(E' + D). 
\]
Since $D$ is $\psi'$-ample, $E'$ is $\psi'$-anti-ample. 
Hence we get $D \neq E'$, which implies $\alpha >0$. 
It holds that 
\begin{align*}
14 &= (-K_{X'})^3 \\
&= (-\sigma'^*K_{X'_{\kappa}})^3 \\
&= (-\sigma'^*K_{X'_{\kappa}})^2 \cdot (-K_{Y'} +D) \\
&= (-\sigma'^*K_{X'_{\kappa}})^2 \cdot (-K_{Y'})\\
&= (-K_{Y'} +D)^2 \cdot (-K_{Y'}) \\
&=  (-(\alpha+1)K_{Y'} -\beta E')^2 \cdot (-K_{Y'})\\
&= (\alpha+1)^2 (-K_{Y})^3 - 2(\alpha+1)\beta (-K_Y)^2 \cdot E  + \beta^2(-K_Y) \cdot E^2\\
&= 4(\alpha+1)^2 -10 (\alpha+1)\beta, 
\end{align*}
where the last equality holds by Lemma \ref{l qell blowup basic}. 
For $\gamma := \alpha +1$,  we have $\gamma \geq 2$ and 
\[
\gamma \Z \ni 2 \gamma^2 -5\gamma \beta = 7. 
\]
We then get  $\gamma =7$ and $\beta = 13/5$,  which contradicts $\beta \in \Z$.

(ii) Assume $\psi : Y \xrightarrow{\simeq} Z$. 
The Galois involution $\iota : Y \to Y$ on $Y$ induces 
the Galois involution $\iota_X :X \dashrightarrow X$ on $X$, i.e., 
 $\iota_X :X \dashrightarrow X$  is the birational map completing the following commutative diagram: 
\[
\begin{tikzcd}
Y \arrow[r, "\iota"] \arrow[d, "\sigma"] & Y\arrow[d, "\sigma"] \\
X \arrow[r, dashrightarrow, "\iota_X"]& X. 
\end{tikzcd}
\]
If $\iota_X$ is a morphism, then $\iota_X : X \to X$ is an isomorphism (because $\iota_X$ is a birational morphism of projective normal varieties of Picard number one). 
In this case, we get the induced morphism 
$V = Y/ \{ {\rm id}, \iota\} \to X/\{ {\rm id}, \iota_X\}$, which contracts at least one curve. This contradicts $\rho(V)=1$. 
Then $\iota_X$ is not a morphism, i.e., 
there exists a one-dimensional fibre $\zeta$ of $\sigma: Y \to X$ which is not contracted by the composition $\sigma' := \iota \circ \sigma : Y \xrightarrow{\iota} Y \xrightarrow{\sigma} X$. 
Set $D := \Ex(\sigma')$. 
We have 
\[
D \sim -\alpha K_Y -\beta E
\]
for some $\alpha, \beta \in \Z$. 
Since there exists  a  one-dimensional fibre $\zeta$ of $\sigma: Y \to X$ which is not contracted by $\sigma'$. 
we get $D \neq E$. 
By the same argument as in (i), we get $\alpha>0$ and $14 =(-K_X)^3 =4(\alpha+1)^2 -10(\alpha+1)\beta$, which is impossible. 
\end{proof}


    

\begin{prop}\label{p -K_Y not ample}
We use Notation \ref{n qell blowup}. 
Then $-K_Y$ is not ample. 
\end{prop}

\begin{proof}
Suppose that $-K_Y$ is ample. 
We have 
\[
\varphi : Y \xrightarrow{\simeq, \psi}  Z \to V (\subset \P^4_{\kappa}) 
\]
and hence $\varphi : Y \to V$ is a finite birational morphism onto a hypersurface $V \subset \P^4_{\kappa}$ of degree $4$. 
Since $Y$ is normal, $\varphi$ is the normalisation of $V$. It holds that 
\[
K_Y + C = \varphi^*K_V
\]
for the conductor divisor $C$, which is an effective divisor on $Y$ satisfying $\varphi(\Supp C) = \Sing V$. 
On the other hand, it holds that 
\[
\MO_Y(-K_Y) \simeq \varphi^*(\MO_{\P^4}(1)|_V) \simeq \varphi^*\MO_V(-K_V). 
\]
Hence we get $C \sim 0$, which implies $C=0$. 
Therefore, $\varphi : Y \to V$ is an isomorphism. 
However, this implies that $Y$ is a regular hypersurface on $\P^4_{\kappa}$, and hence $\rho(Y)=1$, which contradicts $\rho(Y)=2$. 
\end{proof}

\subsection{Description of two-ray game}\label{ss 7.2}

Let $X \subset \P^9_k$ be a prime Fano threefold of genus $8$. 
Set $\kappa := K(\P^3_k)$. 
Then 
there exists a geometrically integral regular genus-one curve $C$ on $X_{\kappa} := X \times_k \kappa$ such that $-K_{X_{\kappa}} \cdot C=5$ and $\langle C \rangle \cap X_{\kappa} = C$ (Proposition \ref{p q-ell exist}). 
The purpose of this subsection is to 
provide a concrete description 
for the two-ray game starting with the blowup $X' \to X_{\kappa}$ of $X_{\kappa}$ along $C$. 
For simplicity, we consider only the case when $k=\kappa$ and $C$ is a smooth elliptic curve. 
We omit some details, which 
are left to the reader.

\begin{nasi}[Diagram obtained from $X \subset \Gr(6, 2)$]
Let $X \subset \P^9_k$ be a prime Fano threefold of genus $8$. 
Fix a closed embedding $X \subset \Gr(6, 2)$ such that $X$ is a linear section of $\Gr(6, 2)$ (Theorem \ref{t main}). 
Recall that we have the smooth cubic threefold $Y$ of $X$ 
and the induced morphism $\theta : Y \to \Gr(6, 2)$ 
(Definition \ref{d cubic of X}, Proposition \ref{p cubic of X is smooth}). 
Set $E := \mathcal Q|_X$ (resp. $F := \theta^*\mathcal Q$), 
which is a locally free sheaf of rank two on $X$ (resp. on $Y$). 
The images of the induced morphisms 
$\alpha : \P_X(E) \to \P^5_k$ and 
$\beta : \P_Y(F) \to \P^5_k$ 
coincide, i.e., 
\[
P := \Im(\alpha) = \Im(\beta) \subset  \P^5, 
\]
which we call the Palatini quartic (cf.\ Subsection \ref{ss Palatini}). 
To summarise, we obtain  the following diagram: 
\begin{equation}\label{e1 nasi Pfaff diagram}
\begin{tikzcd}
\P_X(E) \arrow[rd,"a"] \arrow[d, "\P^1\text{-bundle}"']  &   & \P_Y(F) \arrow[ld,"b"'] \arrow[d, "\P^1\text{-bundle}"] \\
X&  P & Y.
\end{tikzcd}
\end{equation}
Moreover, $P$ is the image of the morphism $\P_X(E) \to \P^5$ 
induced by the base point free complete linear system  $|\MO_{\P_X(E)}(1)|$. 
In particular, we have the following induced isomorphisms: 
\begin{equation}\label{e nasi Pfaff diagram E-to-P}
H^0(X, E) \simeq H^0(\P_X(E), \MO_{\P_X(E)}(1)) \simeq H^0(P, \MO_{\P^5}(1)|_P). 
\end{equation}

Take a nonzero element $t \in H^0(P, \MO_{\P^5}(1)|_P)$. 
For its zero locus $Z:=Z(t)$ (which is a hyperplane section of $P$), 
the diagram (\ref{e1 nasi Pfaff diagram}) induces 
the following diagram: 
\begin{equation}\label{e2 nasi Pfaff diagram}
\begin{tikzcd}
a^{-1}(Z) \arrow[rd,"a_H"] \arrow[d]  &   & b^{-1}(Z) \arrow[ld,"b_H"'] \arrow[d] \\
X&  Z & Y.
\end{tikzcd}
\end{equation}
\end{nasi}

\begin{nasi}[Two-ray game diagram obtained from quintic elliptic curves]
Let $X \subset \P^9$ be a prime Fano threefold of genus $8$. 
Let $C$ be a smooth elliptic curve on $X$ such that $-K_X \cdot C =5$ and 
$\langle C \rangle  \cap X = C$. 
For the blowup $\sigma : X' \to X$ along $C$, 
$|-K_{X'}|$ is base point free. 
The image $Z := \varphi_{|-K_{X'}|}(X') \subset \P^4$ of the induced morphism 
$\psi : X' \to Z$ is a normal quartic threefold 
(Proposition \ref{p qell blowup main}) 
and $\psi$ is a flopping contraction 
(Proposition \ref{p-div-cont1}, Proposition \ref{p -K_Y not ample}). 
For its flop $\wt{\psi}: Y' \to Z$, 
we have the contraction $\tau : Y' \to Z$ of the $K_{Y'}$-negative extremal ray. 
To summarise, we obtain the following diagram consisting of birational morphisms: 
\begin{equation}\label{e1 nasi 2ray quintic}
\begin{tikzcd}
X' \arrow[rd,"\psi"] \arrow[d, "\sigma"]  &   & Y' 
\arrow[ld,"\wt{\psi}"'] \arrow[d, "\tau"] \\
X&  Z & Y.
\end{tikzcd}
\end{equation}
Then we can check that $Y \subset \P^4$ is a smooth cubic threefold 
by the standard two-ray-game analysis (cf.\ \cite[the proof of Theorem 8.1]{FanoII}). 


On the other hand, there exist a locally free sheaf $E$ on $X$ of rank $2$, 
a closed embedding $X \subset \Gr(6, 2)$, and a nonzero element $s \in H^0(X, E)$ such that $Z(s) =C$ and $E \simeq \mathcal Q|_X$ (cf.\ Section \ref{s linear section}). 
Then we get the diagram (\ref{e1 nasi Pfaff diagram}). 
Via the isomorphism (\ref{e nasi Pfaff diagram E-to-P}), 
we have the nonzero element $t \in H^0(P, \MO_{\P^5}(1)|_P)$ corresponding to $s$: 
\begin{eqnarray*}
H^0(X, E) &\simeq& H^0(P, \MO_{\P^5}(1)|_P)\\
s & \leftrightarrow& t. 
\end{eqnarray*}
Hence we get the diagram (\ref{e2 nasi Pfaff diagram}). 
We can show that these two diagrams (\ref{e2 nasi Pfaff diagram}) and  (\ref{e1 nasi 2ray quintic}) coincide. 
In particular, the end result (i.e., $Y$ in (\ref{e1 nasi 2ray quintic})) 
of the two-ray game starting from the blowup along $C$ coincides with the smooth cubic threefold of $X$ determined by the closed embedding $X \subset \Gr(6, 2)$. 
\end{nasi}






\begin{rem}
Although $C$ is assumed to be a smooth elliptic curve on $X$, 
the above description of the two-ray game should be extended to the case when 
$C$ is a geometrically integral regular genus-one curve on $X \times_k \kappa$ for a field extension $\kappa/k$. 
\end{rem}

\section{F-splitting and quasi-F-splitting}

By \cite[Theorem 1.2]{SS10} and \cite[Theorem 1.1]{FanoIV}, there exists an integer $p_0 >0$ such that 
an arbitrary smooth Fano threefold of characteristic $p>p_0$ 
is $F$-split. 
On the other hand, any explicit lower bound $p_0$ is not known. 
The purpose of this section is to prove that every prime Fano threefold of genus $8$ is $F$-split when $p>2$.

\begin{lem}\label{l GFS criterion}
Let $X$ be a prime Fano threefold of genus $8$. 
\begin{enumerate}
\item If $H^2(X, \Omega_X^1(-p^rK_X))=0$ for every $r>0$, 
then $X$ is quasi-$F$-split. 
\item  
If $H^2(X, \Omega_X^1(-pK_X))=0$ and $H^1(X, \Omega_X^2(-pK_X))=0$, 
then $X$ is $F$-split. 
\end{enumerate}
\end{lem}

\begin{proof}
It holds that 
$H^1(X, \Omega_X^1(K_X))=0$ \cite[Proposition 4.5]{KTLift1} 
and $H^0(X, \Omega_X^2)=0$ \cite[Proposition 4.1, Lemma 8.2]{KTLift1}. 
Then the assertion (1) (resp. (2)) 
follows from \cite[Proposition 2.20]{KTLift2} (resp. 
\cite[Proposition 2.15]{KTLift2}).  
\end{proof}

\subsection{$H^2(X, \Omega_X^1(-pK_X))=0$ $(p>5)$}\label{ss QFS}

\begin{prop}\label{p Gr Omega1 vanishing}
Take integers $n > m \geq 1$. 
 Then 
 \[
 H^i(\Gr(n,m),\Omega_{\Gr(n,m)}^1(k)) = 0
 \]
 for $i>0$ and $k>0$.
\end{prop}

\begin{proof}
Set $W = \Gr(n,m)$.
 Let
\[
0 \to \cK \to \MO_W^{\oplus n} \to \cQ \to 0
\]
be the universal sequence.
Then $\Omega_W^1 \simeq \cK \otimes \cQ^\vee$.

Consider the following diagram obtained by taking the fibre product $F=  \P_W(\cK) \times_W \P_W(\cQ^\vee)$: 
\[
\begin{tikzcd}
 & F=  \P_W(\cK) \times_W \P_W(\cQ^\vee)  \arrow[ld,"r"'] \arrow[rd,"s"] \arrow[dd, "\pi"]  & \\
 \P_W(\cK) \arrow[rd,"p"'] & &\P_W(\cQ^\vee) \arrow[ld,"q"] \\
  &W&
\end{tikzcd}
\]
Then we have isomorphisms $F \simeq \Fl(n;m+1,m,m-1)$, $\P_W(\cK) \simeq \Fl (n;m+1,m)$ and $\P_W(\cQ^\vee ) \simeq \Fl (n;m,m-1)$.
Moreover $p$, $q$, $r$, $s$ are the natural projections of these flag varieties.

Let $H$ be the generator of $\Pic(W)$, $\xi$ the tautological divisor of $\P_W(\cK)$, and $\eta$ the tautological divisor of $\P_W(\cQ^\vee)$.
Since there are two isomorphisms $\cK(1) \simeq \wedge^{n-m-1} \cK^{\vee}$ and $\cQ^{\vee}(1) \simeq \wedge^{m-1} \cQ$, 
the bundles $\cK(1)$ and $\cQ^{\vee}(1)$ are globally generated. 
Thus 
\begin{itemize}
 \item $\xi + p^* H$ and $\eta +q^* H $ are nef,
 \item $-K_F = (n-m) r^*\xi + m s^*\eta + (n+2) \pi^*H$ 
 (Lemma \ref{l univ bdl det}, Lemma \ref{l P(E) times P(E')}).
\end{itemize}

We have
\begin{align*}
H^i(W,\Omega^1_{W}(k)) &=H^i(W, (\cK\otimes \cQ^\vee)(k)) \\
&\overset{(\star)}{=}H^i(F, r^*\xi +s^*\eta +k \pi^* H) \\
&=H^i(F, K_F+ (n-m+1)r^*\xi + (m+1)s^*\eta +(n+2+k) \pi^* H) \\
&=H^i(F, K_F+ (n-m+1)(r^*\xi +\pi^*H) + (m+1)(s^*\eta +\pi^*H) +k \pi^* H) \\
&\overset{(\star\star)}{=}0. 
\end{align*}
Here $(\star)$ follows from Lemma \ref{l P(E) times P(E')}, 
and 
$(\star\star)$ holds by the fact that 
$F$ is $F$-split and $(n-m+1)(r^*\xi +\pi^*H) + (m+1)(s^*\eta +\pi^*H) +k \pi^* H$ is ample when $k>0$. 
\end{proof}

\begin{lem}\label{l P(E) times P(E')}
Let $W$ be a smooth projective variety. 
Take coherent locally free sheaves $E$ and $E'$ on $W$. 
Set $r := \rank\,E$ and $r':= \rank\,E'$. 
Note that we have the following induced cartesian diagram:  
\[
\begin{tikzcd}
 & V:=  \P_W(E) \times_W \P_W(E')  \arrow[ld,"\rho"'] \arrow[rd,"\rho'"] \arrow[dd, "\theta"] & \\
 \P_W(E) \arrow[rd,"\pi"'] & &\P_W(E') \arrow[ld,"\pi'"] \\
  &W&
\end{tikzcd}
\]
Then the following hold 
for the tautological divisor $\xi$ (resp. $\xi'$) of $\pi$ (resp. $\pi'$). 
\begin{enumerate}
\item $K_V \sim -r\rho^*\xi -r'\rho'^*\xi' +\theta^*(K_W + \det E + \det E')$. 
\item $\theta_*\MO_V(n\rho^*\xi + n'\rho'^*\xi') \simeq S^n(E) \otimes_{\MO_W} S^{n'}(E')$ 
for every $n>0$ and $n'>0$. 
\end{enumerate}
\end{lem}

\begin{proof}
Let us show (1). 
By \cite[Ch. III, Exercise 8.4(b)]{Har77}, we have 
\[
K_{\P_W(E)/W} \sim -r \xi + \pi^*(\det E), \qquad 
K_{\P_W(E')/W} \sim -r' \xi' + \pi^*(\det E'). 
\]
Since dualising sheaves commute with base changes, we have 
$\rho'^*K_{\P_W(E')/W} \sim K_{V/\P_W(E)}$. Hence 
\begin{align*}
K_V 
&\sim \rho^*K_{\P(E)} + \rho'^*K_{\P(E')} -\theta^*K_W\\
&\sim \rho^*( \pi^*K_W + -r \xi + \pi^*(\det E)) 
+\rho'^*(\pi'^* K_W + -r' \xi' + \pi'^*(\det E')) - \theta^*K_W\\
&\sim -r\rho^*\xi -r'\xi' +\theta^*(K_W + \det E + \det E'). 
\end{align*}
Thus (1) holds. 

Let us show (2). 
Since $\rho : V = \P_{\P_W(E)}(\pi^*E') \to \P_W(E)$ 
is the projective space bundle associated with $\pi^*E'$ and 
$\rho'^*\xi'$ is its tautological divisor, 
we get 
$\rho_*\MO_V(n'\rho'^*\xi') \simeq S^{n'}(\pi^*E') \simeq \pi^*S^{n'}(E')$ 
\cite[Ch. II, Proposition 7.11]{Har77}. 
Then it holds that 
\begin{align*}
\theta_*\MO_V(n\rho^*\xi + n'\rho'^*\xi')
&\simeq \pi_*\rho_*(\rho^*\MO_{\MO_{\P_W(E)}}(n\xi) \otimes_{\MO_V}  \MO_V(n'\rho'^*\xi'))\\
&\overset{{\rm (i)}}{\simeq} \pi_*( \MO_V(n\xi)  \otimes_{\MO_{\P_W(E)}} \rho_*\MO_V(n'\rho'^*\xi') )\\
&
\simeq 
\pi_*( \MO_V(n\xi)   \otimes_{\MO_{\P_W(E)}} \pi^*S^{n'}(E') )\\
&\overset{{\rm (ii)}}{\simeq} (\pi_*\MO_V(n\xi)) \otimes_W S^{n'}(E')\\
&\overset{{\rm (iii)}}{\simeq} S^n(E) \otimes_W S^{n'}(E'), 
\end{align*}
where (i) and (ii) follow from the projection formula,  
and \cite[Ch. II, Proposition 7.11]{Har77} implies (iii).  
Thus (2) holds. 
\end{proof}

\begin{thm}\label{t g=8 QFS}
Let $X$ be a prime Fano threefold of genus $8$. 
Then 
the following hold. 
\begin{enumerate}
\item 
$H^{>0}(X, \Omega_X^1(-nK_X))=0$ for every integer $n \geq 6$. 
\item 
If $p>5$, then $X$ is quasi-$F$-split. 
\end{enumerate}
\end{thm}

\begin{proof}
By Lemma \ref{l GFS criterion}, (1) implies (2). 
Let us show (1). 
Recall that $X \subset \Gr(6, 2)$ and $X = \Gr(6, 2) \cap H_1 \cap \cdots \cap H_5$. 
Set $X_8 := \Gr(6, 2)$ and 
$X_d := \Gr(6, 2) \cap H_1 \cap \cdots \cap H_{8-d}$ for every 
$3 \leq d \leq 7$. 
We have $\dim X_d = d$ and $X_3=X$. 


Note that, since $X \subset X_8$ is a linear section, the conormal bundle $N^*_{X_3/X_8}$ is isomorphic to $\MO_{X_3}(-1)^{\oplus 5}$ \cite[B.7.4]{Ful98}. 
By $N^*_{X/X_8} \simeq \MO_{X}(-1)^{\oplus 5}$, 
we have the following conormal exact sequence:  
\[
0 \to \MO_{X}(-1)^{\oplus 5} \to \Omega^1_{X_8}|_{X} \to \Omega_{X}^1 \to 0. 
\]
Since $H^{>0}(X,\MO_{X}(n-1))=0$ for $n>0$, it is enough to show the following: 
\[
\text{$H^{>0}(X, \Omega^1_{X_8}(n)|_{X}) = 0$ for every $n\geq 6$.}
\]
Since the ideal sheaf $I_{X_d \subset X_{d+1}}$ on $X_{d+1}$ defining $X_d$ is isomorphic to $\MO_{X_{d+1}}(-1)$, we have an exact sequence 
\[
0\to \Omega^1_{X_8}(-1)|_{X_{d+1}}  
\to \Omega^1_{X_8}|_{X_{d+1}} \to \Omega^1_{X_8}|_{X_d}  \to 0. 
\]
By $H^{>0}(X_8, \Omega^1_{X_8}(n))=0$ for $n \geq 1$ (Proposition \ref{p Gr Omega1 vanishing}), 
we get $H^{>0}(X_7, \Omega^1_{X_8}(n)|_{X_7})=0$ for $n \geq 2$. 
Applying this argument for $d = 6$, 
we obtain $H^{>0}(X_6, \Omega^1_{X_8}(n)|_{X_6})=0$ for $n \geq 3$. 
Repeating this procedure, 
it holds that $H^{>0}(X_3, \Omega^1_{X_3}(n))=0$ for $n \geq 6$. 
Thus (1) holds.
\end{proof}

\subsection{$H^1(X, \Omega_X^2(-pK_X))=0$ $(p>5)$}

\begin{lem}\label{l P(wedge^2K) GFR}
Set $W := \Gr(6, 2)$ and let $\cK$ be the universal subbundle on $W= \Gr(6, 2)$. 
Then the following hold for $ \pi \colon \P := \P_{W}(\wedge^2 \mathcal K) \to W$, the tautological divisor $\xi$ on $\P$ of $\pi$, and 
a hyperplane section $H$ on $W$. 
\begin{enumerate}
\item 
There exists a birational morphism $\varphi : \P_W(\wedge^2 \cK) \to \cC$ to the Pfaffian cubic $\cC \subset \P^{14}$ such that $\varphi_*\MO_{\P_W(\wedge^2 \cK)} = \MO_{\cC}$. 
\item 
For $F := \Ex(\varphi)$, it holds that $F \simeq {\rm Fl}(6; 4, 2)$ and $F \simeq {\rm Fl}(6; 4, 2) \to W = \Gr(6, 2)$ is the natural projection. 
\item 
$-K_{\P} \sim 6\xi + 9\pi^*H$. 
\item 
$-(K_{\P}+F) \sim 4\xi+8\pi^*H$. 
\end{enumerate}
\end{lem}

\begin{proof}
Set $U_6 := k^6$, and consider $W$ as $\Gr(U_6,2)$.
Note that $\wedge^2 \cK (1) \simeq \wedge^2 \cK^{\vee}$.
By applying $\wedge^2$ to 
the surjection $\MO_{\Gr(6,2)}\otimes_k U_6^\vee  \to \cK^\vee$, we have a surjection $\MO_{\Gr(6,2)}\otimes_k \wedge^2 U_6^\vee \to \wedge^2 \cK^\vee$.
This induces a morphism 
$\psi \colon \P_W(\wedge^2 \cK) \simeq \P_W(\wedge^2 \cK^\vee) \to \P(\wedge^2 U_6^\vee)$ such that $\psi^*\MO_{\P(\wedge^2 U_6^\vee)}(1) 
\simeq \MO_{\P_W(\wedge^2 \cK)}(\xi + \pi^*H)$. 
By the construction, the image of $\psi$ is the Pfaffian cubic hypersurface $\cC$. 
Hence we have the following induced morphisms: 
\[
\psi \colon  \P_W(\wedge^2 \cK) \xrightarrow{\varphi}  \cC \hookrightarrow \P(\wedge^2 U_6^\vee). 
\]
By the standard Schubert calculus, we see that $(\xi + \pi^* H)^{13} = 3$.
Thus $\varphi \colon  \P_W(\wedge^2 \cK) \to \cC$ is birational.
Note that $\Sing (\cC) = \Gr(2,U_6)$, because 
 there are only two $\PGL_6$-orbits on $\cC$. 
Since $\dim \Sing(\cC) = \dim \Gr(2,6) = 8$, the hypersurface $\cC$ is normal.
Thus $\varphi_* \MO_{\P_W(\wedge^2 \cK)} = \MO_{\cC}$, and (1) follows.

Let us show (2). 
By taking the relative Pl\"ucker embedding, we have a closed immersion  $\Fl(U_6;4,2) \hookrightarrow \P_W(\wedge^2\cK)$ such that the morphism  $\varphi$ restricts to the natural projection $\Fl(U_6;4,2) \to \Gr(U_6,4) \simeq \Gr(2,U_6)$ (by the construction).
Thus $F$ contains $\Fl(U_6;4;2)$.
Since the Picard rank of $\P_W(\wedge^2\cK)$ is two, $\varphi$ is a contraction of an extremal ray.
Since the excetpional locus of a divisorial  contraction 
is a prime divisor 
(see, e.g., \cite[Proposition 3.2]{Kol91}), we have $F = \Fl(U_6;4,2)$. 
Thus  (2) holds. 

Since $c_1(\wedge^2 \cK) = -3H$ (see, e.g., \cite[Appendix A, \S 3]{Har77}), we have
\[
-K_{\P} = 6\xi + \pi^*(3H) + \pi^*(-K_W) = 6\xi + 9 \pi^* H. 
\]
Thus (3) holds. 

Finally, we show (4).
Let $f_1 \colon \Fl(U_6;4,2) \to \Gr(U_6,2)$ and $f_2 \colon \Fl(U_6;4,2) \to \Gr(U_6,4)$ be the natural projections.
Note that the $f_1$-fibres are isomorphic to $\Gr(4,2) \simeq \Q^4$, and that $f_2$ restricts to the Pl\"ucker embeddings to these fibres. 
We can write $-K_F = a f_1^*\MO_{\Gr(U_6,2)}(1) + b f_2^* \MO_{\Gr(U_6,4)}(1)$ for some $a, b \in \Z$.
Then, on an $f_1$-fibre $\Gr(4,2) \simeq \Q^4$, we have
\[
\MO_{\Q^4}(4) \simeq -K_F|_{\Gr(4,2)} = (a f_1^*\MO_{\Gr(U_6,2)}(1) + b f_2^* \MO_{\Gr(U_6,4)}(1))|_{\Gr(4,2)} = \MO_{\Q^4}(b).
\]
Hence $b=4$.
By symmetry, we have $a=4$.
Thus
\[
(-K_\P-F)|_F = -K_F
=4 f_1^*\MO_{\Gr(U_6,2)}(1) +4 f_2^* \MO_{\Gr(U_6,4)}(1))
=(4\pi^*H+4(\xi+\pi^*H))|_F.
\]
Since $\Pic \P \to \Pic F$ is injective, we have (4).
\end{proof}

\begin{prop}\label{p PtimesP over W GFS}
$\P_W(\wedge^2\cK) \times_W \P_W(\cQ^\vee)$ is a globally $F$-regular 
smooth Fano variety. 
\end{prop}

\begin{proof}
Set $W := \Gr(6, 2)$, $\P := \P_W(\wedge^2\cK)$, $\Q := \P_W(\cQ^{\vee})$, $\R := \P \times_W \Q$. 
We have the following cartesian diagram: 
\[
\begin{tikzcd}
 & \R=  \P \times_W \Q  \arrow[ld,"a"'] \arrow[rd,"b"] \arrow[dd, "e"]& \\
 \P = \P_W(\wedge^2\cK) \arrow[rd,"c"'] & &\Q = \P_W(\cQ^\vee) \arrow[ld,"d"] \\
  &W.&
\end{tikzcd}
\]
Recall that 
$ c : \P \to W$ and $\varphi : \P \to \mathcal C$ are the contractions of $\P$, 
where $\mathcal C \subset \P^{14}$ is the Pfaffian cubic (Lemma \ref{l P(wedge^2K) GFR}). 
Set $F := \Ex(\varphi) = \text{Fl}(6; 4, 2)$ and $F' :=a^{-1}(F) = \text{Fl}(6; 4, 2, 1)$. 
In particular, $F'$ is $F$-split. 

Since dualising sheaves commute with base changes, 
we have $a^* \omega_{\P/W} \simeq \omega_{\R/\Q}$. 
Hence the following hold: 
\[
-K_{\R} = -a^* K_{\P} -b^*K_{\Q} + e^*K_W, 
\]
\[
-(K_{\R} + F') = -a^*(K_{\P} +F) -b^*K_{\Q/W}. 
\]
Let $H_W$ be the ample generator of $\Pic W$. 
For the  tautological bundle $\eta$ of the $\P^1$-bundle $\Q = \P_W(\cQ^\vee) \to W$, it holds that 
$K_{\Q} \sim -2 \eta + d^*(K_W + \det (\cQ^\vee))$, and hence 
$-K_{\Q/W}=2 \eta + d^*H_W$ (Lemma \ref{l univ bdl det}). 
It holds that 
\[
-K_{\R} = -a^*K_{\P} -b^*K_{\Q/W} 
\overset{(\star)}{=} 6a^*\xi + 9e^*H_W   +2b^*\eta +e^*H_W 
\]
\[
=
(6a^*\xi + 7 e^*H_W) + (2b^*\eta +3e^*H_W) = a^*A_{\P} + b^*A_{\Q},  
\]
for $A_{\P} := 6\xi + 7 c^*H_W$ and $A_{\Q} := 2\eta +3d^*H_W$, and 
\[
-(K_{\R} +F') =  -a^*(K_{\P} +F) -b^*K_{\Q/W} 
\overset{(\star\star)}{=} 4a^*\xi + 8e^*H_W   +2b^*\eta +e^*H_W
\]
\[
= (4a^*\xi + 5e^*H_W) + (2b^*\eta + 4e^*H_W)  = a^*A'_{\P} + b^*A'_{\Q} 
\]
for  $A'_{\P} := 4\xi + 5 c^*H_W$ and $A'_{\Q} := 2\eta +4d^*H_W$, 
where $(\star)$ and $(\star\star)$ follow from 
$-K_{\Q/W}=2 \eta + d^*H_W$ and 
Lemma \ref{l P(wedge^2K) GFR}. 
Note that $A_{\P}$ and $A'_{\P}$ (resp. $A_{\Q}$ and $A'_{\Q}$) 
are ample divsiors on $\P$ (resp. $\Q$). 
Then each of $-K_{\R}$ and $-(K_{\R}+F')$ is ample, 
because we have the induced closed immersion 
$\R = \P \times_W \Q \hookrightarrow \P \times_k \Q$ and 
the corresponding divisors 
$\pr_1^*A_{\P} + \pr_2^*A_{\Q}$ and $\pr_1^*A'_{\P} + \pr_2^*A'_{\Q}$ 
on $\P \times_k \Q$ are ample. 
By \cite[Lemma 2.7]{CTW17}, $\R$ is $F$-split. 
Since $\R$ is a smooth Fano variety, $\R$ is globally $F$-regular. 
\end{proof}

\begin{rem}
In particular, the Pfaffian cubic $\mathcal C \subset \P^{14}$ is globally $F$-regular. 
\end{rem}
\begin{prop}\label{p Gr Omega2 vanish}
If $p\neq 2$, 
then 
\[
H^i(\Gr(6, 2), \Omega^2_{\Gr(6, 2)}(k)) =0 
\]
 for $i>0$ and  $k >0$. 
\end{prop}

\begin{proof}
Assume $p\neq 2$. 
Set $W := \Gr(6, 2)$. 
By $p \neq 2$ and $\Omega_W^1 \simeq \cK \otimes \cQ^\vee$, we have
\[
\Omega^2_W  = \wedge^2 \Omega^1_W  \simeq \wedge^2(\cK \otimes \cQ^\vee) \simeq S^2\cK \otimes \wedge^2 (Q^\vee) \oplus \wedge^2\cK \otimes S^2 (Q^\vee).
\]
Hence
\[
H^i(W, \wedge^2 \Omega_W(k)) \simeq H^i(W, S^2\cK \otimes \wedge^2 (Q^\vee) (k)) \oplus H^i(W, \wedge^2\cK \otimes S^2 (Q^\vee)(k)).
\]

We have 
\[
K_{\P_W(\cK)} = -2\xi_{\cK} +\pi_{\cK}^*(K_W + \det (\cK)) = 
 -2\xi_{\cK} -7\pi_{\cK}^*H_W, 
\]
where $\xi_{\cK}$ denotes the tautological divisor 
of the $\P^3$-bundle $\pi_{\cK}:\P_W(\cK) \to W$. 
Then the following divisor is ample for every $k > 0$: 
\[
2\xi_{\cK} + (k-1) \pi_{\cK}^*H_W - K_{\P_W(\cK)} 
= 4\xi_{\cK} + (k+6)\pi_{\cK}^*H_W. 
\]
Since $\P_W(\cK)$ is $F$-split, 
$\P_W(\cK)$ satisfies the Kodaira vanishing, and hence 
\[
 H^{>0}(W, S^2\cK \otimes \wedge^2 (Q^\vee) (k)) 
 \simeq H^{>0}(\P_W(\cK), \MO_{\P_W(\cK)}(2\xi_{\cK} + (k-1)\pi^*_{\cK}H_W)) =0
\]
for every $k >0$.

In what follows, we use the same notation as in the proof of 
Proposition \ref{p PtimesP over W GFS}. 
We have 
\[
-K_{\R} = 6a^*\xi + 2b^*\eta +10e^*H_W. 
\]
Then the following divisor is ample for every integer $k \geq 1$: 
\begin{align*}
&a^*\xi + 2b^*\eta + ke^*H_W -K_{\R} \\
=& 
6a^*\xi + 2b^*\eta +10e^*H_W+ a^*\xi + 2b^*\eta + ke^*H_W\\
=& 
(7a^*\xi + (7+\frac{1}{2}) e^*H_W) 
+ 
(3b^*\eta + (3+ \frac{1}{2})e^*H_W) + (k-1)e^*H_W. 
\end{align*}
Since the Kodaira vanishing holds on $\R$ 
(Proposition \ref{p PtimesP over W GFS}), it holds that 
\[
H^{>0}(W, \wedge^2\cK \otimes S^2 (Q^\vee)(k))
\simeq H^{>0}(\R, a^*\xi + 2b^*\eta + ke^*H_W)=0
\]
for every integer $k\geq 1$. 
\end{proof}

\begin{thm}\label{t g=8 Fsplit}
Let $X$ be a prime Fano threefold of genus $8$. 
\begin{enumerate}
\item 
If $p>2$, then $H^i(X, \Omega_X^2(-nK_X))=0$ for every $i>0$ 
and $n \geq 7$. 
\item If $p>5$, then $X$ is $F$-split. 
\end{enumerate}
\end{thm}

\begin{proof}
By Lemma \ref{l GFS criterion} and Theorem \ref{t g=8 QFS}, 
(1) implies (2). 
Let us show (1). 
Assume $p>2$. 
Set $W := \Gr(6, 2)$. 
By applying  $\wedge^2$ 
to the conormal exact sequence 
$0 \to \MO_X(-1)^{\oplus 5} \to \Omega_W^1|_X \to \Omega_X^1 \to 0$ \cite[B.7.4]{Ful98}, 
we get the following exact sequences for some coherent locally free sheaf $F$ on $X$: 
\[
0 \to F \to \Omega_W^2|_X \to \Omega_X^2 \to 0, 
\]
\[
0 \to \wedge^2(\MO_X(-1)^{\oplus 5}) \simeq \MO_X(-2)^{\oplus 10} \to F \to \Omega_X^1 \otimes  \MO_X(-1)^{\oplus 5} \to 0. 
\]
Since $H^{>0}(X, \MO_X(n-2)) =0$ for every integer $n \geq 2$, it is enough to show  (I) and (II) below. 
\begin{enumerate}
\item[(I)] $H^{>0}(X, \Omega_X^1 (n-1)) =0$ 
for every integer $n \geq 7$. 
\item[(II)] $H^{>0}(X, \Omega_W^2(n)|_X)=0$ for every $n \geq 7$. 
\end{enumerate}
(I) follows from Theorem \ref{t g=8 QFS}. 
Let us show (II). 
By $H^{>0}(W, \Omega_W^2(n))=0$ for every $n\geq 1$ (Proposition \ref{p Gr Omega2 vanish}), 
we get $H^{>0}(X, \Omega_W^2(n)|_X)=0$ for every $n \geq 6$ 
(use the same argument as in the proof of Theorem \ref{t g=8 QFS}).
\end{proof}

\subsection{Computer-assisted statements}\label{ss Macaulay2}


\begin{prop}\label{p Macaulay2 Gr}
Set $W := \Gr(6, 2)$. 
Then the following hold:
\begin{enumerate}
 \item Assume that $p=2$, $3$, or $5$.
 For $i= -3$, \dots, $2$ and $k \geq i-1$, the cohomology group $H^{4-i}(W, \Omega^1_W (k))$ is zero.
 \item Assume that $p=2$, $3$, or $5$.
 For $i= -2$, \dots, $3$ and $k \geq i$, the cohomology groups $H^{4-i}(W, S^2\cK \otimes \wedge^2 (\cQ^\vee) (k))$ and $H^{4-i}(W, \wedge^2\cK \otimes S^2 (\cQ^\vee) (k))$ are zero.
 \item Assume that $p=3$, or $5$.
 For $i= -2$, \dots, $3$ and $k \geq i$, the cohomology group $H^{4-i}(W, \Omega^2_W (k))$ is zero.
 \end{enumerate}
\end{prop}

\begin{proof}
As we saw in Proposition~\ref{p Gr Omega2 vanish}, if $p\neq 2$, then we have
\[
H^i(W, \wedge^2 \Omega^1_W(k)) \simeq H^i(W, S^2\cK \otimes \wedge^2 (Q^\vee) (k)) \oplus H^i(W, \wedge^2\cK \otimes S^2 (Q^\vee)(k)).
\]
Thus the last assertion follows from the former assertions.

The first and the second assertions can be checked by using computer algebra systems, e.g., Macaulay2 \cite{M2}.
For instance, the following computes the cohomology groups when $p=2$.
Note that we have isomorphisms $S^2\cK \otimes \wedge^2 (\cQ^\vee) \simeq S^2\cK(-1)$ and $\wedge^2\cK \simeq (\wedge^2 (\cK^\vee))(-1)$.
Also, by letting $p=3$ or $5$, we have the assertions for the remaining cases.
\begin{lstlisting}
p=2;
I=Grassmannian(1,5,CoefficientRing=>ZZ/p,Variable=>x);
R=ring I;
G=R/I;
M=matrix{{0,x_(0,1),x_(0,2),x_(0,3),x_(0,4),x_(0,5)},{-x_(0,1),0,x_(1,2),x_(1,3),x_(1,4),x_(1,5)},{-x_(0,2),-x_(1,2), 0,x_(2,3),x_(2,4),x_(2,5)},{-x_(0,3),-x_(1,3),-x_(2,3),0,x_(3,4),x_(3,5)},{-x_(0,4),-x_(1,4),-x_(2,4),-x_(3,4),0,x_(4,5)},{-x_(0,5),-x_(1,5),-x_(2,5),-x_(3,5),-x_(4,5),0}};
K=prune dual sheaf cokernel M;
Q=prune dual sheaf image M;
Omega=prune (K ** (dual Q));
sym2K=prune symmetricPower(2,K);
wedge2K=prune((exteriorPower(2,prune dual K))(-1));
sym2Qdual=prune symmetricPower(2,dual Q);
F1=prune(sym2K(-1));
F2=prune(wedge2K ** sym2Qdual);
apply(-3..2, i -> prune truncate(i,HH^(4-i)(Omega(>=i-1))))
apply(-2..3, i -> prune truncate(i,HH^(4-i)(F1(>=i))))
apply(-2..3, i -> prune truncate(i,HH^(4-i)(F2(>=i))))
\end{lstlisting}

\end{proof}

\begin{thm}\label{t Fsplit optimal}
Let $X$ be a prime Fano threefold of genus $8$. 
Then $X$ is quasi-$F$-split. 
Moreover, if $p>2$, then $X$ is $F$-split. 
\end{thm}

\begin{proof}
By the same argument as in Theorem \ref{t g=8 QFS}, the first assertion is reduced to prove
\[
H^2(W,\Omega_W^1(p^r)) =H^3(W,\Omega_W^1(p^r-1)) = \cdots = H^7(W,\Omega_W^1(p^r-5)) = 0,
\]
which follow from Proposition~\ref{p Macaulay2 Gr}.

Similarly, by the same argument as in Theorem \ref{t g=8 Fsplit}, the second assertion is reduced to prove
\begin{align*}
& H^2(W,\Omega_W^1(p-1)) =H^3(W,\Omega_W^1(p-2)) = \cdots = H^7(W,\Omega_W^1(p-6)) = 0,\\
& H^1(W,\Omega_W^2(p)) =H^2(W,\Omega_W^2(p-1)) = \cdots = H^6(W,\Omega_W^2(p-5)) = 0,
\end{align*}
which also follow from Proposition~\ref{p Macaulay2 Gr}.
\end{proof}

\begin{cor}\label{c g=8 MukaiVar Fsplit}
Assume $p>2$. 
Let $\Gr(6, 2)  \subset \P^{14}$ be the Pl\"{u}cker embedding. 
Take a linear subvariety $L \subset \P^{14}$ such that 
$9 \leq \dim L \leq 14$ and 
that 
$Y := \Gr(6, 2) \cap L$ is a smooth projective variety 
with $\dim Y = \Gr(6, 2) + 14 -\dim L$ (in particular, 
$Y$ is a complete intersection of $\Gr(6, 2)$ and hyperplane sections on $\P^{14}$). 
Then $Y$ is globally $F$-regular. 
\end{cor}

\begin{proof}
For $m:= \dim Y -3$ and 
general hyperplanes $H_1, ..., H_m$ on $\P^{14}$, 
the intersection $X := Y \cap H_1 \cap \cdots \cap H_m$ is a prime Fano threefold of genus $8$, and hence $X$ is globally $F$-split 
(Theorem \ref{t Fsplit optimal}). 
By \cite[Lemma 2.7]{CTW17}, $Y$ is $F$-split. 
Since $Y$ is a smooth Fano variety, $Y$ is globally $F$-regular. 
\end{proof}


\begin{ex}\label{e Klein cubic}
Let $Y \subset \P^4_{(x,y,z,v,w)}$ be the Klein cubic threefold  over $k$ defined by
 \[
 x^2y+y^2z+z^2v+v^2w+w^2x = 0.
 \]
$Y$ is defined by the Pfaffian of the following alternating matrix 
\cite[\S 47]{AR96}, 
\cite[Theorem 2.6]{GP01}: 
\[A=
\begin{pmatrix}
0&v&w&x&y&z\\
-v&0&0&z&-x&0\\
-w&0&0&0&v&-y\\
-x&-z&0&0&0&w\\
-y&x&-v&0&0&0\\
-z&0&y&-w&0&0
\end{pmatrix}
\]

This Pfaffian representation of $Y$ corresponds to a Fano $3$-fold $X$ of genus $8$ as follows:
Let $W = \Gr(U_6,2) \subset \P(\wedge^2 U_6) \simeq \P^{14}$  be the $8$-dimensional Grassmann variety.
With respect to the Pl\"ucker coordinates $\{x_{i,j}\}_{0 \leq i<j \leq 5}$, $W \subset \P^{14}$ is defined by $15$ Pfaffian quadratic equations:
\[
x_{ij}x_{kl}-x_{ik}x_{jl}+x_{il}x_{jk} \quad (0\leq i<j<k<l \leq 5).
\]
Then the alternating matrix $A$ defines $5$ linear equations as follows:
\[
\begin{cases}
 x_{0,3}-x_{1,4}=0,\\
 x_{0,4}-x_{2,5}=0,\\
 x_{1,3}+x_{0,5}=0,\\
 x_{0,1}+x_{2,4}=0,\\
 x_{0,2}+x_{3,5}=0.
\end{cases}
\]
If $p \neq 11$, the linear section $X$ of $W$ with respect to these linear equations defines a 
prime Fano threefold of genus $8$. 
Here the smoothness of $X$ can be checked as follows. 
By Jacobian criterion, $Y$ is smooth when $p \neq 11$. 
Then the same argument as in Proposition \ref{p cubic of X is smooth}(1) 
implies that $X$ is smooth (alternatively, we can directly check the smoothness of $X$ by using Macaulay2 \cite{M2}). 
\end{ex}


\begin{prop}\label{p p=2 non-F-split}
We use the same notation as in Example \ref{e Klein cubic}. 
If $p =2$, then $X$ is not $F$-split.
\end{prop}

\begin{proof}
Assume $p=2$ and set $R = k[ \{ x_{ij} \,|\, 0 \leq i < j \leq 5\}]$ (15 variables).
Set 
\[
I_W = (x_{ij}x_{kl}+x_{ik}x_{jl}+x_{il}x_{jk} \,|\, 0 \leq i<j<k<l\leq 5),
\]
\[
I_X = (x_{0,3}-x_{1,4}, x_{0,4}-x_{2,5}, x_{1,3}+x_{0,5}, x_{0,1}+x_{2,4}, x_{0,2}+x_{3,5}).
\]
Then $Q = R/I_X$ is isomorphic to a polynomial ring with 10 variables, and the ideal $J = I_W / I_X$ defines our variety $X$ in $\P^9$, i.e., $Q/J$ is the homogeneous coordinate ring of $X \subset \P^9$.
It is enough to show $(J^{[2]}:J) \subset \mathfrak{m}^{[2]}$ 
\cite[Proposition 1.7]{Fed83}, 
\cite[Proposition 3.1]{Smi00}.
This can be checked by using Macaulay2 \cite{M2} as follows:
\begin{lstlisting}[caption=,label=]
p=2; 
IW=Grassmannian(1,5,CoefficientRing=>ZZ/p,Variable=>x);   
R=ring IW; 
IX=ideal(x_(0,3)-x_(1,4),x_(0,4)-x_(2,5),x_(1,3)+x_(0,5),x_(0,1)+x_(2,4),x_(0,2)+x_(3,5));
T=minimalPresentation (R/(IW+IX));
Q=ambient T;
J=ideal T;
colon=quotient(J^[2],J);
m=ideal(vars Q);
isSubset(colon,m^[2])
\end{lstlisting}
\end{proof}

\bibliographystyle{skalpha}
\bibliography{reference.bib}

@book {AR96,
    AUTHOR = {Adler, Allan and Ramanan, S.},
     TITLE = {Moduli of abelian varieties},
    SERIES = {Lecture Notes in Mathematics},
    VOLUME = {1644},
 PUBLISHER = {Springer-Verlag, Berlin},
      YEAR = {1996},
     PAGES = {vi+196},
      ISBN = {3-540-62023-0},
   MRCLASS = {14K10 (11G15 14K25)},
  MRNUMBER = {1621185},
MRREVIEWER = {H.\ Lange},
       DOI = {10.1007/BFb0093659},
       URL = {https://doi-org.kyoto-u.idm.oclc.org/10.1007/BFb0093659},
}

@article{BKM24,
  author =        {Bayer, Arend and Kuznetsov, Alexander and Macrì, Emanuele},
  journal =       {arXiv:2402.07154},
  title =         {Mukai bundles on {F}ano threefolds},
  year =          {2024},
}

@article{BKM25,
  author =        {Bayer, Arend and Kuznetsov, Alexander and Macrì, Emanuele},
  journal =       {arXiv:2501.16157},
  title =         {Mukai models of {F}ano varieties},
  year =          {2025},
}

@article {BTLM97,
    AUTHOR = {Buch, Anders and Thomsen, Jesper F. and Lauritzen, Niels and
              Mehta, Vikram},
     TITLE = {The {F}robenius morphism on a toric variety},
   JOURNAL = {Tohoku Math. J. (2)},
  FJOURNAL = {The Tohoku Mathematical Journal. Second Series},
    VOLUME = {49},
      YEAR = {1997},
    NUMBER = {3},
     PAGES = {355--366},
      ISSN = {0040-8735},
   MRCLASS = {14F17 (14M15 14M25)},
  MRNUMBER = {1464183},
MRREVIEWER = {Konrad Drechsler},
       DOI = {10.2748/tmj/1178225109},
       URL = {https://doi-org.utokyo.idm.oclc.org/10.2748/tmj/1178225109},
}

@article {CTW17,
    AUTHOR = {Cascini, Paolo and Tanaka, Hiromu and Witaszek, Jakub},
     TITLE = {On log del {P}ezzo surfaces in large characteristic},
   JOURNAL = {Compos. Math.},
  FJOURNAL = {Compositio Mathematica},
    VOLUME = {153},
      YEAR = {2017},
    NUMBER = {4},
     PAGES = {820--850},
      ISSN = {0010-437X,1570-5846},
   MRCLASS = {14E30 (13A35 14F17 14J45)},
  MRNUMBER = {3621617},
MRREVIEWER = {Sho\ Tanimoto},
       DOI = {10.1112/S0010437X16008265},
       URL = {https://doi-org.kyoto-u.idm.oclc.org/10.1112/S0010437X16008265},
}

@article{Ciu,
  author =        {Ciurca, Tudor},
  journal =       {preprint available at arXiv:2409.15580v1},
  title =         {Prym varieties and cubic threefolds over $\mathbb{Z}$},
  year =          {2024},
}

@article {CG72,
    AUTHOR = {Clemens, C. Herbert and Griffiths, Phillip A.},
     TITLE = {The intermediate {J}acobian of the cubic threefold},
   JOURNAL = {Ann. of Math. (2)},
  FJOURNAL = {Annals of Mathematics. Second Series},
    VOLUME = {95},
      YEAR = {1972},
     PAGES = {281--356},
      ISSN = {0003-486X},
   MRCLASS = {14J10 (14G13 14J05 14K20 14M20 14N99)},
  MRNUMBER = {302652},
MRREVIEWER = {H.\ Popp},
       DOI = {10.2307/1970801},
       URL = {https://doi.org/10.2307/1970801},
}

@incollection {EH87,
    AUTHOR = {Eisenbud, David and Harris, Joe},
     TITLE = {On varieties of minimal degree (a centennial account)},
 BOOKTITLE = {Algebraic geometry, {B}owdoin, 1985 ({B}runswick, {M}aine,
              1985)},
    SERIES = {Proc. Sympos. Pure Math.},
    VOLUME = {46},
     PAGES = {3--13},
 PUBLISHER = {Amer. Math. Soc., Providence, RI},
      YEAR = {1987},
   MRCLASS = {14J40 (14J26)},
  MRNUMBER = {927946},
MRREVIEWER = {Allen B. Altman},
       DOI = {10.1090/pspum/046.1/927946},
       URL = {https://doi.org/10.1090/pspum/046.1/927946},
}

@book {EH16,
    AUTHOR = {Eisenbud, David and Harris, Joe},
     TITLE = {3264 and all that---a second course in algebraic geometry},
 PUBLISHER = {Cambridge University Press, Cambridge},
      YEAR = {2016},
     PAGES = {xiv+616},
      ISBN = {978-1-107-60272-4; 978-1-107-01708-5},
   MRCLASS = {14-01 (14C15 14M15 14N10)},
  MRNUMBER = {3617981},
MRREVIEWER = {Arnaud\ Beauville},
       DOI = {10.1017/CBO9781139062046},
       URL = {https://doi-org.kyoto-u.idm.oclc.org/10.1017/CBO9781139062046},
}

@book {EKM08,
    AUTHOR = {Elman, Richard and Karpenko, Nikita and Merkurjev, Alexander},
     TITLE = {The algebraic and geometric theory of quadratic forms},
    SERIES = {American Mathematical Society Colloquium Publications},
    VOLUME = {56},
 PUBLISHER = {American Mathematical Society, Providence, RI},
      YEAR = {2008},
     PAGES = {viii+435},
      ISBN = {978-0-8218-4329-1},
   MRCLASS = {11Exx (11-02 11E04 11E81 14C15 14C25)},
  MRNUMBER = {2427530},
MRREVIEWER = {Andrzej\ S\l adek},
       DOI = {10.1090/coll/056},
       URL = {https://doi-org.kyoto-u.idm.oclc.org/10.1090/coll/056},
}

@article{Fan30,
 author = {Fano, G.},
 title = {Sulle sezioni spaziali della variet{\`a} grassmanniana delle rette dello spazio a cinque dimensioni.},
 fjournal = {Atti della Accademia Nazionale dei Lincei, Rendiconti, VI. Serie},
 journal = {Atti Accad. Naz. Lincei, Rend., VI. Ser.},
 issn = {0001-4435},
 volume = {11},
 pages = {329--335},
 year = {1930},
 language = {Italian},
 zbMATH = {2564747},
 JFM = {56.0575.03}
}

@article {Fed83,
    AUTHOR = {Fedder, Richard},
     TITLE = {{$F$}-purity and rational singularity},
   JOURNAL = {Trans. Amer. Math. Soc.},
  FJOURNAL = {Transactions of the American Mathematical Society},
    VOLUME = {278},
      YEAR = {1983},
    NUMBER = {2},
     PAGES = {461--480},
      ISSN = {0002-9947,1088-6850},
   MRCLASS = {13H10 (13D03 14B05)},
  MRNUMBER = {701505},
MRREVIEWER = {D.\ Kirby},
       DOI = {10.2307/1999165},
       URL = {https://doi-org.kyoto-u.idm.oclc.org/10.2307/1999165},
}

@book {Fri98,
    AUTHOR = {Friedman, Robert},
     TITLE = {Algebraic surfaces and holomorphic vector bundles},
    SERIES = {Universitext},
 PUBLISHER = {Springer-Verlag, New York},
      YEAR = {1998},
     PAGES = {x+328},
      ISBN = {0-387-98361-9},
   MRCLASS = {14J60 (14-01 14J15 32J15 57R55)},
  MRNUMBER = {1600388},
MRREVIEWER = {I.\ Dolgachev},
       DOI = {10.1007/978-1-4612-1688-9},
       URL = {https://doi-org.kyoto-u.idm.oclc.org/10.1007/978-1-4612-1688-9},
}

@book {Fuj90,
    AUTHOR = {Fujita, Takao},
     TITLE = {Classification theories of polarized varieties},
    SERIES = {London Mathematical Society Lecture Note Series},
    VOLUME = {155},
 PUBLISHER = {Cambridge University Press, Cambridge},
      YEAR = {1990},
     PAGES = {xiv+205},
      ISBN = {0-521-39202-0},
   MRCLASS = {14C20 (14J40 14J60)},
  MRNUMBER = {1162108},
MRREVIEWER = {Elvira Laura Livorni},
       DOI = {10.1017/CBO9780511662638},
       URL = {https://doi-org.utokyo.idm.oclc.org/10.1017/CBO9780511662638},
}

@book {Ful98,
    AUTHOR = {Fulton, William},
     TITLE = {Intersection theory},
    SERIES = {Ergebnisse der Mathematik und ihrer Grenzgebiete. 3. Folge. A
              Series of Modern Surveys in Mathematics [Results in
              Mathematics and Related Areas. 3rd Series. A Series of Modern
              Surveys in Mathematics]},
    VOLUME = {2},
   EDITION = {Second},
 PUBLISHER = {Springer-Verlag, Berlin},
      YEAR = {1998},
     PAGES = {xiv+470},
      ISBN = {3-540-62046-X; 0-387-98549-2},
   MRCLASS = {14C17 (14-02)},
  MRNUMBER = {1644323},
       DOI = {10.1007/978-1-4612-1700-8},
       URL = {https://doi-org.kyoto-u.idm.oclc.org/10.1007/978-1-4612-1700-8},
}

@Misc{M2,
          author = {Grayson, Daniel R. and Stillman, Michael E.},
          title = {Macaulay2, a software system for research in algebraic geometry},
          howpublished = {Available at \url{http://www2.macaulay2.com}}
        }

@article {GP01,
    AUTHOR = {Gross, Mark and Popescu, Sorin},
     TITLE = {The moduli space of {$(1,11)$}-polarized abelian surfaces is
              unirational},
   JOURNAL = {Compositio Math.},
  FJOURNAL = {Compositio Mathematica},
    VOLUME = {126},
      YEAR = {2001},
    NUMBER = {1},
     PAGES = {1--23},
      ISSN = {0010-437X,1570-5846},
   MRCLASS = {14K10 (14J30 14M20)},
  MRNUMBER = {1827859},
MRREVIEWER = {Kieran\ G.\ O'Grady},
       DOI = {10.1023/A:1017518027822},
       URL = {https://doi-org.kyoto-u.idm.oclc.org/10.1023/A:1017518027822},
}

@article {Gus83,
    AUTHOR = {Gushel, N. P.},
     TITLE = {Fano varieties of genus  8},
   JOURNAL = {Russian Math. Surveys},
  FJOURNAL = {Russian Math. Surveys},
    VOLUME = {38},
      YEAR = {1983},
     PAGES = {192--193},
}

@article {Gus92,
    AUTHOR = {Gushel\cprime, N. P.},
     TITLE = {Fano {$3$}-folds of genus {$8$}},
   JOURNAL = {Algebra i Analiz},
  FJOURNAL = {Rossi\u iskaya Akademiya Nauk. Algebra i Analiz},
    VOLUME = {4},
      YEAR = {1992},
    NUMBER = {1},
     PAGES = {120--134},
      ISSN = {0234-0852},
   MRCLASS = {14J45 (14J30 14J60)},
  MRNUMBER = {1171957},
MRREVIEWER = {Jaros\l aw\ A.\ Wi\'sniewski},
}

@article {Har98,
    AUTHOR = {Hara, Nobuo},
     TITLE = {A characterization of rational singularities in terms of
              injectivity of {F}robenius maps},
   JOURNAL = {Amer. J. Math.},
  FJOURNAL = {American Journal of Mathematics},
    VOLUME = {120},
      YEAR = {1998},
    NUMBER = {5},
     PAGES = {981--996},
      ISSN = {0002-9327,1080-6377},
   MRCLASS = {13A99 (14B05 14E05)},
  MRNUMBER = {1646049},
MRREVIEWER = {Karen\ E.\ Smith},
       URL =
              {http://muse.jhu.edu.kyoto-u.idm.oclc.org/journals/american_journal_of_mathematics/v120/120.5hara.pdf},
}

@book {Har77,
    AUTHOR = {Hartshorne, Robin},
     TITLE = {Algebraic geometry},
    SERIES = {Graduate Texts in Mathematics, No. 52},
 PUBLISHER = {Springer-Verlag, New York-Heidelberg},
      YEAR = {1977},
     PAGES = {xvi+496},
      ISBN = {0-387-90244-9},
   MRCLASS = {14-01},
  MRNUMBER = {0463157},
MRREVIEWER = {Robert Speiser},
}

@article {Isk77,
    AUTHOR = {Iskovskih, V. A.},
     TITLE = {Fano threefolds. {I}},
   JOURNAL = {Izv. Akad. Nauk SSSR Ser. Mat.},
  FJOURNAL = {Izvestiya Akademii Nauk SSSR. Seriya Matematicheskaya},
    VOLUME = {41},
      YEAR = {1977},
    NUMBER = {3},
     PAGES = {516--562, 717},
      ISSN = {0373-2436},
   MRCLASS = {14J10 (14M20 14N05)},
  MRNUMBER = {463151},
MRREVIEWER = {Miles Reid},
}

@article {Isk78,
    AUTHOR = {Iskovskih, V. A.},
     TITLE = {Fano threefolds. {II}},
   JOURNAL = {Izv. Akad. Nauk SSSR Ser. Mat.},
  FJOURNAL = {Izvestiya Akademii Nauk SSSR. Seriya Matematicheskaya},
    VOLUME = {42},
      YEAR = {1978},
    NUMBER = {3},
     PAGES = {506--549},
      ISSN = {0373-2436},
   MRCLASS = {14J10 (14J30 14M20 14N05)},
  MRNUMBER = {503430},
MRREVIEWER = {Miles Reid},
}

@article{Isk80,
 author = {Iskovskikh, V. A.},
 title = {Birational automorphisms of three-dimensional algebraic varieties},
 fjournal = {Journal of Soviet Mathematics},
 journal = {J. Sov. Math.},
 issn = {0090-4104},
 volume = {13},
 pages = {815--868},
 year = {1980},
 language = {English},
 doi = {10.1007/BF01084564},
 zbMATH = {3666954},
 Zbl = {0428.14017}
}

@incollection {IP99,
    AUTHOR = {Iskovskikh, V. A. and Prokhorov, Yu. G.},
     TITLE = {Fano varieties},
 BOOKTITLE = {Algebraic geometry, {V}},
    SERIES = {Encyclopaedia Math. Sci.},
    VOLUME = {47},
     PAGES = {1--247},
 PUBLISHER = {Springer, Berlin},
      YEAR = {1999},
   MRCLASS = {14J45 (14E07 14F22 14K30)},
  MRNUMBER = {1668579},
MRREVIEWER = {Takao Fujita},
}

@article {KTTWYY1,
    AUTHOR = {Kawakami, Tatsuro and Takamatsu, Teppei and Tanaka, Hiromu and
              Witaszek, Jakub and Yobuko, Fuetaro and Yoshikawa, Shou},
     TITLE = {Quasi-{$F$}-splittings in birational geometry},
   JOURNAL = {Ann. Sci. \'Ec. Norm. Sup\'er. (4)},
  FJOURNAL = {Annales Scientifiques de l'\'Ecole Normale Sup\'erieure.
              Quatri\`eme S\'erie},
    VOLUME = {58},
      YEAR = {2025},
    NUMBER = {3},
     PAGES = {665--748},
      ISSN = {0012-9593,1873-2151},
   MRCLASS = {14E05 (13A35 14G17)},
  MRNUMBER = {4962159},
}

@article{KTTWYY3,
  title={Quasi-{$F$}-splittings in birational geometry {III}},
  author={Kawakami, Tatsuro and Takamatsu, Teppei and Tanaka, Hiromu and
              Witaszek, Jakub and Yobuko, Fuetaro and Yoshikawa, Shou},
  journal={arXiv:2408.01921},
  year={2024}
}

@article{KTLift1,
  author =        {Kawakami, Tatsuro and Tanaka, Hiromu},
  journal =       {preprint available at arXiv:2503.10236v1},
  title =         {Liftability and vanishing theorems for {F}ano threefolds in positive characteristic {I}},
  year =          {2025},
}

@article{KTLift2,
  author =        {Kawakami, Tatsuro and Tanaka, Hiromu},
  journal =       {preprint available at arXiv:2404.04764v2},
  title =         {Liftability and vanishing theorems for {F}ano threefolds in positive characteristic {II}},
  year =          {2025},
}

@article {Kol91,
    AUTHOR = {Koll\'{a}r, J\'{a}nos},
     TITLE = {Extremal rays on smooth threefolds},
   JOURNAL = {Ann. Sci. \'{E}cole Norm. Sup. (4)},
  FJOURNAL = {Annales Scientifiques de l'\'{E}cole Normale Sup\'{e}rieure. Quatri\`eme
              S\'{e}rie},
    VOLUME = {24},
      YEAR = {1991},
    NUMBER = {3},
     PAGES = {339--361},
      ISSN = {0012-9593},
   MRCLASS = {14J30 (14E30 14J10)},
  MRNUMBER = {1100994},
MRREVIEWER = {Eckart Viehweg},
       URL = {http://www.numdam.org/item?id=ASENS_1991_4_24_3_339_0},
}

@book {KM98,
    AUTHOR = {Koll\'{a}r, J\'{a}nos and Mori, Shigefumi},
     TITLE = {Birational geometry of algebraic varieties},
    SERIES = {Cambridge Tracts in Mathematics},
    VOLUME = {134},
      NOTE = {With the collaboration of C. H. Clemens and A. Corti,
              Translated from the 1998 Japanese original},
 PUBLISHER = {Cambridge University Press, Cambridge},
      YEAR = {1998},
     PAGES = {viii+254},
      ISBN = {0-521-63277-3},
   MRCLASS = {14E30},
  MRNUMBER = {1658959},
MRREVIEWER = {Mark Gross},
       DOI = {10.1017/CBO9780511662560},
       URL = {https://doi-org.utokyo.idm.oclc.org/10.1017/CBO9780511662560},
}

@article {Kuz04,
    AUTHOR = {Kuznetsov, A. G.},
     TITLE = {Derived category of a cubic threefold and the variety
              {$V_{14}$}},
   JOURNAL = {Tr. Mat. Inst. Steklova},
  FJOURNAL = {Trudy Matematicheskogo Instituta Imeni V. A. Steklova},
    VOLUME = {246},
      YEAR = {2004},
     PAGES = {183--207},
      ISSN = {0371-9685,3034-1809},
   MRCLASS = {14J30 (14J45 14N05)},
  MRNUMBER = {2101293},
MRREVIEWER = {Ivan\ Chel\cprime tsov},
}

@article {MR85,
    AUTHOR = {Mehta, V. B. and Ramanathan, A.},
     TITLE = {Frobenius splitting and cohomology vanishing for {S}chubert
              varieties},
   JOURNAL = {Ann. of Math. (2)},
  FJOURNAL = {Annals of Mathematics. Second Series},
    VOLUME = {122},
      YEAR = {1985},
    NUMBER = {1},
     PAGES = {27--40},
      ISSN = {0003-486X,1939-8980},
   MRCLASS = {14M15 (20G10)},
  MRNUMBER = {799251},
MRREVIEWER = {H.\ H.\ Andersen},
       DOI = {10.2307/1971368},
       URL = {https://doi-org.kyoto-u.idm.oclc.org/10.2307/1971368},
}

@article {MM81,
    AUTHOR = {Mori, Shigefumi and Mukai, Shigeru},
     TITLE = {Classification of {F}ano {$3$}-folds with {$B_{2}\geq 2$}},
   JOURNAL = {Manuscripta Math.},
  FJOURNAL = {Manuscripta Mathematica},
    VOLUME = {36},
      YEAR = {1981/82},
    NUMBER = {2},
     PAGES = {147--162},
      ISSN = {0025-2611},
   MRCLASS = {14J30 (14J10)},
  MRNUMBER = {641971},
MRREVIEWER = {Mary Schaps},
       DOI = {10.1007/BF01170131},
       URL = {https://doi.org/10.1007/BF01170131},
}

@incollection {MM83,
    AUTHOR = {Mori, Shigefumi and Mukai, Shigeru},
     TITLE = {On {F}ano {$3$}-folds with {$B_{2}\geq 2$}},
 BOOKTITLE = {Algebraic varieties and analytic varieties ({T}okyo, 1981)},
    SERIES = {Adv. Stud. Pure Math.},
    VOLUME = {1},
     PAGES = {101--129},
 PUBLISHER = {North-Holland, Amsterdam},
      YEAR = {1983},
   MRCLASS = {14J30},
  MRNUMBER = {715648},
MRREVIEWER = {I. Dolgachev},
       DOI = {10.2969/aspm/00110101},
       URL = {https://doi.org/10.2969/aspm/00110101},
}

@article {MM03,
    AUTHOR = {Mori, Shigefumi and Mukai, Shigeru},
     TITLE = {Erratum: ``{C}lassification of {F}ano 3-folds with {$B_2\geq
              2$}'' [{M}anuscripta {M}ath. {\bf 36} (1981/82), no. 2,
              147--162; {MR}0641971 (83f:14032)]},
   JOURNAL = {Manuscripta Math.},
  FJOURNAL = {Manuscripta Mathematica},
    VOLUME = {110},
      YEAR = {2003},
    NUMBER = {3},
     PAGES = {407},
      ISSN = {0025-2611},
   MRCLASS = {14J45 (14E30 14J30)},
  MRNUMBER = {1969009},
       DOI = {10.1007/s00229-002-0336-2},
       URL = {https://doi.org/10.1007/s00229-002-0336-2},
}

@article {Muk89,
    AUTHOR = {Mukai, Shigeru},
     TITLE = {Biregular classification of {F}ano {$3$}-folds and {F}ano
              manifolds of coindex {$3$}},
   JOURNAL = {Proc. Nat. Acad. Sci. U.S.A.},
  FJOURNAL = {Proceedings of the National Academy of Sciences of the United
              States of America},
    VOLUME = {86},
      YEAR = {1989},
    NUMBER = {9},
     PAGES = {3000--3002},
      ISSN = {0027-8424},
   MRCLASS = {14J30 (14J35 14J40)},
  MRNUMBER = {995400},
MRREVIEWER = {A. S. Tikhomirov},
       DOI = {10.1073/pnas.86.9.3000},
       URL = {https://doi-org.utokyo.idm.oclc.org/10.1073/pnas.86.9.3000},
}

@incollection {Muk93,
    AUTHOR = {Mukai, Shigeru},
     TITLE = {Curves and {G}rassmannians},
 BOOKTITLE = {Algebraic geometry and related topics ({I}nchon, 1992)},
    SERIES = {Conf. Proc. Lecture Notes Algebraic Geom., I},
     PAGES = {19--40},
 PUBLISHER = {Int. Press, Cambridge, MA},
      YEAR = {1993},
   MRCLASS = {14H45 (14M15)},
  MRNUMBER = {1285374},
MRREVIEWER = {Raquel Mallavibarrena},
}

@incollection {Muk01,
    AUTHOR = {Mukai, Shigeru},
     TITLE = {Non-abelian {B}rill-{N}oether theory and {F}ano 3-folds
              [translation of {S}\=ugaku {\bf 49} (1997), no. 1, 1--24;
              {MR}1478148 (99b:14012)]},
      NOTE = {Sugaku Expositions},
   JOURNAL = {Sugaku Expositions},
  FJOURNAL = {Sugaku Expositions},
    VOLUME = {14},
      YEAR = {2001},
    NUMBER = {2},
     PAGES = {125--153},
      ISSN = {0898-9583,2473-585X},
   MRCLASS = {14D20 (14H60 14J45)},
  MRNUMBER = {1857462},
}

@article {Mur73,
    AUTHOR = {Murre, J. P.},
     TITLE = {Reduction of the proof of the non-rationality of a
              non-singular cubic threefold to a result of {M}umford},
   JOURNAL = {Compositio Math.},
  FJOURNAL = {Compositio Mathematica},
    VOLUME = {27},
      YEAR = {1973},
     PAGES = {63--82},
      ISSN = {0010-437X,1570-5846},
   MRCLASS = {14C10 (14C15 14J05 14K30)},
  MRNUMBER = {352089},
MRREVIEWER = {D.\ Lieberman},
}

@article {Put82,
    AUTHOR = {Puts, Pierre J.},
     TITLE = {On some {F}ano-threefolds that are sections of
              {G}rassmannians},
   JOURNAL = {Nederl. Akad. Wetensch. Indag. Math.},
  FJOURNAL = {Koninklijke Nederlandse Akademie van Wetenschappen.
              Indagationes Mathematicae},
    VOLUME = {44},
      YEAR = {1982},
    NUMBER = {1},
     PAGES = {77--90},
      ISSN = {0019-3577},
   MRCLASS = {14J30 (14K30)},
  MRNUMBER = {653456},
MRREVIEWER = {Miles\ Reid},
}

@article {SS10,
    AUTHOR = {Schwede, Karl and Smith, Karen E.},
     TITLE = {Globally {$F$}-regular and log {F}ano varieties},
   JOURNAL = {Adv. Math.},
  FJOURNAL = {Advances in Mathematics},
    VOLUME = {224},
      YEAR = {2010},
    NUMBER = {3},
     PAGES = {863--894},
      ISSN = {0001-8708},
   MRCLASS = {14J45 (13A35 14B05)},
  MRNUMBER = {2628797},
MRREVIEWER = {Yukihide Takayama},
       DOI = {10.1016/j.aim.2009.12.020},
       URL = {https://doi-org.utokyo.idm.oclc.org/10.1016/j.aim.2009.12.020},
}

@article {Sho79a,
    AUTHOR = {\v{S}okurov, V. V.},
     TITLE = {The existence of a line on {F}ano varieties},
   JOURNAL = {Izv. Akad. Nauk SSSR Ser. Mat.},
  FJOURNAL = {Izvestiya Akademii Nauk SSSR. Seriya Matematicheskaya},
    VOLUME = {43},
      YEAR = {1979},
    NUMBER = {4},
     PAGES = {922--964, 968},
      ISSN = {0373-2436},
   MRCLASS = {14J30 (14M20)},
  MRNUMBER = {548510},
MRREVIEWER = {Miles Reid},
}

@article {Sho79b,
    AUTHOR = {\v{S}okurov, V. V.},
     TITLE = {Smoothness of a general anticanonical divisor on a {F}ano
              variety},
   JOURNAL = {Izv. Akad. Nauk SSSR Ser. Mat.},
  FJOURNAL = {Izvestiya Akademii Nauk SSSR. Seriya Matematicheskaya},
    VOLUME = {43},
      YEAR = {1979},
    NUMBER = {2},
     PAGES = {430--441},
      ISSN = {0373-2436},
   MRCLASS = {14J30},
  MRNUMBER = {534602},
MRREVIEWER = {Werner Kleinert},
}

@incollection {Smi00,
    AUTHOR = {Smith, Karen E.},
     TITLE = {Globally {F}-regular varieties: applications to vanishing
              theorems for quotients of {F}ano varieties},
      NOTE = {Dedicated to William Fulton on the occasion of his 60th
              birthday},
   JOURNAL = {Michigan Math. J.},
  FJOURNAL = {Michigan Mathematical Journal},
    VOLUME = {48},
      YEAR = {2000},
     PAGES = {553--572},
      ISSN = {0026-2285,1945-2365},
   MRCLASS = {13A35 (14F17 14J45)},
  MRNUMBER = {1786505},
MRREVIEWER = {Scott\ R.\ Nollet},
       DOI = {10.1307/mmj/1030132733},
       URL = {https://doi-org.kyoto-u.idm.oclc.org/10.1307/mmj/1030132733},
}

@article {Tak89,
    AUTHOR = {Takeuchi, Kiyohiko},
     TITLE = {Some birational maps of {F}ano {$3$}-folds},
   JOURNAL = {Compositio Math.},
  FJOURNAL = {Compositio Mathematica},
    VOLUME = {71},
      YEAR = {1989},
    NUMBER = {3},
     PAGES = {265--283},
      ISSN = {0010-437X},
   MRCLASS = {14J30 (14E05)},
  MRNUMBER = {1022045},
MRREVIEWER = {Peter Nielsen},
       URL = {http://www.numdam.org/item?id=CM_1989__71_3_265_0},
}

@article {Tan15,
    AUTHOR = {Tanaka, Hiromu},
     TITLE = {The {X}-method for klt surfaces in positive characteristic},
   JOURNAL = {J. Algebraic Geom.},
  FJOURNAL = {Journal of Algebraic Geometry},
    VOLUME = {24},
      YEAR = {2015},
    NUMBER = {4},
     PAGES = {605--628},
      ISSN = {1056-3911},
   MRCLASS = {14G17 (14Jxx)},
  MRNUMBER = {3383599},
MRREVIEWER = {Christian Liedtke},
       DOI = {10.1090/S1056-3911-2014-00627-5},
       URL = {https://doi-org.utokyo.idm.oclc.org/10.1090/S1056-3911-2014-00627-5},
}

@article {Tan18b,
    AUTHOR = {Tanaka, Hiromu},
     TITLE = {Behavior of canonical divisors under purely inseparable base
              changes},
   JOURNAL = {J. Reine Angew. Math.},
  FJOURNAL = {Journal f\"{u}r die Reine und Angewandte Mathematik. [Crelle's
              Journal]},
    VOLUME = {744},
      YEAR = {2018},
     PAGES = {237--264},
      ISSN = {0075-4102},
   MRCLASS = {14E30},
  MRNUMBER = {3871445},
MRREVIEWER = {Sung Rak Choi},
       DOI = {10.1515/crelle-2015-0111},
       URL = {https://doi.org/10.1515/crelle-2015-0111},
}

@article {Tana,
    AUTHOR = {Tanaka, Hiromu},
     TITLE = {Bertini theorems admitting base changes},
   JOURNAL = {J. Algebra},
  FJOURNAL = {Journal of Algebra},
    VOLUME = {644},
      YEAR = {2024},
     PAGES = {64--125},
      ISSN = {0021-8693,1090-266X},
   MRCLASS = {14G17 (14D06)},
  MRNUMBER = {4695615},
MRREVIEWER = {Lei\ Zhang},
       DOI = {10.1016/j.jalgebra.2023.12.038},
       URL = {https://doi-org.kyoto-u.idm.oclc.org/10.1016/j.jalgebra.2023.12.038},
}

@article {Tan-elliptic,
    AUTHOR = {Tanaka, Hiromu},
     TITLE = {Elliptic singularities and threefold flops in positive
              characteristic},
   JOURNAL = {Proc. Edinb. Math. Soc. (2)},
  FJOURNAL = {Proceedings of the Edinburgh Mathematical Society. Series II},
    VOLUME = {68},
      YEAR = {2025},
    NUMBER = {4},
     PAGES = {1188--1244},
      ISSN = {0013-0915,1464-3839},
   MRCLASS = {14J17 (14E30)},
  MRNUMBER = {4973442},
       DOI = {10.1017/S0013091525000185},
       URL = {https://doi-org.kyoto-u.idm.oclc.org/10.1017/S0013091525000185},
}

@article{FanoI,
  author =        {Tanaka, Hiromu},
  journal =       {arXiv:2308.08121},
  title =         {Fano threefolds in positive characteristic {I}},
  year =          {2023},
}

@article{FanoII,
  author =        {Tanaka, Hiromu},
  journal =       {arXiv:2308.08122},
  title =         {Fano threefolds in positive characteristic {II}},
  year =          {2023},
}

@article{FanoIII,
  author =        {Asai, Masaya and Tanaka, Hiromu},
  journal =       {arXiv:2308.08124},
  title =         {Fano threefolds in positive characteristic {III}},
  year =          {2023},
}

@article{FanoIV,
  author =        {Tanaka, Hiromu},
  journal =       {arXiv:2308.08127},
  title =         {Fano threefolds in positive characteristic {IV}},
  year =          {2023},
}

\end{document}